\documentclass[11pt,reqno, letterpaper]{amsart}
\usepackage[
    top=1in, bottom=1in, inner=1in, outer=1in,
    marginparwidth=2cm %necessary to avoid breaking fixme notes.
    ]{geometry}

\usepackage{amssymb,latexsym, amsmath, amsxtra, mathrsfs, bm}
\usepackage[dvips]{graphics}
\usepackage[T1]{fontenc}

\usepackage[all]{xy}
\usepackage{xcolor}
\usepackage{subcaption}
\usepackage{verbatim}
\usepackage[abs]{overpic}
\usepackage{hyperref}
\usepackage{mathtools}
\usepackage{enumitem}

\usepackage{thmtools}
\usepackage{thm-restate}

\DeclareMathAlphabet{\mathbbold}{U}{bbold}{m}{n}

\allowdisplaybreaks[3]

\makeatletter
\def\th@plain{%
	\thm@notefont{}% same as heading font
	\itshape % body font
}
\def\th@definition{%
	\thm@notefont{}% same as heading font
	\normalfont % body font
}
\makeatother
\theoremstyle{plain}
        \newtheorem{theorem}{Theorem}[section]
        \newtheorem*{theorem*}{Theorem}
                \newtheorem{thmx}{Theorem}

        \newtheorem{lemma}[theorem]{Lemma}
        \newtheorem{prop}[theorem]{Proposition}
        \newtheorem{cor}[theorem]{Corollary}

\theoremstyle{definition}
        \newtheorem{definition}[theorem]{Definition}
        \newtheorem{rem}[theorem]{Remark}

\theoremstyle{remark}

\numberwithin{equation}{section}
\numberwithin{theorem}{section}
\numberwithin{table}{section}
\numberwithin{figure}{section}

\renewcommand{\le}{\leqslant}
\renewcommand{\leq}{\leqslant}
\renewcommand{\ge}{\geqslant}
\renewcommand{\geq}{\geqslant}

\newcommand{\diam}  {\operatorname{diam}}

\newcommand{\id} {\operatorname{id}}

\newcommand{\card} {\operatorname{card}}
\newcommand{\supp}{\operatorname{supp}}
\newcommand{\Per}{\mathrm{Per}}
\newcommand{\ve}{\varepsilon}
\newcommand{\R}{\mathbb{R}}
    
\newcommand{\C}{\mathbb{C}}

\newcommand{\N}{\mathbb{N}}      
\newcommand{\Z}{\mathbb{Z}}      
\providecommand{\abs}[1]{\lvert#1\rvert}
\providecommand{\Absbig}[1]{\bigl\lvert#1\bigr\rvert}

\providecommand{\norm}[1]{\|#1\|}
\providecommand{\Normbig}[1]{\bigl\|#1\bigr\|}

\renewcommand{\:}{\colon}

\newcommand{\LIP}{\operatorname{LIP}}

\newcommand{\CCC}{C}

\newcommand{\MMM}{\mathcal{M}}

\newcommand{\Holder}[1] {\CCC^{0,#1}}

\newcommand{\Hseminorm}[2] {\abs{#2}_{#1}}

\newcommand{\Hseminormbig}[2] {\Absbig{#2}_{#1}}

\newcommand{\Hnorm}[2] {\norm{#2}_{#1}}

\newcommand{\Hnormbig}[2] {\Normbig{#2}_{#1}}

\renewcommand{\=}{\coloneqq}

\newcommand{\Mmax}{\MMM_{\operatorname{max}}}

\newcommand{\Lock}{\operatorname{Lock}}

\newcommand{\mpe}{Q}

\newcommand{\cA}{\mathcal{A}}
\newcommand{\cB}{\mathcal{B}}

\newcommand{\cE}{\mathcal{E}}
\newcommand{\cF}{\mathcal{F}}

\newcommand{\cH}{\mathcal{H}}

\newcommand{\cL}{\mathcal{L}}
\newcommand{\cM}{\mathcal{M}}

\newcommand{\cO}{\mathcal{O}}
\newcommand{\cP}{\mathcal{P}}
\newcommand{\cQ}{\mathcal{Q}}
\newcommand{\cR}{\mathcal{R}}

\newcommand{\cV}{\mathcal{V}}
\newcommand{\cW}{\mathcal{W}}

\newcommand{\tF}{\widetilde{F}}

\newcommand{\tpsi}{\widetilde{\psi}}

\begin{document}
	\title[Joint typical periodic optimization: systems with stable hyperbolicity]{
    Joint typical periodic optimization:\\ systems with stable hyperbolicity}
    
	\author{Zelai~Hao \and Yinying~Huang \and Oliver~Jenkinson \and Zhiqiang~Li}
	%\thanks{The authors were partially supported by NSFC Nos.~12471083, 12101017, 12090010, and 12090015.}
	\address{Zelai~Hao, School of Mathematical Sciences, Peking University, Beijing 100871, China}
	\email{2100010625@stu.pku.edu.cn}
	\address{Yinying~Huang, School of Mathematical Sciences, Peking University, Beijing 100871, China}
	\email{miaoyan@stu.pku.edu.cn}
	\address{Oliver~Jenkinson, School of Mathematical Sciences, Queen Mary, University of London, Mile End Road, London E1 4NS, United Kingdom}
	\email{o.jenkinson@qmul.ac.uk}
	\address{Zhiqiang~Li, School of Mathematical Sciences \& Beijing International Center for Mathematical Research, Peking University, Beijing 100871, China}
	\email{zli@math.pku.edu.cn}

	

	\thanks{Zelai~Hao, Yinying~Huang, and Zhiqiang~Li were partially supported by Beijing Natural Science Foundation (JQ25001) and National Natural Science Foundation of China (12471083).
	}

	\begin{abstract}
		The framework of joint typical periodic optimization, in which both the
dynamical system and the potential function are allowed to vary
simultaneously, was introduced in~\cite{HHJL25},
in a direction motivated by the work of Yang, Hunt \& Ott~\cite{YHO00}.
For certain classes of hyperbolic systems, it was shown there that
optimizing periodic orbits persist under simultaneous perturbation,
yielding joint locking sets that contain open dense subsets of the
relevant product spaces.
In the present article we broaden the scope of this theory,
by developing an axiomatic joint perturbation framework that
accommodates a wider class of stably hyperbolic systems, and by
establishing new joint typical periodic optimization results for
several natural and important families: Axiom~A diffeomorphisms with
the no-cycle property, hyperbolic rational maps on the Riemann sphere,
real quadratic polynomials, and $C^r$ maps in one dimension.
	\end{abstract}
	
	\maketitle

    \setcounter{tocdepth}{1}
	\tableofcontents

\section{Introduction}\label{introduction}

\subsection{Background}

Ergodic optimization is the study of those probability measures, invariant under some dynamical system, that optimize the space average of a given potential function.
For $(X,d)$ a compact metric space, and $f \: X \to X$ a 
continuous self-map,  let $\cM(X,f)$ denote the space of \emph{$f$-invariant Borel probability measures}.
For  a continuous function $\phi \: X \to \R$,
the corresponding \emph{maximum ergodic average}  is defined as
\begin{equation}\label{eq:mea}
	\mpe(X,f,\phi) \= \sup \biggl\{ \int\! \phi \, \mathrm{d}\mu : \mu\in \cM(X,f) \biggr\}.
\end{equation}
A measure $\mu\in \cM(X,f)$ that 
attains the maximum ergodic average (\ref{eq:mea})
is called \emph{$(f,\phi)$-maximizing}, 
and  the set of $(f,\phi)$-maximizing measures is denoted by $\Mmax(X,f,\phi)$, in other words
\begin{equation*}
	\Mmax(X,f,\phi) \= \biggl\{ \mu\in \cM(X,f) : \int\! \phi \, \mathrm{d}\mu = \mpe(X,f,\phi) \biggr\}.
\end{equation*}
For simplicity we write $\mpe(f,\phi)$ and $\Mmax(f,\phi)$ whenever there is no possibility of confusion regarding the space $X$.

We say that $(f,\phi)$ has the \emph{periodic optimization property} (or \emph{PO property}) if there exists a unique $(f,\phi)$-maximizing measure, and this measure is supported on a single $f$-periodic orbit. 
If $\cP$ is a Banach space consisting of continuous real-valued functions defined on $X$, 
we say that $(f,\cP)$ has the \emph{typical periodic optimization property} (or \emph{TPO property}) if there exists an open dense subset $V(f)$ of $\cP$ such that for every $\phi\in V(f)$, the pair $(f,\phi)$ has 
the PO property. 
For $\phi\in \cP$, we say that $(f,\phi)$ has the \emph{locking property} if $(f,\phi)$ has the PO property and $\Mmax(f,\psi)=\Mmax(f,\phi)$ for all $\psi\in\cP$ sufficiently close to $\phi$. 
We write
\begin{equation*}
	\Lock(f,\cP) \= \{\phi\in \cP : (f,\phi) \text{ has the locking property}\},
\end{equation*} 
and call $\Lock(f,\cP)$ the \emph{locking set} of $(f,\cP)$.

If $\cF$ is a topological space consisting of 
self-maps of $X$, and $\cF \times \cP$ 
is equipped with the product topology, we say that $\cF\times \cP$ has the 
\emph{joint typical periodic optimization property}
(or \emph{Joint TPO property}) 
if there exists an open dense subset $U$ of $\cF\times \cP$ such that every pair $(f,\phi)\in U$ has 
the PO property. We write
\begin{equation*}
	\Lock(\cF, \cP) \= \{(f,\phi)\in \cF\times \cP: (f,\phi) \text{ has the locking property}\}
\end{equation*}
and refer to $\Lock(\cF,\cP)$ as the \emph{joint locking set} of $(\cF,\cP)$.

The field of ergodic optimization was significantly influenced by a 1999 conjecture of Yuan \& Hunt \cite{YH99}, that 
$(f,\cP)$ has the TPO property
whenever $f$
is uniformly hyperbolic (an expanding map or an Axiom A diffeomorphism), 
and $\cP$ is the space of Lipschitz functions.
After important work towards this conjecture
(see \cite{Bou01, Bou08, BQ07, CLT01, Mo08, QS12}), it was eventually 
resolved by Contreras \cite{Co16} in the case of expanding maps,
and by Huang, Lian, Ma, Xu \& Zhang \cite{HLMXZ25}
for Axiom A diffeomorphisms.
Going beyond uniform hyperbolicity, Li \& Zhang \cite{LZ25} established TPO in the complex dynamics setting of expanding Thurston maps. We refer the reader to surveys \cite{Boc18, Je06, Je19} for more detailed discussion on ergodic optimization.

In a direction suggested by the work of Yang, Hunt \& Ott~\cite{YHO00}, and by perspectives highlighted in the survey~\cite{Je19},
the notion of joint typical periodic optimization
was introduced in \cite{HHJL25}, 
where, for $\cP$ any Banach space of $\alpha$-H\"older functions,
Joint TPO was established for $\cF$ taken to be each of 
(i) the class of open
distance-expanding Lipschitz maps
on a locally connected space, (ii)
the class of Anosov diffeomorphisms on a compact smooth manifold equipped with a Riemannian metric, and (iii)
the class of beta-transformations on the unit interval.

The present article builds on the joint perturbation framework
of~\cite{HHJL25}, extending the scope of Joint TPO to various
families $\cF$ of maps enjoying some form of hyperbolicity.
A central theme is that Joint TPO can be established for maps
satisfying an appropriately abstract form of \emph{stable hyperbolicity},
via a broad axiomatic joint perturbation framework, allowing us
to deduce Joint TPO for several notable specific classes of hyperbolic
dynamical systems.

This axiomatic strategy consists of the following four steps:

\smallskip

(1) Formulate abstract hyperbolic assumptions ensuring some structural stability of invariant sets, uniform hyperbolicity, and a Ma\~n\'e cohomology lemma with uniform seminorm control.

\smallskip

(2) Prove a joint perturbation theorem showing that if a periodic orbit is maximizing for $(f, \phi)$, then nearby pairs 
$(g,\psi)$ admit nearby periodic maximizing orbits.

\smallskip

(3) Articulate a sufficient condition for a map-function pair to lie in the interior of the joint locking set.

\smallskip

(4) Verify that joint typical periodic optimization 
does indeed hold for the various specific classes of hyperbolic systems considered.

\smallskip

In particular, by proceeding along these lines, we succeed in extending the joint typical periodic optimization theory of \cite{HHJL25} to (i) the class of Axiom A diffeomorphisms with the no-cycle property, (ii) the class of hyperbolic rational maps  on the Riemann sphere,
(iii) the class of real
quadratic polynomials on the Riemann sphere, (iv)
certain classes of real one-dimensional smooth systems,
and (v)  the logistic family on $[0,1]$.

\subsection{Main results}\label{mainresultssubsection}

Throughout the article, we denote by 
$M$
a compact smooth manifold (without boundary), with Riemannian metric $\Hseminorm{}{\,\cdot\,}$ on $M$, and $d$ the induced distance function. 

For $\alpha\in(0,1]$, let $\Holder{\alpha}(M,\R)$ denote the Banach
space of real-valued $\alpha$-H\"{o}lder functions on $M$.
Let $C^1(M,\R)$ denote the Banach space of real-valued $C^1$ functions
on $M$, equipped with the norm
$\Hnorm{C^1,M}{\phi} \= \Hnorm{\infty,M}{\phi} + \Hnorm{\infty,M}{\mathrm{D}\phi}$,
where $\mathrm{D}\phi$ is the derivative of $\phi$.

Recall that a diffeomorphism $f\:M\to M$ is \emph{Axiom~A} if its
nonwandering set $\Omega(f)$ is uniformly hyperbolic (i.e., admitting a
$\mathrm{D}f$-invariant splitting of the tangent bundle over $\Omega(f)$ into
uniformly contracting stable and uniformly expanding unstable
subbundles) and has dense periodic points.
By Smale's spectral decomposition theorem \cite{Sm67},
$\Omega(f)$ decomposes uniquely into finitely many disjoint compact
transitive pieces, the \emph{basic sets}, and $f$ satisfies
the \emph{no-cycle condition} if there is no cyclic chain, via intersections of stable and unstable
manifolds, among the basic sets.
The set of Axiom~A diffeomorphisms with the no-cycle condition is precisely the 
set of diffeomorphisms satisfying $\Omega$-stability 
(see \cite{Pa70, Pa87,Sm70}).

For each $r\in \N$, let $\operatorname{Diff}^r(M)$  denote the space of $C^r$ diffeomorphisms on $M$, 
equipped
with the $C^r$ topology, and 
define $\cA^r(M)$ to be the subspace of $\operatorname{Diff}^r(M)$ consisting of
those Axiom~A diffeomorphisms with the no-cycle property.
We prove:

\begin{thmx}[\bf Joint TPO for Axiom A diffeomorphisms with no cycles]\label{JTPO AxiomA} 
	Let $r\in \N$, $\alpha\in (0,1]$, and $\cP$ denote either $C^1(M,\R)$ or $\Holder{\alpha}(M,\R)$.
    Then $\cA^r(M) \times \cP$ has the Joint TPO property. Moreover, the joint locking set $\Lock(\cA^r(M), \cP)$ is itself an open dense subset of $\cA^r(M)\times \cP$.
\end{thmx}

Note that for every $r\in \N$, the space $\cA^r(M)$ is second countable (see e.g.~\cite[p.~35]{Hi76}), so  combining Theorem~\ref{JTPO AxiomA} and \cite[Lemma~8.42]{Ke95} gives:

\begin{cor}\label{cor}
	For any $r\in \N$ and $\alpha\in (0,1]$, there exists a residual subset $H$ of $\Holder{\alpha}(M,\R)$ (resp.\ $C^1(M,\R)$) such that for each $\phi \in H$, there is an open dense subset $V(\phi)$ of $\cA^r(M)$ such that if $f\in  V(\phi)$ then $(f,\phi)$ has
    the periodic optimization property. 
\end{cor}

Beyond the class of Axiom~A diffeomorphisms with the no-cycle property, our methods can be applied to various other classes of hyperbolic dynamical systems. 
One such class consists of hyperbolic rational maps on the Riemann sphere $\widehat{\C}$. For $m\in \N$, let $\cR^m$ denote the collection of rational maps $f \: \widehat{\C} \to \widehat{\C}$ of degree $m$, equipped with the standard topology on parameter space
(which is embedded homeomorphically as the complement of an algebraic hypersurface in $(2m+1)$-dimensional complex projective space $\C \mathbb{P}^{2m+1}$,
cf.~e.g.~\cite[p.~47]{Be91}, \cite[Appendix B]{Mi93}, \cite{Se79}). For every $f\in \cR^m$, let $J(f)$ denote the \emph{Julia set} of $f$ (cf.~\cite[Definition~4.2]{Mi06}). If $m\ge 2$, a map $f\in \cR^m$ is said to be \emph{hyperbolic} if it is expanding on its Julia set, i.e., there exists a conformal metric
with induced norm $\Hseminorm{}{\,\cdot\,}$, and $\lambda>1$, such that
\begin{equation}\label{expanding}
	\Hseminorm{}{\mathrm{D} f (v)} \ge \lambda \Hseminorm{}{v}
\end{equation} 
for all $z\in J(f)$ and $v\in T_z \widehat{\C}$ (cf.~\cite[p.~205]{Mi06}). For $m\ge 2$, defining $\cH \cR^m$ to be the collection of all hyperbolic rational maps of degree $m$, we prove:

\begin{thmx}[\bf Joint TPO for hyperbolic rational maps]\label{hrational}
	Let $m\ge 2$, $\alpha\in (0,1]$, and $\cP$ denote either $\Holder{\alpha}\bigl(\widehat{\C},\R\bigr)$ or $C^1 \bigl(\widehat{\C}, \R\bigr)$.
    Then $\cH \cR^m \times \cP$ has the Joint TPO property. Moreover, the joint locking set $\Lock ( \cH\cR^m , \cP )$ is itself an open dense subset of $\cH\cR^m \times \cP$.
\end{thmx}
Note that $\cH \cR^m$ is open in $\cR^m$ (see~\cite[p.~205]{Mi06}). Conditional on the Density of Hyperbolicity Conjecture (see e.g.~\cite[p.~87]{Ly99}, \cite[p.~4]{Mc94}),  Theorem~\ref{hrational} ensures
that the Joint TPO property holds on all of $\cR^m \times \Holder{\alpha}\bigl( \widehat{\C}, \R \bigr)$, $\alpha\in (0,1]$, and on all of $\cR^m \times C^1\bigl( \widehat{\C}, \R \bigr)$. 

The question of extending Joint TPO within the settings of Theorems~\ref{JTPO AxiomA} and~\ref{hrational} 
is not pursued in the present paper, but we are able to develop a complementary direction, establishing Joint TPO for certain
natural full families of maps (with no explicit hyperbolicity hypothesis), in settings where
density of hyperbolicity holds unconditionally.

Graczyk \& \'Swiatek \cite{GS97} and
Lyubich \cite[p.~4]{Ly97} proved that hyperbolicity 
(in the sense of (\ref{expanding}))
is a dense property among real quadratic polynomials. If $\cQ$ denotes the set of real quadratic polynomials 
(i.e.,~those maps $f$ of the form $f(z)=z^2+c$ for $c\in[-2,1/4]$)
on $\widehat{\C}$, equipped with the standard topology on the parameter space $[-2,1/4]$, we prove:
\begin{thmx}[\bf Joint TPO for real quadratic polynomials]\label{realqua}
	Let $\alpha\in (0,1]$, and let $\cP$ denote either $\Holder{\alpha}\bigl( \widehat{\C}, \R \bigr)$ or $C^1 \bigl( \widehat{\C}, \R \bigr)$. Then $\cQ\times \cP$ has the Joint TPO property. 
\end{thmx}

Our methods can also be applied to various real
one-dimensional dynamical systems. If $M$ is
either a compact interval or the circle, and $r\in \N$, 
we let
$C^r(M,M)$ denote the space of $C^r$ maps on $M$ equipped with the $C^r$ topology, and prove:

\begin{thmx}[\bf Joint TPO for $C^r$ one-dimensional maps]\label{1dim}
	Let $M$ be a compact interval or the circle, and $r\in \N$. If $\alpha\in (0,1]$, and $\cP$ denotes either $\Holder{\alpha}(M,\R)$ or $C^1(M,\R)$, then $C^r(M,M) \times \cP$ has the Joint TPO property. 
\end{thmx}

Theorem~\ref{realqua} concerns real quadratic polynomials viewed as
holomorphic self-maps of the Riemann sphere $\widehat{\C}$, with potential
functions defined on all of $\widehat{\C}$.
The same family of maps can equivalently be parametrised as the
\emph{logistic family} $g_a : [0,1]\to [0,1]$,
$g_a(x) \= a\,x(1-x)$, $a \in [0,4]$, which are
topologically conjugate to the real quadratic polynomials
as one-dimensional dynamical systems.
However, the two settings lead to distinct ergodic optimization
problems: in Theorem~\ref{realqua} the potential functions are
H\"{o}lder (or $C^1$) on the Riemann sphere $\widehat{\C}$,
whereas in the Joint TPO theorem below they are defined only on the interval $[0,1]$.

\begin{thmx}[\bf Joint TPO for the logistic family]\label{realqua01}
Let $\alpha \in (0,1]$, $\cP$ denote either
$\Holder{\alpha}([0,1],\R)$ or $C^1([0,1],\R)$,
and $\cF \= \{g_a : a \in [0,4]\}$,
where $g_a \: [0,1] \to [0,1]$ is defined by
$g_a(x) \= a\,x(1-x)$.
We equip $\cF$ with the topology induced by the parameter space $[0,4]$.
Then $\cF \times \cP$ has the Joint TPO property.
\end{thmx}

\begin{rem}
	In this article, as in \cite{HHJL25}, the Joint TPO property is established
    for spaces of maps equipped with some topology stronger than the $C^0$ one.
    It turns out that results for the $C^0$ topology have a rather different flavour, and will be presented elsewhere.
\end{rem}

\subsection{Overview of the proof strategy}

The main results of Subsection~\ref{mainresultssubsection} are facilitated
by the abstract hyperbolic framework developed in Section~\ref{joint perturbation},
and the joint perturbation theorems established in that setting.
For a family $\cF$ of
Lipschitz self-maps of a compact metric space $(X,d)$, we define a notion of 
\emph{$\cF$-stable hyperbolicity} for a map $f\in\cF$, given as Definition~\ref{Fstablyhyperbolic}
(comprising the three conditions of \emph{Intertwining Stability} (IS), \emph{Robust Hyperbolic Estimates} (RHE), and \emph{Ma\~n\'e Lemma} (ML)); the proof
of Theorems~\ref{JTPO AxiomA}, \ref{hrational}, \ref{realqua}, \ref{1dim}, and~\ref{realqua01} will then involve verifying that the appropriate $\cF$-stable hyperbolicity holds.
The key result of Section~\ref{joint perturbation} is the joint perturbation theorem
(Theorem~\ref{jointperturbation}): if $f$ is $\cF$-stably hyperbolic
and $\cO$ is an $f$-periodic orbit with
$\Mmax(f,\phi)=\{\mu_{\cO}\}$, then for all $g$ near $f$ the orbit
$\cO_g\=h_g(\cO)$ is the unique $(g,\phi')$-maximizing orbit for every function
$\phi'$ in a certain explicit neighbourhood of $\phi$ (which shrinks at a
controlled rate as $g$ approaches $f$).
This joint perturbation theorem underpins our whole approach, and in particular leads to
the Interior Condition Theorem
(Theorem~\ref{openness}): if $f$ is
$\cF$-stably hyperbolic and $(f,\phi)\in\Lock(\cF,\cP)$ with $\phi$
nonconstant, then $(f,\phi)$ lies in the interior of
the joint locking set
$\Lock(\cF,\cP)$.

To prove Theorem~\ref{JTPO AxiomA}, 
the  $\cA^r(M)$-stable hyperbolicity
of every $f\in\cA^r(M)$ 
is established via
Lemmas~\ref{A(a)},~\ref{A(b)}, and~\ref{Mane lemma}. Condition (IS) in 
Definition~\ref{Fstablyhyperbolic} follows from $\Omega$-stability: since Axiom~A with
the no-cycle condition is equivalent to $\Omega$-stability
\cite{Pa70,Sm70}, the required intertwining maps $h_g$, $i_g$ exist
and converge to the identity.
Definition~\ref{Fstablyhyperbolic}~(RHE) requires a strengthening of the well-known
hyperbolicity of $\Omega(f)$ for a single map (cf.~\cite[Proposition~6.4.16]{KH95}):
the relevant constants must be chosen uniformly over
a neighbourhood of $f$, necessitating a careful
analysis of the adapted Riemannian metric. 
Condition (ML) in Definition~\ref{Fstablyhyperbolic} is established by adapting the 
analysis of \cite{STY24} (cf.~also the closely related \cite{Bou11}).
Individual TPO for each fixed Axiom~A diffeomorphism
(Proposition~\ref{TPO AxiomA}) follows from~\cite{HLMXZ25},
and Theorem~\ref{openness} then allows joint locking to be deduced, as well
as the fact that $\Lock(\cA^r(M),\cP)$ is itself open.

To prove Theorem~\ref{hrational}, the verification of Definition~\ref{Fstablyhyperbolic} for each hyperbolic rational map
proceeds via tools from complex dynamics, and a key
difference arises regarding condition (ML).
For condition (IS) in Definition~\ref{Fstablyhyperbolic}, the structural stability conjugacies under perturbation are 
provided by the theory of holomorphic motions \cite{Ly83, MSS83}, and
for condition (RHE) in Definition~\ref{Fstablyhyperbolic}, Lemma~\ref{disexpanding} shows that the expanding
constants of $f|_{J(f)}$ can be chosen uniformly over a neighbourhood
of $f$.
For condition (ML) in Definition~\ref{Fstablyhyperbolic}, unlike in the Axiom~A case where
a single Ma\~n\'e lemma argument applies
throughout, the nonwandering set of a hyperbolic rational map splits
into the expanding Julia set and finitely many attracting periodic
orbits: Lemma~\ref{A(c)hr} constructs the sub-action $u$
separately on each piece (via distance-expanding Ma\~n\'e lemma techniques on $J(f)$,
and an explicit orbit-average formula on the attracting orbits),
and then patches these together.

To prove Theorem~\ref{realqua}, the stable hyperbolicity properties
of Theorem~\ref{hrational} apply, and can be combined with the 
result of Graczyk--\'{S}wi\k{a}tek \cite{GS97} and Lyubich \cite{Ly97},
that hyperbolicity is an open dense property
in the real quadratic polynomial family.

To prove Theorem~\ref{1dim},
the most delicate and technically demanding result in this article,
we must show that the interior of the joint locking set is dense in $C^r(M,M)\times\cP$. For density,
we start by noting that the theorem of Kozlovski,
Shen \& van Strien~\cite{KSS07}
means that any $C^r$ map on $M$ can be approximated
by a hyperbolic map $f$ whose
boundary values lie in the interior of $M$.
We then construct (via Lemma~\ref{extension})
an extension $F$ of $f$ to a strictly larger interval $M_0$
that preserves the new endpoints $\partial M_0$, and that is itself
hyperbolic (Lemma~\ref{extensionprop}~(iii));
in particular, $F$ belongs to the space $C^r_0(M_0,M_0)$
of endpoint-preserving maps, to which the abstract framework applies.
Proposition~\ref{endpointsp} then verifies that $F$ is stably
hyperbolic within $C^r_0(M_0,M_0)$, by showing that a certain restriction
is distance-expanding with respect to an adapted metric.
Given any $\phi\in\cP$, we extend it to some $\Phi$ belonging to an appropriate space of potentials $\cP_0$
on $M_0$, whose values at $\partial M_0$ are prescribed to lie strictly
below $\mpe(f,\phi)$; Lemma~\ref{inTPO1dim} (Individual TPO for $F$)
then produces a nearby $\Psi\in\Lock(F,\cP_0)$, and
Theorem~\ref{openness} yields a locking neighbourhood of
$(F,\Psi)$ in $C^r_0(M_0,M_0)\times\cP_0$.
In Lemma~\ref{extensionprop} we establish that this locking phenomenon descends
to the original space: every $F$-invariant measure is supported on
$M\cup\partial M_0$, and since the values of $\Psi$ at $\partial M_0$
remain strictly below the maximum ergodic average, every
$(G,\Xi)$-maximizing measure, for $(G,\Xi)$ in the locking
neighbourhood, is automatically supported on $M$, and hence
coincides with the $(g,\xi)$-maximizing measure, where $g\=G|_M$
and $\xi\=\Xi|_M$.
This extension procedure, and particularly Lemma~\ref{exnormbound},
which provides norm bounds ensuring that the locking neighbourhood in
$C^r_0(M_0,M_0)\times\cP_0$ pulls back to one in
$C^r(M,M)\times\cP$, is a key novelty in this $C^r$ one-dimensional
case. Because the initial $C^r$ map and potential were arbitrary, showing that their approximations always admit such a locking neighbourhood establishes that the interior of the joint locking set is dense.

To prove Theorem~\ref{realqua01}, we note that every logistic map
satisfies $g_a(0) = g_a(1) = 0$, so $\cF \subseteq C^1_0([0,1],[0,1])$,
and the abstract framework of Proposition~\ref{endpointsp} applies directly.
Density of hyperbolicity in $\cF$ again follows from
Graczyk--\'Swi\k{a}tek~\cite{GS97} and Lyubich~\cite[p.~4]{Ly97},
and Individual TPO for each hyperbolic $g_a \in \cF$ follows from
Lemma~\ref{inTPO1dim}.
The proof then proceeds by the same argument as for Theorem~\ref{realqua}:
$\cF$-stable hyperbolicity is inherited from
Proposition~\ref{endpointsp} by restriction, and
Theorem~\ref{openness} yields the Joint TPO property.

\begin{rem}[\bf Relation to \cite{HHJL25}]
\label{rem:HHJL25}
The present article is a companion to~\cite{HHJL25},
which introduced the notion of Joint TPO
and established it in three settings:
open Lipschitz distance-expanding maps,
Anosov diffeomorphisms, and beta-transformations.
Here we clarify the relationship between
the two works, and their somewhat different routes through related material.

\vspace{2pt}

\noindent\emph{Open distance-expanding maps.}
Let $\cE(X)$ denote the space of open Lipschitz distance-expanding maps
on a compact locally connected metric space $X$,
as in~\cite{HHJL25}.
Since $\Omega(T)=X$ for every $T\in\cE(X)$, one can verify
that conditions (RHE) and (ML) of Definition~\ref{Fstablyhyperbolic}
hold for every $T\in\cE(X)$:
condition (RHE) holds with $K=1$, using the
distance-expanding property and the 
machinery of \cite{HHJL25} establishing
uniformity of expanding constants
over an appropriate neighbourhood, 
while condition (ML) is precisely the
Ma\~n\'e lemma of~\cite{HHJL25}, whose specific form is contained in \cite{LS26}.
Condition (IS), however, requires structural stability of $\cE(X)$,
and although 
in principle
provable in this generality (as noted in~\cite{HHJL25}),
it is not yet available in the literature for open distance-expanding maps
on arbitrary compact locally connected metric spaces,
indeed the method of~\cite{HHJL25} deliberately circumvents it,
establishing Joint TPO instead via a more elementary
Locally Connected Shadowing Lemma.
Note that, once structural stability is established,
the abstract framework of this article would yield only a \emph{qualitative}
Joint TPO statement,
whereas~\cite{HHJL25} establishes a stronger
\emph{effective} Joint TPO theorem,
with explicit quantitative bounds on the size of the
joint locking neighbourhood in terms of the expanding constants
and the H\"older seminorm of the potential function.
The two approaches are thus complementary:
the present article provides a broader axiomatic framework,
while~\cite{HHJL25} achieves sharper quantitative conclusions
via a more direct route.

\vspace{2pt}

\noindent\emph{Anosov diffeomorphisms.}
The approach of the present paper applies to
$C^r$ Axiom~A diffeomorphisms with the no-cycle property
(Theorem~\ref{JTPO AxiomA}), and in particular to Anosov
diffeomorphisms (which are structurally stable, hence have the no-cycle condition).
The paper~\cite{HHJL25} includes a self-contained proof of
Joint TPO for Anosov diffeomorphisms as well, 
presented there as a deliberately less detailed preview of the present work.

\vspace{2pt}

\noindent\emph{Beta-transformations.}
The beta-transformation results of~\cite{HHJL25} lie entirely outside
the scope of the present article, the two approaches having no
overlap in this case.
Beta-transformations $T_\beta\:[0,1]\to[0,1]$, defined as $T_\beta(x)=\beta x \pmod 1$,
are not continuous,
so do not belong to any space of Lipschitz self-maps of a compact
metric space in the sense required by Definition~\ref{Fstablyhyperbolic}.
Moreover, the sub-action for a beta-transformation in the Ma\~n\'e lemma of \cite{HHJL25}
need not be continuous, so condition (ML)
fails in the sense required here.
The analysis in~\cite{HHJL25} of Joint TPO for beta-transformations is therefore 
independent of the methods presented here, and highlights that the Joint TPO phenomenon is not confined to the class of systems handled by the present article.
\end{rem}

\subsection{Organisation of the article}

In Subsection~\ref{Notation} below, we collect some basic notation that is used throughout the article. In Section~\ref{joint perturbation}, we 
introduce an abstract framework accommodating a broad notion of hyperbolicity. Under these assumptions, we first establish a joint perturbation theorem (cf.~Theorem~\ref{jointperturbation}) for H\"older potential functions $\phi$, which generalises 
results of
\cite{HHJL25}. We then use Theorem~\ref{jointperturbation} to establish a $C^1$ analogue (cf.~Theorem~\ref{jointperturbation'}), which is used as an ingredient for establishing Joint TPO for $\cA^r(M) \times C^1(M,\R)$, and also give an important sufficient condition for a pair $(f,\phi)$ to belong to the interior of the joint locking set (cf.~Theorem~\ref{openness}). In Section~\ref{sec_Axiom_A} we prove Theorem~\ref{JTPO AxiomA},
in Section~\ref{sec.hr} we prove
Theorems~\ref{hrational}
and~\ref{realqua}, 
and in Section~\ref{sec.r1d}
we prove Theorems~\ref{1dim} and~\ref{realqua01};
each of these proofs, for a different family $\cF$ of hyperbolic maps, proceeds by showing that the abstract $\cF$-stable hyperbolicity of Section~\ref{joint perturbation} is satisfied by the relevant family $\cF$.
Some of the more technical proofs 
are deferred to Appendix~\ref{A}.

\subsection{Notation}\label{Notation}

Here we collect notation used throughout the article.

The set of positive integers is denoted by $\N$, and the set of nonnegative integers by $\N_0 \= \N \cup \{0\}$.

Let $(X,d)$ be a 
metric space.
For a subset $Y\subseteq X$, and $\ve>0$, the $\ve$-neighbourhood of $Y$ will be denoted by
$B_X(Y,\ve ) \= \{x\in X: d(x,Y)<\ve \}$,
and if $Y=\{y\}$ is a singleton, we write 
$B_X(y,\ve ) \= B_X(\{y\},\ve )$.

In the definitions below, suppose moreover that $X$ is compact.

For Lipschitz maps $f,\,g \: X \to X$, we write 
\begin{equation*}
	\LIP_d(f) \= \sup  \{  d(f(x), f(y)) / d(x,y)  : x,\, y \in X, \, x\neq y \}.
\end{equation*}
If there is no confusion regarding the metric, we shall omit the subscript $d$ and simply write $\LIP(f)$.
We denote the uniform distance between $f$ and $g$ by 
\begin{equation*}
	d_\infty(f,g) \= \sup \{ d(f(x), g(x)) : x\in X \}.
\end{equation*}

Given a map $f\: X\to X$, for each $x\in X$ we denote the $f$-forward orbit of $x$ by $\cO^f(x) \= \{ f^n(x) \}_{n=0}^{+\infty}$. 
If $\phi \: X \to \R$ is a function, for every $n\in \N$, we denote
\begin{equation*}
	S_n^f \phi \= \sum_{i=0}^{n-1} \phi \circ f^i.
\end{equation*}
If $K \subseteq X$ is a finite subset, we denote the cardinality of $K$ by $\card K$, and when $K$ is nonempty we define the \emph{gap} of $K$, denoted by $\Delta(K)$, to be $\Delta(K) \= +\infty$ if $\card K =1$ and 
$
	\Delta(K) \= \min \{ d(x,y) : x,\, y \in X, \, x\neq y \}
$
if $\card K \ge 2$.  
The collection of all $f$-periodic orbits is denoted by $\Per(f)$. For every $x\in X$, we denote the Dirac measure at $x$ by $\delta_x$. If $\cO\in \Per(f)$, we denote
\begin{equation*}
	\mu_{\cO} \= \frac{1}{\card\cO} \sum_{x\in \cO} \delta_x.
\end{equation*}

For $\alpha\in (0,1]$, let $\Holder{\alpha}(X,\R)$ denote the Banach space of real-valued $\alpha$-H\"older functions defined on $X$, equipped with the norm
\begin{equation*}
	\Hnorm{\alpha,X}{\phi} \= \Hnorm{\infty,X}{\phi} + \Hseminorm{\alpha,X}{\phi}, 
\end{equation*}
where $\Hnorm{\infty,X}{\phi} \= \sup_{x\in X} \abs{\phi(x)}$ and $\Hseminorm{\alpha,X}{\phi} \= \sup \bigl\{ \frac{\abs{\phi(x)-\phi(y)}}{d(x,y)^\alpha} : x,\, y \in X, \, x\neq y \bigr\}$.
Let $C^1(M,\R)$ denote the Banach space of real-valued $C^1$ functions defined on a compact Riemannian manifold $M$, equipped with the norm
\begin{equation*}
	\Hnorm{C^1,M}{\phi} \= \Hnorm{\infty,M}{\phi} + \Hnorm{\infty,M}{\mathrm{D}\phi}, 
\end{equation*}
where $\mathrm{D}\phi$ is the derivative of $\phi$. We often omit the subscripts $X$ or $M$ if there is no risk of confusion.

\section{Axiomatic hyperbolic framework and joint perturbation theorems}\label{joint perturbation}

In this section, we first prove two \emph{joint perturbation} theorems (Theorems~\ref{jointperturbation} and~\ref{jointperturbation'}) 
within an abstract 
hyperbolic framework, generalising that of
\cite{HHJL25}.
We then use these joint perturbation theorems 
to give
a sufficient condition for a pair $(f,\phi)$ to lie in the interior of the joint locking set (see Theorem~\ref{openness}). 

The abstract \emph{stable hyperbolicity} condition 
underlying our theory is defined as follows.

\begin{definition}[\bf $\cF$-stably hyperbolic maps]\label{Fstablyhyperbolic}
    Let $\cF$ be a topological space of Lipschitz self-maps on a compact metric space $(X,d)$. We say that $f\in \cF$ is \emph{$\cF$-stably hyperbolic} if the following conditions hold:
	\begin{enumerate}
		\smallskip
		\item[(\textbf{IS})] (\textbf{Intertwining stability.}) 
		For every $\varepsilon>0$, there exists a neighbourhood $U$ of $f$ in $\cF$ such that if $g\in U$ then $\LIP(g) < 2\LIP(f)$, and there exist maps
	\begin{equation*}
		h_g \: \Omega(f) \to \Omega(g)
		\quad\text{ and }\quad
		i_g \: \Omega(g) \to \Omega(f)
	\end{equation*}
		satisfying
		$h_g \circ f = g \circ h_g$, $i_g \circ g = f \circ i_g$, $d_\infty (h_g , \id) < \ve$, and $d_\infty (i_g , \id) < \ve$.

		\smallskip
		\item[(\textbf{RHE})] (\textbf{Robust hyperbolic estimates.})  
		There exists a neighbourhood $U$ of $f$, and constants $K=K_U>0$, $\delta=\delta_U>0$, $\lambda=\lambda_U>1$, such that for all $g\in U$,  $x,\,y \in \Omega(g)$, and  $n\in\N$, if
		\begin{equation*}
				\max_{0\le m\le n} d(g^m(x),g^m(y)) < \delta,
		\end{equation*}
	then for all $0\le m \le n$,
		\begin{equation}\label{assumptionb}
			d(g^m(x),g^m(y))
			\le
			K\,\lambda^{-\min\{m,\,n-m\}}
			(d(x,y)+d(g^n(x),g^n(y))).
		\end{equation}

		\item[(\textbf{ML})] (\textbf{Ma\~n\'e lemma.})  
		For every $\alpha\in(0,1]$, there exists a constant $L>0$ such that for all $\phi\in\Holder{\alpha}(X,\R)$ there exists a function $u\in\Holder{\alpha}(\Omega(f),\R)$ satisfying
		\begin{equation*}
			 \phi+u-u\circ f  \le \mpe(f,\phi) \ \ \text{on }\Omega(f), \quad \text{ and } \quad
			\Hseminorm{\alpha}{u}  \le L\,\Hseminorm{\alpha}{\phi} .
		\end{equation*}
	\end{enumerate}
\end{definition}    

\begin{rem}\label{rm_assumption_(b)_sum}
\smallskip
    (i)
	Suppose that $\cF$ is a topological space of Lipschitz self-maps on a compact metric space $(X,d)$, that $f\in \cF$ satisfies condition (RHE) in Definition~\ref{Fstablyhyperbolic}, and that $\alpha\in (0,1]$. By direct calculation, we see that if $x$ and $y$ satisfy $\max\limits_{0\leq m\leq n} \{d ( g^m(x), g^m(y)  ) \}<\rho\leq \delta$, then
	\begin{equation*}
		\sum_{m=0}^{n} d(g^m(x), g^m(y))^\alpha 
        \leq K^\alpha \sum_{m=0}^{n} (\lambda^{-m\alpha}+\lambda^{-(n-m)\alpha}) 
		(d(x,y)+d(g^n(x),g^n(y)))^\alpha
		 \leq \frac{ 4 K^\alpha \rho^\alpha\lambda^\alpha}{\lambda^\alpha-1}.
	\end{equation*}

    (ii) Note that by combining condition (ML) and the McShane extension theorem (see e.g.~\cite[Theorem~1.33]{Wea18}), the function $u$, which is defined on $\Omega(f)$ and $\alpha$-H\"older, can be extended to a function in $\Holder{\alpha}(X,\R)$, with the same H\"older seminorm $\Hseminorm{\alpha,\Omega(f)}{u}$. 
\end{rem}

We begin by establishing a joint perturbation theorem for H\"older potentials, generalising that of \cite{HHJL25}.

\begin{theorem}[{\bf Joint perturbation for H\"older potentials}]
	\label{jointperturbation}
	Let $\alpha\in (0,1]$, and let $\cF$ be a topological space of Lipschitz self-maps on a compact metric space $(X,d)$.  
	If $f\in \cF$ is $\cF$-stably hyperbolic, and $\cO$ is an $f$-periodic orbit,
	then there exist a neighbourhood $U$ of $f$, and $C>0$, such that for all $g\in U$, the following hold:
	\begin{enumerate}[label=\rm{(\roman*)}]
		\smallskip
		\item The maps $h_g$ and $i_g$ in condition (IS) in Definition~\ref{Fstablyhyperbolic} exist. 
		
		\smallskip
		\item If $\cO_g \= h_g(\cO)$, $d_g\=\max\{d_\infty(h_g,\id),\, d_\infty(i_g,\id)\}$, and a nonconstant function $\phi\in \Holder{\alpha}(X,\R)$ satisfies $\Mmax(f,\phi)=\{\mu_{\cO}\}$,
		then 
		\begin{equation*}
			\{\mu_{\cO_g} \} = \Mmax \bigl( g, \phi - 2C \, \Hseminorm{\alpha}{\phi}\, d_g^{\alpha/2} \, d(\cdot,\cO_g)^\alpha+\xi \bigr)
		\end{equation*}
		for all $\xi\in \Holder{\alpha}(X,\R)$ satisfying
		$\Hseminorm{\alpha}{\xi} \le 5C \Hseminorm{\alpha}{\phi} d_g^{\alpha/2}$ and $\Hnorm{\infty}{\xi} \le \Hseminorm{\alpha}{\phi}  d_g^\alpha$.
	\end{enumerate}

\end{theorem}
\begin{proof}
    Fix an $\cF$-stably hyperbolic map $f$ and a periodic orbit $\cO$ of $f$. By condition (IS) in Definition~\ref{Fstablyhyperbolic}, there is a neighbourhood $U_a$ of $f$ such that for all $g\in U_a$, the maps $h_g$ and $i_g$ exist, 
	\begin{align}
		\LIP(g) &\le 2 \LIP(f), \quad \text{ and } \label{lipbound}\\
		d_g & < \min \{  \Delta(\cO) /4, \, 1\}.\label{hgbound}
	\end{align}
	By condition (RHE) in Definition~\ref{Fstablyhyperbolic}, there is a neighbourhood $U_b$ of $f$ and constants $K>0$, $\delta>0$, and $\lambda>1$ such that (\ref{assumptionb}) holds for $g\in U_b$. Let $L>0$ be the constant obtained by condition (ML) in Definition~\ref{Fstablyhyperbolic}. Then (i) holds in the neighbourhood 
	\begin{equation}\label{U}
		U \= U_a \cap U_b.
	\end{equation}

	Define constants 
	\begin{align}
		p_0&\=\card \cO, \label{p_0}\\
		r &\= \min  \{  \Delta(\cO) / (8\LIP(f)), \, \delta \}, \label{r}\\
		L_1 &\= 5+2L, \label{L1}\\
        L_2 &\= 2(2K)^\alpha\lambda^\alpha/(\lambda^\alpha-1) ,  \label{L2} \\
		L_3 &\= 1+L+L(2\LIP(f))^\alpha,\quad \text{ and} \label{L3} \\
        C \= \max \{ 1, \, 
          10 L_2 L_1 & r^{-\alpha}(2\LIP(f))^\alpha     , \, 
        2(p_0+L_2L_3)L_1 r^{-\alpha} (2\LIP(f))^\alpha \}.   \label{C}
    \end{align}
	Fix a nonconstant $\phi\in \Holder{\alpha}(X,\R)$ with $\Mmax(f,\phi)=\{ \mu_{\cO} \}$, $g\in U$, and $\xi \in \Holder{\alpha}(X,\R)$, with 
	\begin{equation}\label{xibound}
		\Hseminorm{\alpha}{\xi} \le 5C \Hseminorm{\alpha}{\phi} d_g^{\alpha/2} \quad\text{ and } \quad \Hnorm{\infty}{\xi} \le \Hseminorm{\alpha}{\phi}  d_g^\alpha.
	\end{equation}
	In particular, since $\phi$ is nonconstant then $\Hseminorm{\alpha}{\phi}>0$.

Since $h_g \circ f = g \circ h_g$, the image $\cO_g \= h_g(\cO)$ is a 
    $g$-periodic orbit. For any two distinct $x,\, x'\in\cO$, the triangle 
    inequality and $d_\infty(h_g,\id)\le d_g < \Delta(\cO)/4$ give
    \begin{equation*}
        d(h_g(x),h_g(x')) \ge d(x,x') - 2d_g \ge \Delta(\cO) - \Delta(\cO)/2 = \Delta(\cO)/2.
    \end{equation*}
    In particular $h_g$ is injective on $\cO$, so $p \= \card\cO_g = \card\cO = p_0$
    and $\Delta(\cO_g) \ge \Delta(\cO)/2$.
    Hence $C$ as defined in~(\ref{C}) satisfies 
    \begin{equation}\label{C_p}
        C \ge 2(p + L_2 L_3)L_1 r^{-\alpha}(2\LIP(f))^\alpha.
    \end{equation}

	If $\LIP(g)$ were strictly smaller than $1$ then $g$ would be a strict contraction, with $\Omega(g)$ consisting 
of a single fixed point $q$, and $\mu_{\cO_g} = \delta_q$ the unique 
$g$-invariant probability measure; hence $\Mmax(g, \psi) = \{\delta_q\} 
= \{\mu_{\cO_g}\}$ for every continuous $\psi$, and part~(ii) would hold 
trivially. We therefore assume that $\LIP(g) \ge 1$ for the remainder 
of the proof.
    
    From (\ref{r}) and (\ref{lipbound}),
	\begin{equation}\label{rbound}
		r 
        \le  \Delta(\cO) / (8\LIP(f))
        \le  \Delta(\cO_g) / (2\LIP(g)).
	\end{equation}
	By condition (ML) in Definition~\ref{Fstablyhyperbolic} applied to the map $f$, and Remark~\ref{rm_assumption_(b)_sum}~(ii), there exists $u\in \Holder{\alpha}(X,\R)$ satisfying
	\begin{equation}\label{mane_psi}
		\psi\= \overline{\phi}+u-u\circ f\leq 0 \text{ on } \Omega(f)  \quad \text{ and } \quad
        		\Hseminorm{\alpha}{u}\leq L\Hseminorm{\alpha}{\phi}.
	\end{equation}
	where $\overline{\phi}$ denotes $ \phi- \mpe(f,\phi)$.
	 Define potentials
	\begin{align}
		\Phi &\= \overline{\phi} - C \, \Hseminorm{\alpha}{\phi}\, d_g^{\alpha/2} \, d(\cdot,\cO_g)^\alpha+\xi, \label{Phi}\\
		\psi_g&\=  \overline{\phi}+ u - u \circ g +\xi,\label{psig}\\
		\Psi_g&\=\psi_g- C \, \Hseminorm{\alpha}{\phi}\, d_g^{\alpha/2} d(\cdot,\cO_g)^\alpha=\Phi+u-u\circ g \label{Psig}.
	\end{align}
	Define constants
	\begin{align}
		\tau &\= (3+2L) \Hseminorm{\alpha}{\phi} d_g^\alpha, \label{tau} \\
		\eta &\= \int \! (  \overline{\phi} + \xi   )\,\mathrm{d}\mu_{\cO_g} 
        = \int  \!  \psi_g  \,\mathrm{d}\mu_{\cO_g} 
        = \int  \! \Psi_g \,\mathrm{d}\mu_{\cO_g}
        =\frac{1}{p_0}\sum\limits_{x\in \cO}(\overline{\phi}+\xi)(h_g(x)), \label{eta}
	\end{align}
	where the second equality in (\ref{eta}) follows from (\ref{psig}) and the third from (\ref{Psig}). 

For every $x\in \Omega(g)$, we have $i_g(x) \in \Omega(f)$. So by (\ref{mane_psi}),
$\overline{\phi}(i_g(x)) \le u(f(i_g(x))) - u(i_g(x))$.
By the H\"{o}lder continuity of $\overline{\phi}$, and the fact that $d_\infty(i_g, \id) \le d_g$,
\begin{equation*}
    \overline{\phi}(x) \le \overline{\phi}(i_g(x)) + \Hseminorm{\alpha}{\phi}\, d(x,i_g(x))^\alpha
    \le u(f(i_g(x))) - u(i_g(x)) + \Hseminorm{\alpha}{\phi}\, d_g^\alpha.
\end{equation*}
Using this, the fact that $i_g\circ g=f\circ i_g$ (see~(i) and condition (IS) in Definition~\ref{Fstablyhyperbolic}), and
(\ref{psig}), (\ref{xibound}), (\ref{mane_psi}), (\ref{tau}), we obtain
\begin{align}\label{psig<tau}
    \psi_g(x) 
    &= \overline{\phi}(x) + u(x) - u(g(x)) + \xi(x)
    \leq u(x) - u(i_g(x)) + u(i_g(g(x))) - u(g(x))
      + 2\Hseminorm{\alpha}{\phi}\, d_g^\alpha  \notag \\
    &\leq \Hseminorm{\alpha}{u}\, d(x,i_g(x))^\alpha
      + \Hseminorm{\alpha}{u}\, d(i_g(g(x)),g(x))^\alpha
      + 2\Hseminorm{\alpha}{\phi}\, d_g^\alpha
    \leq (2L+2)\Hseminorm{\alpha}{\phi}\, d_g^\alpha 
    < \tau.
\end{align}
	 By (\ref{eta}), (\ref{xibound}), and (\ref{hgbound}), we have
	\begin{equation}\label{etabound}
		\eta 
        = \int  \!   \overline{\phi} + \xi \,\mathrm{d}\mu_{\cO_g} 
        \ge \frac{1}{p_0} \sum_{x\in \cO} (\overline{\phi}\circ  h_g)(x) - \Hseminorm{\alpha}{\phi} d_g^\alpha 
        \ge \int \! \overline{\phi}\,\mathrm{d}\mu_{\cO} - 2\Hseminorm{\alpha}{\phi} d_g^\alpha  
        \ge - 2\Hseminorm{\alpha}{\phi} d_g^\alpha.
	\end{equation}
	Using (\ref{tau}), (\ref{etabound}), and (\ref{L1}), we estimate
	\begin{equation}\label{tau-eta}
		\tau-\eta\leq (5+2L)\Hseminorm{\alpha}{\phi} d_g^{\alpha} \le L_1 \Hseminorm{\alpha}{\phi} d_g^{\alpha}. 
	\end{equation}
	By (\ref{eta}) and (\ref{psig<tau}), we have
	$\eta=\int  \!  \psi_g  \,\mathrm{d}\mu_{\cO_g} <\tau$.
	So we can define 
	\begin{equation}\label{rho}
		\rho \= \bigl( C \Hseminorm{\alpha}{\phi} d_g^{\alpha/2} \big/ (\tau-\eta) \bigr)^{-1/\alpha}>0.
	\end{equation}
	By (\ref{Psig}), (\ref{psig<tau}), and (\ref{rho}), we have
	\begin{equation}\label{psihC<eta}
		\Psi_{g}(x) 
        < \tau-C \, \Hseminorm{\alpha}{\phi}\, d_g^{\alpha/2}\rho^\alpha 
        =\eta 
        \quad \text{ if } \quad x\notin B(\cO_g, \rho).
	\end{equation}
	Moreover, using (\ref{tau-eta}) and (\ref{rho}), we estimate
	\begin{equation}\label{rhoe}
		\rho \le (L_1/C)^{1/\alpha} \cdot d_g^{1/2} .
	\end{equation}
	
	We wish to prove that $ \mu_{\cO_g}  \in \Mmax (g, \Phi)$. Since every $g$-invariant probability measure is supported on $\Omega(g)$ (cf.~\cite[Theorem~6.15~(i) and Theorem~5.6~(i)]{Wa82}), this is equivalent to proving that $\mu_{\cO_g} \in \Mmax \bigl(\Omega(g), g|_{\Omega(g)}, \Phi|_{\Omega(g)} \bigr)$. By (\ref{Psig}), \cite[Proposition~2.2]{Je19}, and (\ref{eta}), it suffices to establish that 
	\begin{equation}\label{aim}
		\liminf_{n\to +\infty} \frac{1}{n} S_n^g \Psi_g (x)=\liminf_{n\to +\infty} \frac{1}{n} S_n^g \Phi (x) \le \eta,\quad\text{for every }	x\in \Omega(g).
	\end{equation}

Fix $x\in \Omega(g)$. We recursively construct a sequence $\{x_t\}_{t=1}^{s+1}$ 
of points in $\cO^g(x)$ and a sequence $\{n_t\}_{t=1}^{s}$ of positive integers, 
for some $s\in \N \cup \{+\infty\}$, such that
	\begin{equation}\label{contruction}
		x_{t+1}  = g^{n_t}(x_t) \quad \text{ and } \quad S_{n_t}^g \Psi_{g}(x_t) \le n_t \eta \quad \text{ for every } t\in \N \cap [0,s].
	\end{equation}

	\smallskip
	\emph{Base step.} Define $x_1 \= x$.
	
	\smallskip
	\emph{Recursive step.} Assume that for some $t\in \N$, the finite sequences $\{x_i\}_{i=1}^{t}$ and $\{n_i\}_{i=1}^{t-1}$ are defined. 
	Consider the following three cases.
	
	\smallskip
	\emph{Case~A.} Assume $x_t\notin B(\cO_g, \rho)$. In this case, we define $n_t \= 1$ and $x_{t+1} \= g(x_t)$. By (\ref{psihC<eta}),
	\begin{equation}\label{case1}
		S_{n_t}^g \Psi_{g}(x_t) = \Psi_{g}(x_t)< \eta = n_t \eta.
	\end{equation}
	
	\smallskip
	\emph{Case~B.} Assume $\cO^g(x_t) \subseteq B(\cO_g, r)$. Let $y\in \cO_g$ be such that $d(x_t,y) = d(x_t, \cO_g)$. By (\ref{rbound}), we have $d(g(x_t), g(y)) \leq\LIP(g)d(x_t,y)< r\LIP(g)\leq  2^{-1} \Delta(\cO_g)$, which implies $d(g(x_t), g(y)) = d(g(x_t), \cO_g)\le r$. Using an inductive argument, we conclude that $d\bigl(g^l(x_t), g^l(y)\bigr)< 2^{-1} \Delta(\cO_g)$ for all $l\in \N$ and 
	\begin{equation}\label{dequal}
		d\bigl(g^l(x_t), g^l(y)\bigr) = d\bigl(g^l(x_t), \cO_g\bigr)< r \quad \text{ for all } \quad l\in \N.
	\end{equation}
	By (\ref{Psig}) and (\ref{dequal}), we have
	\begin{equation*}\label{c2e1}
			\liminf_{n\to +\infty} \frac{1}{n} S_n^g \Psi_{g} (x) 
            = \liminf_{n\to +\infty} \frac{1}{n} S_n^g \Psi_{g} (x_t) 
			\le \liminf_{n\to +\infty} \frac{1}{n} \biggl( S_n^g \Psi_{g} (y) + \Hseminorm{\alpha}{\Psi_{g}} \sum_{l=0}^{n-1} d\bigl(g^l(x_t), g^l(y)\bigr)^\alpha \biggr).
	\end{equation*}
	As $r\le \delta$ (see (\ref{r})), by Remark~\ref{rm_assumption_(b)_sum} we have
	\begin{equation*}\label{c2e2}
		\sum_{l=0}^{n-1} d\bigl(g^l(x_t), g^l(y)\bigr)^\alpha 
        \le \frac{4 K^\alpha r^\alpha\lambda^\alpha}{\lambda^\alpha-1}.
	\end{equation*}
	Combining the above two inequalities and (\ref{eta}), we obtain 
	$\liminf\limits_{n\to +\infty} \frac{1}{n} S_n^g \Psi_{g} (x) \leq  \eta$,
	which is precisely the required (\ref{aim}). This completes the recursive step.
	
	\smallskip
	\emph{Case~C.} Assume $\cO^g(x_t) \nsubseteq B(\cO_g, r)$ but $x_t \in B(\cO_g, \rho)$. Let $y\in \cO_g$ satisfy $d(x_t,y)= d(x_t, \cO_g)$. Noting that $d_g\le 1$ by (\ref{hgbound}), $C\ge 10 L_2 L_1 r^{-\alpha}(2\LIP(f))^\alpha$ by (\ref{C}), and our assumption that $\LIP(g)\ge 1$, using (\ref{rhoe}) and (\ref{lipbound}), we obtain that
	\begin{equation}\label{rhobound}
		\rho \le   (10L_2)^{-1/\alpha}   \cdot 2^{-1} r / \LIP(f)  \le  r / \LIP(g) \le r.
	\end{equation}
	Now define integers 
	\begin{align}
		N &\= \min \bigl\{ i\in\Z : i \ge -1 ,\, d\bigl(g^{i+1}(x_t), g^{i+1}(y)\bigr) \ge r \bigr\}, \label{N}\\
		m &\= \max \bigl\{ i\in\Z : 1\le i \le N ,\, d\bigl(g^{i-1}(x_t), g^{i-1}(y)\bigr) < \rho \bigr\}, \label{m}
	\end{align}
	where the existence of $N$ follows from the assumptions of this case, and (\ref{rhobound}) implies that $1 \le N$, so $m$ is well defined. Using (\ref{N}), and an argument analogous to the one used to prove (\ref{dequal}) in Case~B, yields
	\begin{equation}\label{deq}
		d\bigl(g^i(x_t), g^i(y)\bigr) = d\bigl(g^i(x_t), \cO_g\bigr) <r \quad \text{ for all } 0\le i\le N. 
	\end{equation}
By definition of $N$ (cf.~(\ref{N})), we have 
    $d\bigl(g^N(x_t), g^N(y)\bigr) < r$ but 
    $d\bigl(g^{N+1}(x_t), g^{N+1}(y)\bigr) \ge r$.
    Since $d\bigl(g^{N+1}(x_t),g^{N+1}(y)\bigr) \le \LIP(g)\, 
    d\bigl(g^N(x_t),g^N(y)\bigr)$, we obtain
    \begin{equation}\label{Case_C_g_N}
        d\bigl(g^N(x_t), g^N(y)\bigr) 
        \ge r/\LIP(g)
        \ge r/(2\LIP(f))
        \= r_0,
    \end{equation}
    where the last inequality uses~(\ref{lipbound}).
    
	In this case we define $n_t \= N+1$ and $x_{t+1} \= g^{N+1}(x_t)$. 
	
	Next, we estimate $S_{n_t}^g \Psi_{g}(x_t)$ so as to deduce (\ref{contruction}). Since $\Psi_g\leq \psi_g$, direct calculation gives
	\begin{equation}\label{c3e0}
		\begin{aligned}
			S_{n_t}^g \Psi_{g} (x_t) 
            &\le S_m^g \psi_{g}(x_t) + S_{N-m}^g \Psi_{g} (g^m(x_t)) + \Psi_{g} \bigl(g^N(x_t)\bigr) \\
			&\le S_m^g \psi_g(y) + \abs{ S_m^g \psi_g(x_t) - S_m^g \psi_g(y) } +  S_{N-m}^g \Psi_{g} (g^m(x_t)) + \Psi_{g} \bigl(g^N(x_t)\bigr).
		\end{aligned}
	\end{equation}
	Next we estimate the four terms on the righthand side of (\ref{c3e0}) separately. For the first term, write $m=pq+l$ for some $q\in \N_0$ and $0\le l \le p-1$, then by (\ref{eta}) and (\ref{psig<tau}), we have
	\begin{equation}\label{c3e1}
		S_m^g \psi_g(y) = q S_p^g \psi_g(y) + S_l^g \psi_g (y) \le pq \eta + l \tau \le m \eta + (p-1)(\tau-\eta).
	\end{equation}
For the second term, since $d(g^i(x_t),g^i(y)) < r \le \delta$ for all $0\le i\le m-1$ (by \eqref{deq}), and the endpoints of this orbit segment satisfy $d(x_t,y)<\rho$ and $d(g^{m-1}(x_t),g^{m-1}(y))<\rho$ (by the definitions of $x_t$ and $m$), applying condition (RHE) in Definition~\ref{Fstablyhyperbolic} to this segment yields
\begin{equation*}
        d(g^i(x_t),g^i(y)) < 2K\lambda^{-\min\{i, \, m-1-i\}}\rho \quad \text{ for each } 0\le i\le m-1.
    \end{equation*}
    Therefore, using (\ref{psig}), condition (ML) in Definition~\ref{Fstablyhyperbolic}, (\ref{lipbound}),
(\ref{xibound}), (\ref{L2}), and (\ref{L3}),
\begin{align}
    \abs{ S_m^g \psi_g(x_t) - S_m^g \psi_g(y) } 
    &\le \Hseminorm{\alpha}{\psi_g}
       \sum_{i=0}^{m-1} d \bigl( g^i(x_t), g^i(y) \bigr)^\alpha
    \label{c3e2}\\
    &\le \Hseminorm{\alpha}{\psi_g}
       (2K\rho)^\alpha \frac{2\lambda^\alpha}{\lambda^\alpha-1}
    \le L_2\rho^\alpha\,\Hseminorm{\alpha}{\phi}
       \bigl(L_3+5Cd_g^{\alpha/2}\bigr). \notag
\end{align}
	For the third term, by (\ref{m}), (\ref{deq}), and (\ref{psihC<eta}), we have
	\begin{equation}\label{c3e3}
		S_{N-m}^g \Psi_{g} (g^m(x_t)) \le (N-m)\eta.
	\end{equation}
	For the fourth term, by (\ref{Psig}), (\ref{psig<tau}), (\ref{N}), (\ref{deq}), and (\ref{Case_C_g_N}), we have
	\begin{equation}\label{c3e4}
		\Psi_{g} \bigl(g^N(x_t)\bigr) \le \tau - C\Hseminorm{\alpha}{\phi} d_g^{\alpha/2} d\bigl(g^N(x_t), \cO_g\bigr)^\alpha \le \tau - C\Hseminorm{\alpha}{\phi} d_g^{\alpha/2} r_0^\alpha.
	\end{equation}
	Finally, combining (\ref{c3e0})--(\ref{c3e4}), and using (\ref{tau-eta}) and (\ref{rhoe}), we obtain
	\begin{equation}\label{case3}
		\begin{aligned}
			S_{n_t}^g \Psi_{g} (x_t) - n_t \eta 
			& \le p(\tau-\eta) + \bigl(L_3+5Cd_g^{\alpha/2} \bigr)L_2 \rho^\alpha \Hseminorm{\alpha}{\phi} - C\Hseminorm{\alpha}{\phi} d_g^{\alpha/2} r_0^\alpha \\
			&  \le p L_1 \Hseminorm{\alpha}{\phi} d_g^\alpha + L_1 L_2 L_3 \Hseminorm{\alpha}{\phi} d_g^{\alpha/2} + 5CL_2 \rho^\alpha d_g^{\alpha/2}\Hseminorm{\alpha}{\phi} - C r_0^\alpha \Hseminorm{\alpha}{\phi} d_g^{\alpha/2}.
		\end{aligned}
	\end{equation}
	Since $d_g\le 1$ (cf.\ (\ref{hgbound})) and $C\ge 2(p+L_2L_3)L_1 r_0^{-\alpha}$ (cf.\ (\ref{C_p})), we obtain
	\begin{equation}\label{final1}
		p L_1 \Hseminorm{\alpha}{\phi} d_g^\alpha + L_1 L_2 L_3 \Hseminorm{\alpha}{\phi} d_g^{\alpha/2} \le 2^{-1} C  r_0^\alpha \Hseminorm{\alpha}{\phi} d_g^{\alpha/2}.
	\end{equation}
	Using (\ref{rhobound}) and (\ref{Case_C_g_N}), we obtain
	$5CL_2 \rho^\alpha d_g^{\alpha/2}\Hseminorm{\alpha}{\phi} 
    \le 2^{-1} C r_0^\alpha \Hseminorm{\alpha}{\phi} d_g^{\alpha/2}$.
	This, combined with (\ref{case3}) and (\ref{final1}), gives
	\begin{equation}\label{case3final}
		S_{n_t}^g \Psi_{g} (x_t) \le n_t \eta.
	\end{equation}
So the required inequality (\ref{contruction}) holds, and therefore the recursive step is complete.

    If the recursion never terminates (i.e.,~$s=+\infty$, meaning Case~B never 
occurs),
    then setting $N_t \= \sum_{i=1}^{t} n_i$ for every $t\in \N$, by (\ref{case1}) and (\ref{case3final}), we have
	\begin{equation*}
		\liminf_{n\to +\infty} \frac{1}{n}S_n^g \Psi_g (x) 
        \le \liminf_{t\to +\infty} \frac{1}{N_t} \sum_{i=1}^{t} S_{n_i}^g \Psi_g(x_i)
        \le \lim_{t\to +\infty} \frac{1}{N_t} \sum_{i=1}^{t} n_i \eta 
        \= \eta.
	\end{equation*}
Thus (\ref{aim}) holds, so $\mu_{\cO_g}\in \Mmax(g,\Phi)$.
Now define $h \= -C \, \Hseminorm{\alpha}{\phi}\, d_g^{\alpha/2} \, d(\cdot,\cO_g)^\alpha$.
Since $h \le 0$ with $h = 0$ exactly on $\cO_g$, the maximum ergodic average 
$\mpe(g, h) = 0$ is achieved precisely by those $g$-invariant measures supported 
on $\cO_g$. Since $\cO_g$ is a $g$-periodic orbit, the unique $g$-invariant measure 
supported on $\cO_g$ is $\mu_{\cO_g}$, so $\Mmax(g, h) = \{\mu_{\cO_g}\}$.
This, combined with the fact that $\mu_{\cO_g}\in\Mmax(g,\Phi)$, gives
$\Mmax(g, \Phi + h) = \{\mu_{\cO_g}\}$.
As $\Phi + h = \phi - \mpe(f,\phi) - 2C\,\Hseminorm{\alpha}{\phi}\,d_g^{\alpha/2}
d(\cdot,\cO_g)^\alpha + \xi$, we conclude that $\{\mu_{\cO_g} \} = 
\Mmax \bigl(g,\phi - 2C\,\Hseminorm{\alpha}{\phi}\,d_g^{\alpha/2}\,d(\cdot,\cO_g)^\alpha+\xi \bigr)$,
as required.
\end{proof}

Now we prove a $C^1$ analogue of Theorem~\ref{jointperturbation},
i.e.,~the joint perturbation theorem for $C^1$ potentials on a compact smooth manifold $M$.

\begin{theorem}[\bf  Joint perturbation for $C^1$ potentials]\label{jointperturbation'}
	Let $M$ be a compact smooth manifold equipped with a Riemannian metric, and let $\cF$ be a topological space consisting of Lipschitz self-maps of $M$. 
    If $f\in \cF$ is $\cF$-stably hyperbolic, and $\cO$ is an $f$-periodic orbit,
	then there exist a neighbourhood $U$ of $f$, and $C>0$, such that for all $g\in U$, the following hold:
	\begin{enumerate}[label=\rm{(\roman*)}]
		\smallskip
		\item The maps $h_g$ and $i_g$ in condition (IS) in Definition~\ref{Fstablyhyperbolic} exist.
		
		\smallskip
		\item Denote $\cO_g \= h_g(\cO)$ and $d_g \= \max\{ d_\infty(h_g,\id), \, d_\infty(i_g,\id)\}$. There exists $v_g\in C^1(M,\R)$ such that
		\begin{equation}\label{e_property_v_g}
			\Hnorm{\infty}{\mathrm{D} v_g} \le 3C  d_g^{1/2}\quad \text{ and } \quad \mu_{\cO_g} \in\Mmax(g,v_g),
        \end{equation}
		and for all $\phi\in C^1(M,\R)$ with $\Mmax(f,\phi) = \{\mu_{\cO}\}$,  $$\Mmax(g,\phi+\Hnorm{\infty}{\mathrm{D}\phi}v_g) = \{\mu_{\cO_g}\}.$$
	\end{enumerate}

\end{theorem}
\begin{proof}
	Without loss of generality, assume that $\diam M=1$. 
	Let the neighbourhood $U$ and constant $C>0$ be as given by Theorem~\ref{jointperturbation} for the case $\alpha\=1$.
    Then (i) follows from Theorem~\ref{jointperturbation}~(i).

 	Fix $g\in U$ and $\phi\in C^1(M,\R)$ with $\Mmax(f,\phi)=\{\mu_{\cO}\}$.
	 By \cite[Theorem~2.7]{HLMXZ25}, there exists $w\in C^1(M,\R)$ satisfying
	 \begin{equation}\label{e.def.w}
	 	\Hnorm{\infty}{\mathrm{D}w} < 3/2 \quad \text{ and } \quad
        \Hnorm{\infty}{w+d(\cdot, \cO_g)} < d_g^{1/2} \big/(2C).
	 \end{equation}
	Define $\xi \= 2C \Hnorm{\infty}{\mathrm{D}\phi} d_g^{1/2} (d(\cdot,\cO_g) + w) $. By (\ref{e.def.w}), we obtain
	\begin{equation}\label{e.xi}
		\Hseminorm{1}{\xi} \le 5C\Hnorm{\infty}{\mathrm{D}\phi} d_g^{1/2} \quad \text{ and } \quad 
        \Hnorm{\infty}{\xi} \le \Hnorm{\infty}{\mathrm{D}\phi} d_g.
	\end{equation}
	Hence denoting $v_g \= 2Cd_g^{1/2}w$, by Theorem~\ref{jointperturbation}~(ii) and (\ref{e.def.w}), we obtain $\Mmax(g, \phi+\Hnorm{\infty}{\mathrm{D}\phi} v_g) = \{\mu_{\cO_g}\}$ and $\Hnorm{\infty}{\mathrm{D}v_g} \le 3Cd_g^{1/2}$. 
	
	So it suffices to show that $\mu_{\cO_g}\in \Mmax(g,v_g)$. 
    
	For each $n\in \N$, define $\phi_n \= \max \{-1/n, \, -d(\cdot, \cO) \} \in\Holder{1}(M,\R)$. Then clearly for each $n\in \N$,
	\begin{equation}\label{phin}
		\Hseminorm{1}{\phi_n} = 1, \qquad -1/n \le \phi_n \le 0, \qquad\text{ and } \cO=\phi_n^{-1}(0).
	\end{equation}
	Obviously, $\Mmax(f,\phi_n) = \{\mu_{\cO}\}$ for each $n\in \N$. 
	Fix arbitrary $\mu\in \MMM(M,g)$ and $n\in\N$.
	By Theorem~\ref{jointperturbation}~(ii) and (\ref{e.def.w}), we have $\Mmax(g, \phi_n + v_g) = \{ \mu_{\cO_g}\}$. So $\int \! ( \phi_n + v_g)  \,\mathrm{d}\mu_{\cO_g} > \int \!  (\phi_n + v_g)  \,\mathrm{d}\mu$. 
	By (\ref{phin}), $\int \! \phi_n  \,\mathrm{d}\mu \geq - \frac1n \ge \int \! \phi_n  \,\mathrm{d}\mu_{\cO_g} -\frac1n $. So we have $\int \!  v_g \,\mathrm{d}\mu_{\cO_g} > \int \! v_g  \,\mathrm{d}\mu - \frac1n$ for all $n\in\N$, 
	and hence $\int \!  v_g \,\mathrm{d}\mu_{\cO_g} \ge \int \! v_g  \,\mathrm{d}\mu $. 
This means that    $\mu_{\cO_g} \in \Mmax(g,v_g)$, as required.	
\end{proof}

To conclude this section, we now deduce a sufficient condition for a pair $(f,\phi)$ to belong to the interior of the joint locking set.

\begin{theorem}[\bf Interior Condition]\label{openness}
	Let $\alpha \in (0,1]$, $(X,d)$ be a compact metric space, and $(M, \rho)$ be a compact smooth manifold equipped with a Riemannian metric. 
	\begin{enumerate}[label=\rm{(\roman*)}]
		\smallskip
		\item If $\cF$ is a topological space consisting of Lipschitz self-maps of $X$, and $f\in\cF$ is $\cF$-stably hyperbolic, and $(f,\phi) \in \Lock \bigl(\cF, \Holder{\alpha}(X,\R)\bigr)$ where $\phi$ is nonconstant,  then $(f,\phi)$ belongs to the interior of the joint locking set $ \Lock\bigl(\cF, \Holder{\alpha}(X,\R)\bigr)$.
		
		\smallskip
		\item If $\cF$ is a topological space consisting of Lipschitz self-maps of $M$, and $f\in \cF$ is $\cF$-stably hyperbolic, and $(f,\phi) \in \Lock\bigl(\cF, C^1(M,\R)\bigr)$ where $\phi$ is nonconstant, then $(f,\phi)$ belongs to the interior of the joint locking set $ \Lock\bigl(\cF, C^1(M,\R)\bigr)$.
	\end{enumerate}
\end{theorem}

\begin{proof}
	Since the proofs of (i) and (ii) share a common  core, here we prove the more complicated (ii) in detail and give a more abbreviated proof of (i).

	(ii) Fix $(f,\phi) \in \Lock\bigl(\cF, C^1(M,\R)\bigr)$ such that $f$ is $\cF$-stably hyperbolic. Since $\phi$ is not a constant function, $\Hnorm{\infty}{\mathrm{D}\phi}>0$.
 Let us write $B(\phi, r)\= \bigl\{\psi \in C^1(M,\R) : \Hnorm{C^1}{\psi-\phi}<r \bigr\}$.
 
	Let $\cO$ be the $f$-periodic orbit satisfying $ \Mmax(f,\phi)=\{\mu_{\cO}\}$. Then there exists $\theta>0$ such that if $\psi \in B(\phi,\theta)$ then
	$\Mmax(f,\psi)=\{\mu_{\cO}\}$.
	We can assume without loss of generality that $\theta$ is small enough (i.e.,~$\theta \leq 2^{-1}\Hnorm{\infty}{\mathrm{D}\phi}$) such that
	\begin{equation}\label{seminormcontrol}
		\frac{1}{2}\Hnorm{\infty}{\mathrm{D}\phi}\leq \Hnorm{\infty}{\mathrm{D}\psi}\leq \frac{3}{2}\Hnorm{\infty}{\mathrm{D}\phi} \quad \text{ for all } \psi \in B(\phi, \theta).
	\end{equation}
	Applying Theorem~\ref{jointperturbation} (to $f$ and $\cO$), let the neighbourhood $U$ of $f$, and the constant $C>0$, be as in that theorem. 
	By condition (IS) in Definition~\ref{Fstablyhyperbolic}, there exists a neighbourhood $V\subseteq U$ of $f$ such that
	\begin{equation}\label{dg_bound_V}
		d_g<(9C(1+\diam(M))\Hnorm{\infty}{\mathrm{D}\phi}/\theta)^{-2} \quad \text{ for all }g\in V.
	\end{equation}
	For every $g\in V$, let $v_g$ be the function obtained from Theorem~\ref{jointperturbation'}~(ii), and assume without loss of generality that $v_g(x)=0$ for some $x\in M$. By Theorem~\ref{jointperturbation'}~(ii) and (\ref{dg_bound_V}), we have
	\begin{equation}\label{normbound_vg}
    \begin{aligned}
       		 \Hnorm{C^1}{v_g}
             = \Hnorm{\infty}{v_g}+\Hnorm{\infty}{\mathrm{D}v_g}
             &\leq (1+\diam (M))\Hnorm{\infty}{\mathrm{D}v_g} \\
             &\leq 3Cd_g^{1/2}(1+\diam (M))
             < \theta / (3\Hnorm{\infty}{\mathrm{D}\phi}). 
    \end{aligned}
	\end{equation} 
	By Theorem~\ref{jointperturbation'}, for all $g\in V$ and $\psi\in B(\phi,\theta)$, we have
	\begin{equation}\label{vgproperty}
		\Mmax(g, \psi+\Hnorm{\infty}{\mathrm{D}\psi}v_g) = \{\mu_{\cO_g}\} \subseteq \Mmax(g, v_g) .
	\end{equation} 
	Now fix $(g,\psi)\in V\times B(\phi,\theta/2)$, and set
	\begin{equation}\label{psi'}
		\psi'\=\psi-(3/2)\Hnorm{\infty}{\mathrm{D}\phi}v_g.
	\end{equation}  
	By (\ref{psi'}) and (\ref{normbound_vg}), we obtain
	$\Hnorm{C^1}{\psi-\psi'}\leq (3/2)\Hnorm{\infty}{\mathrm{D}\phi}\Hnorm{C^1}{v_g}\leq  \theta / 2$.
	Using this, and the fact that $\psi \in B(\phi,\theta/2)$, we deduce that $\psi'\in B(\phi,\theta)$. Thus applying (\ref{seminormcontrol}) to $\psi$ and $\psi'$, we obtain
	\begin{equation}\label{psi'phiseminorm}
		\Hnorm{\infty}{\mathrm{D}\psi}
        \geq  \Hnorm{\infty}{\mathrm{D}\phi} /2 
        \geq  \Hnorm{\infty}{\mathrm{D}\psi'} /3.
	\end{equation}
	From (\ref{vgproperty}) we obtain that $\{\mu_{\cO_g}\} = \Mmax(g, \psi'+\Hnorm{\infty}{\mathrm{D}\psi'}v_g )$.
    By (\ref{psi'}),
	\begin{equation*}
		\psi=\psi'+ (3/2) \Hnorm{\infty}{\mathrm{D}\phi}v_g=\psi'+\Hnorm{\infty}{\mathrm{D}\psi'}v_g+((3/2)\Hnorm{\infty}{\mathrm{D}\phi}-\Hnorm{\infty}{\mathrm{D}\psi'})v_g.
	\end{equation*}
	By (\ref{psi'phiseminorm}), $\mu_{\cO_g}\in \Mmax(g, ((3/2)\Hnorm{\infty}{\mathrm{D}\phi}-\Hnorm{\infty}{\mathrm{D}\psi'})v_g)$. Hence $\Mmax(g,\psi) = \{\mu_{\cO_g}\}.$
	
	Therefore, we have established that the neighbourhood $V \times B(\phi, \theta/2)$ of $(f,\phi)$ is contained in $\Lock(\cF , C^1(M,\R))$, so~(ii) is proved.
	\smallskip
	
	(i) Setting $v_g \= -2 C \Hseminorm{\alpha}{\phi} d_g^{\alpha/2} d(\cdot, \cO_g)^\alpha$, and using Theorem~\ref{jointperturbation}, we obtain properties of $v_g$ analogous to Theorem~\ref{jointperturbation'}~(ii), then~(i) follows by using the same argument as in the proof of~(ii).
\end{proof}

\begin{rem}\label{nonwander}
	\smallskip
	(i) 
    The role of the nonwandering set $\Omega(f)$ in Definition~\ref{Fstablyhyperbolic} can be generalised as follows.
    Let $\varsigma \: \cF \to 2^X$ be a map such that for every $f\in \cF$, $f(\varsigma(f)) \subseteq \varsigma(f)$, and every $f$-invariant measure is supported on $\varsigma(f)$. If $\Omega(f)$ is replaced by $\varsigma(f)$ throughout Definition~\ref{Fstablyhyperbolic}, then Theorem~\ref{openness} remains valid (a fact we shall use in the proof of Theorem~\ref{1dim}).
    This is because the only property of 
$\Omega(f)$ used in the proofs of Theorems~\ref{jointperturbation},~\ref{jointperturbation'}, and~\ref{openness} was that every $f$-invariant measure on $M$ is an invariant measure for the subsystem $(\Omega(f),f)$.
	
	\smallskip
	(ii) Using \cite[Proposition~2.2]{Je19}, we can see that if $f\in\cF$ satisfies conditions (IS) and (ML) in Definition~\ref{Fstablyhyperbolic} 
(or their modifications from~(i) above), and $(\phi, f)\in\Lock\bigl(\cF,
\Holder{\alpha}(X,\R)\bigr)$ for some $\alpha\in(0,1]$, 
then $\lim_{g\to f}\mpe(g,\phi)=\mpe(f,\phi)$. 
     Indeed from (\ref{psig<tau}) we  see that $\limsup_{g\to f} \mpe(g,\phi) \le \mpe(f,\phi)$, while from Theorem~\ref{jointperturbation}~(ii) we see that
    \begin{equation*}
        \liminf\limits_{g\to f} \mpe(g,\phi) \ge \liminf\limits_{g\to f} \mpe\bigl(g, \phi - 2C d_g^{\alpha/2} d(\cdot, \cO_g)^\alpha \bigr) = \liminf\limits_{g\to f} \int \! \phi \,\mathrm{d}\mu_{\cO_g},
    \end{equation*}
    and $\lim_{g\to f} d(h_g, \id|_{\Omega(f)}) = 0$, so $\lim_{g\to f} \int \! \phi \,\mathrm{d}\mu_{\cO_g} = \int \! \phi \,\mathrm{d}\mu_{\cO}$, and thus $\liminf_{g\to f} \mpe(g,\phi) \ge \mpe(f,\phi)$.

    \smallskip
    (iii) If $\phi$ is a constant potential, 
     then $(\phi,f)\in \Lock\bigl(\cF,
\Holder{\alpha}(X,\R)\bigr)$ if and only if the system is uniquely ergodic with unique invariant probability measure a periodic measure.
The assumption in Theorem~\ref{openness} that the potential $\phi$ is nonconstant
is included only to exclude the degenerate situation, which may occur in the abstract setting,
of the system being uniquely ergodic with unique invariant probability measure a periodic measure
(in which case the notion of ``locking'' becomes somewhat vacuous).
 In particular, for our applications in this paper, 
the nonconstancy of $\phi$ is an immediate consequence of
membership of $(\phi,f)$ in $ \Lock\bigl(\cF,
\Holder{\alpha}(X,\R)\bigr)$, rather than 
a genuine restriction.
Moreover, when $X$ is a manifold in our applications, and $\cP=\Holder{\alpha}(X)$ or $C^1(X)$,
the subspace of constant functions is a proper closed linear subspace of $\cP$,
and in particular nowhere dense in $\cP$.
\end{rem}

\section{Axiom A diffeomorphisms: proof of Theorem~\ref{JTPO AxiomA}}\label{sec_Axiom_A}

If $(X,d)$ is a compact metric space, and $f \: X \to X$ a continuous map, recall that a point $x\in X$ is said to be \emph{wandering} if there is a neighbourhood $U$ of $x$ such that $f^n(U) \cap U = \emptyset$ for each $n\in \N$, and the \emph{nonwandering set} $\Omega(f)$ of $f$ is defined as $$\Omega(f) \= X \smallsetminus \{ x\in X : x \text{ is wandering} \}.$$ It is well known that 
the nonwandering set is nonempty and compact, and every $f$-invariant probability measure has its support contained in $\Omega(f)$.
Moreover, if $f$ is a homeomorphism then $\Omega(f)$ is $f$-invariant, in the sense that $f(\Omega(f)) = \Omega(f)$ (cf.\ \cite[Theorem~1.4.9~(e)]{URM22}). 
A compact $f$-invariant subset $\Lambda\subseteq X$ is called \emph{(topologically) transitive} if 
 $\cO^f(x)$ is dense in $\Lambda$
for some $x\in \Lambda$.

Now let $M$ be
a compact smooth manifold (without boundary), with Riemannian metric $\Hseminorm{}{\,\cdot\,}$ on $M$, and  the induced distance function $d$. 

\begin{definition}[Axiom~A diffeomorphisms]\label{d.Axiom.A.diff}
	Let $f \: M \to M$ be a diffeomorphism, and $\mathrm{D}f$ its derivative. An $f$-invariant set $\Lambda$ is called \emph{hyperbolic} if for each $x\in \Lambda$, the tangent space $T_x M$ can be split into a direct sum $T_x M = E^s(x) \oplus E^u(x)$, where the subspaces $E^s(x)$ and $E^u(x)$ are $\mathrm{D}f$-invariant, i.e., $\mathrm{D}f(x) E^s(x) = E^s(f(x))$ and $\mathrm{D}f(x) E^u(x) = E^u(f(x))$, and there exist constants $C\ge1$ and $0<\xi<1$ such that 
	\begin{align*}
		\abs{\mathrm{D}f^n(x)(u)} &\le  C \xi^n \abs{u} \quad \text{ for all } x\in \Lambda, \, u\in E^s(x), n\in \N_0,\\
		\abs{\mathrm{D}f^{-n}(x)(u)} &\le  C \xi^n \abs{u} \quad \text{ for all } x\in \Lambda, \, u\in E^u(x), n\in \N_0.
	\end{align*}
	If the nonwandering set $\Omega(f)$ is hyperbolic, and the set of $f$-periodic points is dense in $\Omega(f)$, then $f$ is called an \emph{Axiom~A diffeomorphism}. 
\end{definition}

Let $f\: M \to M$ be an Axiom~A diffeomorphism. The spectral decomposition theorem (see \cite{Sm67} and cf.~e.g.~\cite[Theorem~5.2]{Wen16}) states that $\Omega(f)$ is uniquely decomposed into finitely many disjoint compact transitive subsets, i.e., $\Omega(f) = \bigcup_{i=1}^k \Omega_i$, where $\Omega_i$ is compact and transitive for each $1\le i \le k$, and $\Omega_i \cap \Omega_j = \emptyset$ if $i\neq j$. For each $1\le i\le k$, the \emph{stable manifold} $W^s(\Omega_i)$ and \emph{unstable manifold} $W^u(\Omega_i)$ are defined as
\begin{align*}
	W^s(\Omega_i) &\= \bigl\{ x\in M : \lim_{n\to +\infty} d( f^n(x) , \Omega_i)=0 \bigr\}, \\ 
	W^u(\Omega_i) &\= \bigl\{ x\in M : \lim_{n\to +\infty} d( f^{-n}(x) , \Omega_i)=0 \bigr\}.
\end{align*}
For $1\le i, \, j \le k$, we write $\Omega_i \rightharpoonup \Omega_j$ if 
\begin{equation*}
	W^u(\Omega_i) \cap W^s(\Omega_j) \nsubseteq \Omega(f).
\end{equation*}
The diffeomorphism $f$ is said to have a \emph{cycle} if there exist $1\le i_1, \, \dots , \, i_m \le k$ such that
\begin{equation*}
	\Omega_{i_1} \rightharpoonup \Omega_{i_2} \rightharpoonup \dots \rightharpoonup \Omega_{i_m} \rightharpoonup \Omega_{i_1}.
\end{equation*}
We say that $f$ has the \emph{no-cycle property}
if it does not have any cycles.

Fix $r\in \N$ and $\alpha\in (0,1]$. Suppose that $M$ is a compact smooth manifold equipped with a Riemannian metric,
and recall that
$\cA^r(M)$ denotes the space
of those $C^r$ Axiom~A diffeomorphisms on $M$ with the no-cycle property.
We first prove that each $f\in \cA^r(M)$ is $\cA^r(M)$-stably hyperbolic, by means of the following three 
results (Lemmas~\ref{A(a)},~\ref{A(b)}, and~\ref{Mane lemma}).

\smallskip

Firstly, it is well known that an Axiom~A diffeomorphism
with the no-cycle condition 
satisfies the following enhanced version of condition (IS) in Definition~\ref{Fstablyhyperbolic}:

\begin{lemma}\label{A(a)}
		 Suppose $M$ is a compact smooth manifold equipped with a Riemannian metric, $r\in \N$, and $f\in \cA^r(M)$. Then $f$ satisfies condition (IS) in Definition~\ref{Fstablyhyperbolic} for $\cF \= \cA^r(M)$. Moreover, if $g\in \cA^r(M)$ is sufficiently close to $f$ then $h_g$ can be taken to be a homeomorphism.
\end{lemma}
\begin{proof}
	By \cite[Theorem~5.8]{Wen16}, for every $r\in \N$,
	an Axiom~A diffeomorphism $f\in \cA^r(M)$ satisfying the no-cycle condition has $C^r$ \emph{$\ve$-$\Omega$-stability}, i.e., for all $\ve>0$, there exists a $C^r$ neighbourhood $U\subseteq \operatorname{Diff}^r(M)$ such that for every $g\in U$, 
	the restriction
	$g|_{\Omega(g)}$ is topologically conjugate to $f|_{\Omega(f)}$, and the conjugacy is $\ve$-close to the identity.  
\end{proof}

Regarding condition (RHE) in Definition~\ref{Fstablyhyperbolic}, we have:

\begin{lemma}\label{A(b)}
	 Let $M$ be a compact smooth manifold equipped with a Riemannian metric $\Hseminorm{}{\,\cdot\,}$, and the induced distance function $d$. If $r\in \N$ and $f\in \cA^r(M)$, then $f$ satisfies condition (RHE) in  Definition~\ref{Fstablyhyperbolic} for $\cF\=\cA^r(M)$.
\end{lemma}

\begin{proof}
	See Appendix~\ref{A}.
\end{proof}

\begin{rem}
    The single-map version of Lemma~\ref{A(b)}, in which the
    constants $K$, $\delta$, $\lambda$ depend on $f$ alone, is
    the content of \cite[Proposition~6.4.16]{KH95}. 
    The essential new content of Lemma~\ref{A(b)} is that the
    constants  $K$, $\delta$, $\lambda$
of condition (RHE) in Definition~\ref{Fstablyhyperbolic}
can be chosen \emph{uniformly} over a neighbourhood
    of $f$ in $\cA^r(M)$.
\end{rem}

Regarding condition (ML) in Definition~\ref{Fstablyhyperbolic}, we have:

\begin{lemma}\label{Mane lemma}
	Suppose $M$ is a compact smooth manifold equipped with a Riemannian metric. Every map in $\cA^1(M)$ satisfies condition (ML) in Definition~\ref{Fstablyhyperbolic}.
\end{lemma}

\begin{proof}
	Fix $f\in \cA^1(M)$ and $\alpha\in (0,1]$. By \cite[Theorem~5.3]{Wen16}, the nonwandering set $\Omega(f)$ is \emph{locally maximal}, in the sense that there exists a compact neighbourhood $K$ of $\Omega(f)$ such that $\Omega(f) = \bigcap_{n\in \Z} f^n(K)$. In the Lipschitz case $\alpha=1$, the lemma follows immediately from \cite[Theorem~1.3]{STY24}, while if $\alpha\in(0,1)$ we can proceed by adapting
    arguments in \cite{STY24}. More specifically, using the same arguments as in \cite[Section~2]{STY24} (specifically, Step~5 in the proof there), we obtain a modified form of \cite[Corollary~1.6]{STY24} with \cite[(1.5)]{STY24} replaced by the inequality
	\begin{equation*}
		\sum_{i=1}^n d\bigl(x_i, f^i(p)\bigr)^\alpha \le C \sum_{k=1}^{n} d(f(x_{k-1}), x_k)^\alpha,
	\end{equation*}
	where $C$ is a constant depending only on $f$ and $M$. By then using a modified discrete Lax--Oleinik operator (cf.~\cite[Definition~3.1]{STY24}), with the expression (3.1) of \cite{STY24} replaced by 
	\begin{equation*}
		T[u](x) \= \inf_{x'\in K} \{ u(x') + \phi(x') - \overline{\phi}_\Lambda + C d(f(x'),x)^\alpha  \} \quad \text{ for } x\in K, 
	\end{equation*}
	the result then follows 
    using arguments analogous to those in \cite[Sections~3 and~4]{STY24}.
\end{proof}

As a corollary of \cite[Theorems~1.1 and~1.2]{HLMXZ25}, if $f\: M \to M$ is an Axiom~A diffeomorphism, then $(f,\cP)$ satisfies the TPO property. More precisely, we have:
\begin{prop}[Individual TPO for Axiom~A diffeomorphisms]\label{TPO AxiomA}
	Suppose $M$ is a compact smooth manifold equipped with a Riemannian metric, $\alpha\in (0,1]$, and $\cP$ is either $\Holder{\alpha}(M,\R)$ or $C^1(M,\R)$. For every Axiom~A diffeomorphism $f$ on $M$, the locking set $\Lock(f,\cP)$ is an open dense subset of $\cP$.
\end{prop}

\begin{proof}
	Let $f \: M \to M$ be an Axiom~A diffeomorphism. Clearly,  $\Lock(f,\cP)$ is an open subset of $\cP$,  so it suffices to prove that $\Lock(f,\cP)$ is dense in $\cP$.
    
    We consider the dynamical system $(\Omega(f), f|_{\Omega(f)} )$. By \cite[Propositions~6.4.15 \&~6.4.16]{KH95}, and Lemma~\ref{Mane lemma}, $(\Omega(f), f|_{\Omega(f)} )$ satisfies the (ACP), (EI), and (NLP) properties of \cite{HLMXZ25}. Note that for every $\phi\in \cP$, $\Mmax(M,f,\phi) = \Mmax (\Omega(f), f|_{\Omega(f)}, \phi|_{\Omega(f)} )$. 
	
	If $\cP = C^1(M,\R)$, then by \cite[Theorem~1.2]{HLMXZ25}, $\Lock\bigl(f,C^1(M,\R)\bigr)$ is dense in $C^1(M,\R)$. 
	
	Assume that $\cP = \Holder{\alpha}(M,\R)$. Fix $\ve>0$ and $\phi\in \Holder{\alpha}(M,\R)$.
	By \cite[Theorem~1.1]{HLMXZ25}, there exists $\psi\in \Holder{\alpha}(\Omega(f))$ with $\Hnorm{\alpha,\Omega(f)}{\psi} < \ve$, $\delta>0$, and a periodic orbit $\cO$ of $f$ such that if $\xi\in \Holder{\alpha}(\Omega(f))$ with $\Hnorm{\alpha,\Omega(f)}{\xi}<\delta$ then $\Mmax (\Omega(f), f|_{\Omega(f)}, \phi|_{\Omega(f)} + \psi + \xi ) = \{\mu_{\cO}\}$. 
    By \cite[Theorem~1.33]{Wea18}, there exists $\tpsi\in \Holder{\alpha}(M,\R)$ with $\Hnormbig{\alpha,M}{\tpsi} \le \ve$ and $\tpsi|_{\Omega(f)} = \psi$. For every $\zeta\in \Holder{\alpha}(M,\R)$ with $\Hnorm{\alpha,M}{\zeta}<\delta$, then $\Hnorm{\alpha,\Omega(f)}{\zeta|_{\Omega(f)}}<\delta$. So we have 
	\begin{equation*}
		\Mmax \bigl(M,f,\phi+\tpsi+\zeta \bigr) 
        = \Mmax  \bigl(\Omega(f), f|_{\Omega(f)}, (\phi + \psi +\zeta)|_{\Omega(f)} \bigr) 
        = \{\mu_{\cO}\}.
	\end{equation*}
	Hence, $\phi+\tpsi \in \Lock\bigl(f,\Holder{\alpha}(M,\R)\bigr)$, which implies $\Lock\bigl(f,\Holder{\alpha}(M,\R)\bigr)$ is dense in $\Holder{\alpha}(M,\R)$. 
\end{proof}

\begin{proof}[\bf Proof of Theorem~\ref{JTPO AxiomA}]
    Suppose $r\in \N$, and $\cP$ is either $C^1(M,\R)$ or $\Holder{\alpha}(M,\R)$.
    By Proposition~\ref{TPO AxiomA}, $\Lock(\cA^r(M) , \cP)$ is dense in $\cA^r(M) \times \cP$.
    So it suffices to prove that $\Lock(\cA^r(M) , \cP)$ is open.

Note that every $\phi \in \cP$ that appears in a pair 
    $(f,\phi)\in \Lock(\cA^r(M),\cP)$ must be nonconstant. 
    Indeed, if $\phi$ is a constant function, then every $f$-invariant measure 
    attains the maximum ergodic average $\mpe(f,\phi)$, so 
    $\Mmax(f,\phi) = \cM(M,f)$; in particular the maximizing measure 
    is not unique, contradicting the PO property. 
    Hence no constant function $\phi$ can belong to $\Lock(\cA^r(M),\cP)$.
    
    Now fix $(f,\phi) \in \Lock(\cA^r(M) , \cP)$; by the above observation, 
    $\phi$ is nonconstant. 
    By Lemmas~\ref{A(a)},~\ref{A(b)}, and~\ref{Mane lemma}, 
    $f$ is $\cA^r(M)$-stably hyperbolic. 
    Since $\phi$ is not a constant, Theorem~\ref{openness} implies 
    that $(f,\phi)$ is in the interior of $\Lock(\cA^r(M), \cP)$. 
    Therefore, $\Lock(\cA^r(M) , \cP)$ is an open subset of $\cA^r(M) \times \cP$.
\end{proof}

\begin{rem}
    Lemma~\ref{A(a)} provides the stronger conclusion that $h_g$ can 
    be taken to be a homeomorphism for $g$ sufficiently close to $f$, 
    but in fact this additional property is not 
needed
for the proof of Theorem~\ref{JTPO AxiomA} itself.
    The abstract framework of Section~\ref{joint perturbation} requires 
    only the intertwining properties of $h_g$ and $i_g$ from 
    condition (IS) in Definition~\ref{Fstablyhyperbolic}, and these alone suffice to establish that 
    $\{(f,\phi)\in\Lock(\cA^r(M),\cP):\phi\text{ is nonconstant}\}$ 
    is open and dense in $\cA^r(M)\times\cP$.

    In all applications treated in this paper, we prove the existence of a homeomorphism $h_g$, so that condition (IS) becomes automatic. We nevertheless state (IS) in the above abstract form, since for other natural choices of $\cF$ such a normalisation may not be available; in that context the full generality of (IS) may become relevant.
\end{rem}

\section{Rational maps on the Riemann sphere: proof of Theorems~\ref{hrational} and~\ref{realqua}}\label{sec.hr}

Let $\Hseminorm{}{\,\cdot\,}$ be a conformal metric on $\widehat{\C}$ with induced distance function $d$. 
Let $f \: \widehat{\C} \to \widehat{\C}$ be a rational map with $\deg f\ge 2$, 
and denote the Julia set of $f$ by $J(f)$. 
An $f$-periodic point $x$ of period $p\in \N$,
and its orbit, are called \emph{attracting} (resp.~\emph{repelling}) if $\abs{(f^p)'(x)} <1$ (resp.~$\abs{(f^p)'(x)} >1$), and \emph{hyperbolic} if $x$ is either attracting or repelling. 
Let $AP(f)$ denote the set of all attracting periodic points of $f$.
For an attracting periodic orbit $\cO$, the \emph{attracting basin} of $\cO$ is defined as $B_f(\cO) \= \bigl\{x\in \widehat{\C}: \lim_{n\to +\infty} d(f^n(x),\cO)=0\bigr\}$.

The following lemma is a standard result:

\begin{lemma}\label{basicp.rational}
	If $f\: \widehat{\C}\to \widehat{\C}$ is a hyperbolic rational map with $\deg f \ge 2$, then its nonwandering set $\Omega(f)$ is equal to the disjoint union of its Julia set and its finitely many attracting periodic orbits, and also equal to the closure of the set of $f$-periodic points. 
\end{lemma}

\begin{proof} 
	By \cite[Theorem~19.1]{Mi06}, every orbit in the Fatou set of $f$ converges to an attracting periodic orbit, so $\Omega(f) \subseteq J(f) \cup AP(f)$.
    On the other hand, by \cite[Theorem~14.1]{Mi06}, the set of repelling periodic points is dense in $J(f)$,
    so the set of $f$-periodic points is dense in $J(f)\cup AP(f)$,
    and every periodic point is nonwandering, so $J(f) \cup AP(f) \subseteq \Omega(f)$. 
    Since $f|_{J(f)}$ is expanding, $J(f) \cap AP(f) = \emptyset$, so indeed 
    $\Omega(f)$ is the disjoint union of $J(f)$ and $AP(f)$, and the $f$-periodic points are dense in $\Omega(f)$.
     The number of attracting periodic points is finite
     (see e.g.~\cite[Theorem~8.2]{Mi06}), so the result is proved.
\end{proof}

We next prove,
by means of
the following Lemmas~\ref{A(a)hr},~\ref{disexpanding},~\ref{A(b)hr}, and~\ref{A(c)hr},
that every hyperbolic rational map of degree $m \ge 2$ is $\cH\cR^m$-stably hyperbolic.
We first prove the following enhanced version of condition (IS) in Definition~\ref{Fstablyhyperbolic}:

\begin{lemma}\label{A(a)hr}
	Suppose $f\in \cH\cR^m$ for some $m\geq 2$. For each $\ve>0$, and for each $g\in \cH\cR^m$ sufficiently close to $f$, there exists a homeomorphism $h_g \: \Omega(f) \to \Omega(g)$ with $d_\infty(h_g,\id)<\ve$ and $ h_g \circ f = g \circ h_g$. Moreover, $h_g(J(f)) = J(g)$ and $h_g(AP(f)) = AP(g)$.
\end{lemma}

\begin{proof}
	Note that $\cH\cR^m$ is an open set (\cite[p.~205]{Mi06}). So rational maps of degree $m$ in a neighbourhood of $f$ are hyperbolic, so each of their periodic points is hyperbolic. For any rational map $g\: \widehat{\C} \to \widehat{\C}$, denote the set of $g$-periodic points by $P(g)$. By Lemma~\ref{basicp.rational}, $\Omega(g)$ is equal to the closure of $P(g)$. Thus, by \cite[Main~Lemma]{Ly83} (see also \cite{MSS83} or \cite{MS98}), for $g\in \cH\cR^m$ sufficiently close to $f$, there exists a topological conjugacy $h_g \: \Omega(f)\to \Omega(g)$. Furthermore (see \cite[Main Lemma]{Ly83} and \cite[Section~2.1]{Ly86}), $h_g$ converges to the identity uniformly as $g$ converges to $f$. If $g \in \cH \cR^m$ and the conjugacy $h_g$ exists, by \cite[Corollary~4.14]{Mi06}, $J(f)$ and $J(g)$ have no isolated points. But every point in $AP(f)$ is isolated in $\Omega(f)$, and every point in $AP(g)$ is isolated in $\Omega(g)$. Now $h_g \: \Omega(f) \to \Omega(g)$ is a homeomorphism, so $h_g(J(f))=J(g)$ and $h_g(AP(f))= AP(g)$, and the result is proved.
\end{proof}

Recall that if $(X,d)$ is a compact metric space and $T \: X \to X$, then $T$ is \emph{distance-expanding} if there exist constants $\eta>0$ and $\theta>1$ such that for any two distinct points $x, \, y\in X$, if $d(x,y)<\eta$ then $d(T(x),T(y)) \ge \theta d(x,y)$. We shall refer to $(\eta,\theta)$ as a pair of \emph{expanding constants} for $T$. 

\begin{lemma}\label{disexpanding}
Suppose $f\in \cH\cR^m$ for some $m\geq 2$.
	 Equip $\widehat{\C}$ with a conformal metric $\Hseminorm{}{\,\cdot\,}$ satisfying (\ref{expanding}) for some $\lambda>1$, and let $d$ denote its induced distance. Then the restriction $f|_{J(f)}$ is an open mapping, and is distance-expanding.
     Moreover, the expanding constants can be chosen uniformly over all 
rational maps in some neighbourhood of $f$.
     More precisely, there exist a neighbourhood $N \subseteq \cH\cR^m$ of $f$, and constants $\eta>0, \, \theta>1$, such that for all $g\in N$ and distinct points $x, \,y \in J(g)$, if $d(x,y) < \eta$, then $d(g(x),g(y)) \ge \theta d(x,y)$.
\end{lemma}

\begin{proof}
	See Appendix~\ref{A}.
\end{proof}

Lemma~\ref{disexpanding} allows us to verify condition (RHE) in Definition~\ref{Fstablyhyperbolic}:

\begin{lemma}\label{A(b)hr}
If $f\in \cH\cR^m$ for some $m\geq 2$,
	then $f$ satisfies condition (RHE) in Definition~\ref{Fstablyhyperbolic} for $\cF\=\cH\cR^m$.
\end{lemma}
\begin{proof}
	Let $\eta>0, \, \theta>1$ be the  constants, and $N$ the neighbourhood, that are guaranteed by Lemma~\ref{disexpanding}. Now $J(f)$ and $AP(f)$ are disjoint compact sets, and $AP(f)$ is finite, so $d(J(f), AP(f))>0$, and $\Delta(AP(f))$ is well defined. Define
	\begin{equation}
		\delta \= \min \bigl\{ 3^{-1} d(J(f), AP(f)), \, 3^{-1}\Delta(AP(f)), \, \eta \bigr\}. \label{delta} 
	\end{equation}
	By Lemma~\ref{A(a)hr}, there exists a neighbourhood $U_0$ of $f$ such that if $g\in U_0$ then $h_g$ exists and $d_\infty(h_g, \id) < \delta$. Thus, by (\ref{delta}), for every $g\in U_0$, we have $d(J(g),AP(g)) > d(J(f),AP(f))-2\delta \ge \delta$ and $\Delta(AP(g)) \ge \Delta(AP(f))-2\delta \ge \delta$. This means that if $g\in U_0$ and $x, \,y \in \Omega(g)$ with $d(x,y)<\delta$, then $x, \, y \in J(g)$. Define $U \= N \cap U_0$. Fix $g\in U$, $n\in \N$, and $x, \, y\in \Omega(g)$ with $\max \bigl\{ d\bigl(g^i(x),g^i(y)\bigr) : 0\leq i\leq n \bigr\} <\delta$. Then $\bigl\{ x, \, y, \, g^1(x), \, g^1(y), \, \dots, \, g^n(x), \,g^n(y) \bigr\} \subseteq J(g)$. Moreover, by Lemma~\ref{disexpanding}, we have
	$d\bigl(g^i(x),g^i(y)\bigr) \le \theta^{i-n} d(g^n(x), g^n(y))$,
	which in particular implies that condition (RHE) in Definition~\ref{Fstablyhyperbolic} holds with $K\=1$ and $\lambda\=\theta$.
\end{proof}

We now verify condition (ML) in Definition~\ref{Fstablyhyperbolic}:

\begin{lemma}\label{A(c)hr}
If $f\in \cH\cR^m$ for some $m\geq 2$,
	then $f$ satisfies condition (ML) in Definition~\ref{Fstablyhyperbolic} for $\cF\=\cH\cR^m$.
\end{lemma}

\begin{proof}
	The strategy of proof will be that for all $\phi \in \Holder{\alpha}\bigl(\widehat{\C}\bigr)$, the function  $u=u_\phi$ can be constructed separately on $J(f)$ and $AP(f)$.

	First we restrict $f$ to the Julia set $J(f)$. 
	By Lemma~\ref{disexpanding}, $f|_{J(f)}$ is open and distance-expanding, so
	by \cite[Proposition~3.6]{LS26} there exists $L_J>0$ such that 
    for all $\phi\in \Holder{\alpha}\bigl(\widehat{\C} \bigr)$,  
    there exists $v_\phi \in \Holder{\alpha}(J(f))$ satisfying $\Hseminorm{\alpha,J(f)}{v_\phi}\le L_J \Hseminorm{\alpha,J(f)}{\phi}$ and $\overline{\phi}+v_\phi-v_\phi\circ f \le 0$ on $J(f)$.
	By adjusting $v_\phi$ by an additive constant, we may assume that $v_\phi(x)=0$ for some $x\in J(f)$, which implies that
    \begin{equation}\label{unormbound}
		\Hnorm{\infty,J(f)}{v_\phi}  \le \bigl(\diam \widehat{\C}\bigr)^\alpha \Hseminorm{\alpha,J(f)}{v_\phi} \le L_J \bigl(\diam \widehat{\C}\bigr)^\alpha \Hseminorm{\alpha,J(f)}{\phi} \le L_J \bigl(\diam \widehat{\C}\bigr)^\alpha \Hseminorm{\alpha,\widehat{\C}}{\phi}.
	\end{equation}
		
	Next we restrict $f$ to $AP(f)$, 
    and define a function $w_\phi \: AP(f) \to \R$ as follows.
    Let $\cO_1,\, \dots,\, \cO_k$ be the finitely many attracting $f$-periodic orbits.
    Define $p_i \= \card \cO_i$ for each $1\le i \le k$, and $p \= \max_{1\le i\le k} p_i$. 
    For each integer $1\le i \le k$, choose a point $x\in \cO_i$ and define 
    \begin{equation*}
        w_\phi \bigl( f^j(x) \bigr) \= S_j^f \overline{\phi}(x) \quad \text{ for each }0\le j \le p_i-1 . 
    \end{equation*}
    By our construction, if $0\le j \le p_i-2$, then $\overline{\phi}\bigl(f^j(x)\bigr) + w_\phi \bigl(f^j(x) \bigr) - w_\phi \bigl(f^{j+1}(x) \bigr) =0$, and $\overline{\phi} \bigl(f^{p_i-1}(x) \bigr) + w_\phi \bigl(f^{p_i-1}(x) \bigr) - w_\phi \bigl(f^{p_i}(x) \bigr) = S_{p_i}^f \overline{\phi}(x) \le 0$, where the last inequality follows from the fact that $\mu_{\cO_{i}}$ is invariant. 
    
    So we conclude that $\overline{\phi} + w_\phi - w_\phi \circ f \le 0$ on $AP(f)$. 
    Moreover, by our construction, we have $\Hnorm{\infty,AP(f)}{w_\phi} \le p\Hnorm{\infty, \widehat{\C}}{\overline{\phi}}$. 
    Note that $\mpe\bigl( f,\overline{\phi} \bigr)=0$, so $\overline{\phi}(x)\ge 0$ at some $x\in \widehat{\C}$ and $\overline{\phi}(y)\le 0$ for some $y\in \widehat{\C}$, which implies that $\Hnorm{\infty,\widehat{\C}}{\overline{\phi}}\le \bigl(\diam \widehat{\C}\bigr)^\alpha \Hseminorm{\alpha,\widehat{\C}}{\phi}$. Therefore, 
	\begin{equation}\label{unormbound2}
		\Hnorm{\infty,AP(f)}{w_\phi} \le p \Hnorm{\infty,\widehat{\C}}{\overline{\phi}} \le p \bigl(\diam \widehat{\C}\bigr)^\alpha \Hseminorm{\alpha,\widehat{\C}}{\phi}.
	\end{equation} 
    
    Let $\tau \= \min \{ d(J(f), AP(f)) , \, \Delta(AP(f))\}$ and then define the function $u_\phi \: \Omega(f) \to \R$ by $u_\phi \= v_\phi$ on $J(f)$ and $u_\phi \= w_\phi$ on $AP(f)$. Fix distinct points $x, \, y\in \Omega(f)$. When $\{x, \,y\} \subseteq AP(f)$, combining the fact that $d(x,y)\ge \tau$ and (\ref{unormbound2}) gives  $\abs{ u_\phi(x) -u_\phi(y) }/d(x,y)^\alpha \le 2p\tau^{-\alpha} \bigl(\diam \widehat{\C}\bigr)^\alpha \Hseminorm{\alpha,\widehat{\C}}{\phi}$.
    When either $x\in AP(f)$ and $y\in J(f)$, or 
    $x\in J(f)$ and $y\in AP(f)$, combining the fact that $d(x,y)\ge \tau$, (\ref{unormbound}), and (\ref{unormbound2})
    gives $\abs{ u_\phi(x) -u_\phi(y) }/d(x,y)^\alpha \le (p+L_J)\tau^{-\alpha} \bigl(\diam \widehat{\C}\bigr)^\alpha \Hseminorm{\alpha,\widehat{\C}}{\phi}$.
    When $\{ x, \,y \}\subseteq J(f)$, then $\abs{ u_\phi(x) -u_\phi(y) }/d(x,y)^\alpha \le L_J \Hseminorm{\alpha,\widehat{\C}}{\phi}$. From these bounds 
    we see that $f$ satisfies condition (ML) in Definition~\ref{Fstablyhyperbolic}
    with $L \= \max \bigl\{ L_J, \, (2p+L_J) \tau^{-\alpha}\bigl(\diam \widehat{\C}\bigr)^\alpha \bigr\}$. 
\end{proof}

The final ingredient needed for the proofs of Theorems~\ref{hrational} and~\ref{realqua} is the Individual TPO for hyperbolic rational maps. 

\begin{lemma}\label{indTPOhr}
    Let $\alpha\in(0,1]$, $f\in \cH\cR^m$ for some $m\geq 2$,
    and $\cP$ denote either $\Holder{\alpha}\bigl( \widehat{\C} \bigr)$ or $C^1 \bigl(\widehat{\C}\bigr)$. Then $\Lock(f,\cP)$ is an open dense subset of $\cP$.
\end{lemma} 

\begin{proof}
	Provided $(\Omega(f), f)$ satisfies
    conditions (ACP), (EI), and (NLP) from \cite{HLMXZ25}, the result follows by an argument analogous to the one used to prove Theorem~\ref{TPO AxiomA}. 
    
    To check that indeed $(\Omega(f),f)$ does satisfy these three conditions, we note that
	(EI) and (NLP) follow from Lemmas~\ref{A(b)hr} and~\ref{A(c)hr}, respectively. For the condition (ACP), if 
    $$0<\ve < \min \{d(J(f), AP(f)), \, \Delta(AP(f))\},$$ then since $f(J(f)) = J(f)$ and $f(AP(f))= AP(f)$, every $\ve$-pseudo orbit is either contained entirely in $J(f)$, or contained 
entirely within one of the attracting periodic orbits that makes up $AP(f)$. 
For pseudo orbits contained in a periodic orbit, property~(ACP) is 
trivially satisfied.
For pseudo orbits contained in $J(f)$, 
property~(ACP) follows from \cite[Proposition~4.2.3]{PU10}, since 
$(J(f),f)$ is open and distance-expanding by Lemma~\ref{disexpanding}.
\end{proof}

Finally, we can prove Theorems~\ref{hrational} and~\ref{realqua}.

\begin{proof}[\bf Proof of Theorem~\ref{hrational}]
	This follows by combining Lemmas~\ref{A(a)hr},~\ref{A(b)hr},~\ref{A(c)hr},~\ref{indTPOhr}, and Theorem~\ref{openness}, and using an argument analogous to the one used to prove Theorem~\ref{JTPO AxiomA}. 
\end{proof}

\begin{proof}[\bf Proof of Theorem~\ref{realqua}]
	Let $\alpha\in (0,1]$, and let $\cP$ denote either $C^1\bigl( \widehat{\C} \bigr)$ or $\Holder{\alpha}\bigl( \widehat{\C} \bigr)$. Let $\cH \cQ$ denote the set of hyperbolic real quadratic polynomials. Note that $\cQ$ and $\cH\cR^2$ are both equipped with the standard topology on the parameter space. Since $\cH \cQ = \cH\cR^2 \cap \cQ$ and $\cH\cR^2$ is open (see \cite[p.~205]{Mi06}), it follows that $\cH \cQ$ is open.
	By \cite{GS97} and \cite[p.~4]{Ly97}, $\cH \cQ$ is a dense subset of $\cQ$, so it suffices to prove the Joint TPO property for $\cH\cQ \times \cP$.
	
	By Lemmas~\ref{A(a)hr},~\ref{A(b)hr}, and~\ref{A(c)hr}, every $f \in \cH\cQ \subseteq \cH\cR^2$ is stably hyperbolic (i.e.,~satisfies Definition~\ref{Fstablyhyperbolic}) for $\cF \= \cH\cR^2$, and hence also for $\cF \= \cH\cQ$. By this and Theorem~\ref{openness}, the set $\cL_0 \= \{(f,\phi) \in \Lock(\cH\cQ,\cP) : \phi \text{ is not a constant}\}$ is an open subset of $\cH\cQ \times \cP$. Since $\operatorname{Const} \= \cH\cQ \times \{ \phi\in \cP : \phi \text{ is a constant}\}$ is nowhere dense in $\cH\cQ \times \cP$, by Lemma~\ref{indTPOhr}, $\cL_0 = \Lock(\cH\cQ,\cP) \smallsetminus \operatorname{Const}$ is dense in $\cH\cQ \times \cP$, so the result is proved.
\end{proof}

\section{$C^r$ one-dimensional maps and the logistic family: proof of Theorems~\ref{1dim} and~\ref{realqua01}}\label{sec.r1d}
We first recall some basic definitions and properties in one-dimensional dynamics (see \cite[Chapter~3]{dMvS93} for more background). We follow the conventions in \cite{dMvS93} and \cite{KSS07}.

Let $r\in \N$, $M$ be a compact interval or a circle equipped with the Euclidean metric, and $f\in C^r(M,M)$. 
A subset $K\subseteq M$ is said to be a \emph{hyperbolic} set for $f$ if $f(K)\subseteq K$ and there exist constants $C>0$ and $\theta>1$ such that $\abs{(f^n)'(x)} \ge C \theta^n$ for all $x\in K$ and $n\in \N$. 
It can be seen from \cite[p.~226]{dMvS93} that there exist a Riemannian metric $\Hseminorm{f}{\, \cdot\, }$ and a constant $\lambda>1$ such that 
\begin{equation}\label{adexpand}
	\Hseminorm{f}{\mathrm{D}f} > \lambda \quad\text{ on } K.
\end{equation}
Since $M$ is compact, any two Riemannian metrics are mutually equivalent, so the topologies of $C^r(M,\R)$ and $C^r(M,M)$ are independent of the choice of metric.

An $f$-periodic point $x$, of period $p$,
is said to be \emph{attracting} if $\abs{(f^p)'(x)} <1$,  and the orbit of an attracting periodic point is called an attracting periodic orbit.
As in the case of rational maps, if $x\in M$ is an attracting periodic point, then so is
every $y\in \cO^f(x)$.  
The map $f$ is said to be \emph{hyperbolic} (or, as alternative terminology, to satisfy \emph{Axiom~A}) if $f$ has a hyperbolic set $K(f)$
and finitely many attracting periodic points such that the forward orbit under $f$ at every point in $M\smallsetminus K(f)$ converges to an attracting periodic orbit (cf.~\cite[p.~221]{dMvS93}). 

Kozlovski,
Shen \& van Strien \cite[Theorem~2]{KSS07} showed that hyperbolic maps are dense in $C^r(M,M)$, for every $r\in \N$. We will use their result, together with our theorems from Section~\ref{joint perturbation}, 
to prove Theorem~\ref{1dim}.

We first restrict our arguments to the
maps in $C^r(M,M)$ that preserve the endpoints, in
the case that $M$ is an interval. More precisely, we denote $C_0^r(M,M) \= C^r(M,M)$ when $M$ is the circle, 
and $C^r_0(M,M) \= \{ f\in C^r(M,M) : f(\partial M) \subseteq \partial M\}$ when $M$ is a compact interval. 
Let $HC^r_0(M,M)$ denote the space of hyperbolic maps
in $C^r_0(M,M)$.

For such maps, we have:

\begin{prop}\label{endpointsp}
	Let $r\in \N$, $\alpha\in (0,1]$, $M$ be a compact interval or a circle, and $\cP$ denote $\Holder{\alpha}(M,\R)$ or $\cP \= C^1(M,\R)$. 
    The space $HC^r_0(M,M)$ 
    is an open subset of $C^r_0(M,M)$.
    If $f\in HC_0^r(M,M)$, then $f$ is $C_0^r(M,M)$-stably hyperbolic in the sense of Remark~\ref{nonwander}~(i), and if $\phi\in \Lock(f, \cP)$ is not a constant, then $(f,\phi)$ is in the interior of $\Lock(C^r_0(M,M), \cP)$.
\end{prop}

\begin{proof}
       Denote by $HC^r_0(M,M)$ the set of hyperbolic maps in $C^r_0(M,M)$. 
       It follows from \cite[Theorem~2.4, Chapter~3]{dMvS93} that $HC^r_0(M,M)$ is an open subset of $C^r_0(M,M)$. 

       For each $g \in HC^r_0(M,M)$, define $L(g)$ to be the union of $K(g)$ and the set of all attracting periodic points of $g$. 
       By definition, we have $g(L(g)) \subseteq L(g)$. 
       Since the forward orbit of every point outside $L(g)$ converges to an attracting periodic orbit, it follows that $\Omega(g) \subseteq L(g)$. 
       Consequently, every $g$-invariant measure is supported on $L(g)$.

       Replacing $\Omega(g)$ by $L(g)$ in Definition~\ref{Fstablyhyperbolic}, it follows from Theorem~\ref{openness}, Remark~\ref{nonwander}~(i), and the openness of $HC^r_0(M,M)$ in $C^r_0(M,M)$, that it suffices to verify that every hyperbolic map $g \in C^r_0(M,M)$ is $HC^r_0(M,M)$-stably hyperbolic.

      Condition~(IS) in Definition~\ref{Fstablyhyperbolic} for each $f \in HC^r_0(M,M)$ follows from \cite[Theorem~2.4, Chapter~3]{dMvS93}. 
      Indeed, this result states that if $f \in C^r_0(M,M)$ is hyperbolic, then for every $g \in C^r_0(M,M)$ sufficiently close to $f$, there exists a conjugacy $h_g \: K(f) \to K(g)$ that converges uniformly to the identity as $g $ tends to $ f$. 
      Moreover, an inspection of the proof of \cite[Theorem~2.4, Chapter~3]{dMvS93} shows that $h_g$ can be extended to a conjugacy $H_g \: L(f) \to L(g)$ that also converges uniformly to the identity as $g$ tends to $f$.

      For conditions (RHE) and (ML) in Definition~\ref{Fstablyhyperbolic}, given $f \in HC^r_0(M,M)$, consider the metric $\abs{\,\cdot\,}_f$ and a constant $\lambda > 1$ such that $\Hseminorm{f}{\mathrm{D}f} > \lambda$ on $K(f)$. 
      The verification then proceeds exactly as in the proofs of Lemmas~\ref{disexpanding}, \ref{A(b)hr}, and~\ref{A(c)hr} if we replace $J(f)$ by $K(f)$.
      Hence for each $f\in HC_0^r(M,M)$, $f$ is $HC^r_0(M,M)$-stably hyperbolic. By Theorem~\ref{openness}, if $\phi\in \Lock(f, \cP)$ is not a constant, then $(f,\phi)$ is in the interior of $\Lock(C^r_0(M,M), \cP)$.
      
\end{proof}

\begin{lemma}\label{inTPO1dim}
	Let $r\in \N$, $\alpha\in (0,1]$, $M$ be a circle or a compact interval, $f\in C^r_0(M,M)$ be a hyperbolic map, and $\cP$ denote $C^1(M,\R)$ or $\Holder{\alpha}(M,\R)$. Then $\Lock(f, \cP)$ is an open dense subset of $\cP$. 
\end{lemma}

\begin{proof}
    Arguing as in Lemma~\ref{disexpanding}, by equipping $M$ with the metric $\Hseminorm{f}{\, \cdot \, }$, it can be shown that $f|_{K(f)}$ is distance-expanding. 
    Moreover, $\Hseminorm{f}{\mathrm{D}f}$ does not vanish on a neighbourhood $V$ of $K(f)$, which implies that $f|_V$ is an open map. 
    Since the forward orbit of every point outside $K(f)$ converges to an attracting periodic orbit, it follows that $f^{-1}(K(f)) = K(f)$. 
    Hence, $f|_{K(f)}$ is also an open map. 
    
    Using an argument analogous to that in Lemma~\ref{indTPOhr}, one verifies that $f$ satisfies the (ACP) property in \cite{HLMXZ25}. 
    Furthermore, conditions (RHE) and (ML) in Definition~\ref{Fstablyhyperbolic} established in Proposition~\ref{endpointsp} imply that $f$ also satisfies the properties (EI) and (NLP) of \cite{HLMXZ25}. 
    The conclusion then follows by an argument analogous to that of Proposition~\ref{TPO AxiomA}.
\end{proof}

Applying Proposition~\ref{endpointsp} and Lemma~\ref{inTPO1dim} to Theorem~\ref{1dim} requires a transition from $C^r(M,M)$ to the subclass $C^r_0(M,M)$. 
We facilitate this by developing a method to extend general interval maps to endpoint-preserving ones without altering their fundamental dynamical properties. 
The details of this construction follow.

\begin{lemma}\label{extension}
	Let $M=[a,b]$ be an interval, $r\in \N$, and $f\in C^r(M,M)$. 
    Assume that $\{f(a), \, f(b)\} \subseteq (a,b)$ and $f'(a) \neq 0 \neq f'(b)$. 
    Denote $\tau \= \frac{1}{2}\min\{f(a)-a,\, b-f(b)\}>0$. Then there exists $a_0<a$, $b_0>b$, and $F\in C^r_0([a_0,b_0],[a_0,b_0])$ satisfying the following properties:
	\begin{enumerate}[label=\rm{(\roman*)}]
		\smallskip
		\item $\{F(a_0),\, F(b_0) \} \subseteq \{a_0, \, b_0\}$ and $F|_{M} = f$.
		
		\smallskip
		\item For every $x\in [a_0,a] \cup [b,b_0]$, $F'(x) \neq 0$.
			
		\smallskip
		\item There exist $\theta>1$, $a_1 \in (a_0,a)$, and $b_1 \in (b,b_0)$ such that $F([a_0,a_1) \cup (b_1,b_0]) \cap [a+\tau,b-\tau] =\emptyset$ and $F( [a_1,a]\cup [b,b_1] ) \subseteq [a+\tau,b-\tau]$, and $\abs{F'} > \theta $ on $[a_0,a_1) \cup (b_1,b_0]$.
	\end{enumerate}
\end{lemma}

\begin{lemma}\label{exnormbound}
	Let the intervals $M=[a,b]$ and $M_0 = [a_0,b_0]$ be such that $a_0< a<b<b_0$. Let $r\in \N$. Then there exists a constant $D_r\ge 1$ such that for every $f\in C^r(M,\R)$, there exists a $C^r$ extension $F$ of $f$ to $M_0$ with $F(a_0)=F(b_0)=0$ and $\Hnorm{C^r,M_0}{F} \le D_r \Hnorm{C^r, M}{f}$.
\end{lemma}

The proofs of Lemmas~\ref{extension} and~\ref{exnormbound} rely on some auxiliary technical arguments unconnected to the dynamical ideas used in the main proof of Theorem~\ref{1dim}, so are relegated to Appendix~\ref{A}.

\begin{lemma}\label{extensionprop} 
    Let $M=[a,b]$ be an interval, $r\in \N$, and $f\in C^r(M,M)$, and
    let $M_0 \= [a_0,b_0]$ and $F$ be 
    as in Lemma~\ref{extension}.
    Then $F$ satisfies the following properties:
	\begin{enumerate}[label=\rm{(\roman*)}]
		\smallskip
		\item There exists $\kappa>0$ such that for each $H \in C^r(M_0,\R)$ with $H(a_0)=H(b_0) =0$ and $\Hnorm{C^r,M_0}{H}< \kappa$, we have $F+H\in C_0^r(M_0,M_0)$.

		\smallskip
		\item There exists a neighbourhood $W \subseteq C^r_0(M_0,M_0)$ of $F$ such that for every $G \in W$ with $G(M) \subseteq M$, every $G$-invariant measure is supported on $M \cup \{a_0, \, b_0\}$.
		
		\smallskip
		\item $F$ is hyperbolic.

		\smallskip
		\item If $W$ is the neighbourhood of $F$ as in (ii), then for every $G\in W$ with $G(M)\subseteq M$, if $\phi \in C(M_0)$ is such that $\phi(a_0) < \mpe(G,\phi)$ and $\phi(b_0)< \mpe(G,\phi)$, then every $(G,\phi)$-maximizing measure has support contained in $M$. 
	\end{enumerate}
\end{lemma}

\begin{proof}
	 Our proof will use the properties of $F$ from Lemma~\ref{extension}, where the constants $\tau$ and $\theta$ are as in that lemma.
	
	\smallskip
	(i) Define
    $        \kappa \= \min \{ \theta/2, \, a-a_0 , \, b_0-b \}$.
    Clearly, $\kappa>0$. Let $H\in C^r(M_0,\R)$ satisfy $H(a_0)=H(b_0)=0$ and $\Hnorm{C^r,M_0}{H}< \kappa$. 
    Set $\tF \= F+H$. Then 
    $        \bigl\{\tF(a_0),\, \tF(b_0)\bigr\} \subseteq \{a_0, \,b_0\}$.
    Lemma~\ref{extension}~(iii) implies that $a \le F \le b$ on $[a_1,b_1]$. 
    Hence
    \begin{equation*}
        a_0 \le a-\kappa < \tF < b+\kappa \le b_0\quad \text{ on } [a_1,b_1],
    \end{equation*} which yields $\tF([a_1,b_1]) \subseteq M_0$. 
    Moreover, Lemma~\ref{extension}~(iii) also gives $ \Absbig{\tF'}\geq \abs{F'}-\abs{H'}>\theta-\kappa\geq \theta/2> 0$ on $[a_0,a_1)$, so $\tF$ is monotonic. 
    Since $\tF(a_0)=F(a_0) \in M_0$ and $\tF(a_1) \in M \subseteq M_0$, then $\tF([a_0,a_1)) \subseteq M_0$. Similarly, $\tF((b_1,b_0]) \subseteq M_0$. 
    
    In conclusion, $\tF\in C^r(M_0,\R)$, 
    $\bigl\{\tF(a_0),\, \tF(b_0)\bigr\} \subseteq \{a_0, \,b_0\},\, \tF([a_0,b_0]) \subseteq [a_0,b_0]$,
    and therefore $\tF \in C^r_0(M_0,M_0)$.
	
	\smallskip
	(ii) Define 
    $\eta \= \min \bigl\{ \tau, \, 2^{-1}(\theta-1) \bigr\}\text{ and } W \= B_{C_0^r(M_0,M_0)} (F,\eta)$,
     where $B_{C_0^r(M_0,M_0)}(F,\eta)$ denotes the set of functions $G \in C_0^r(M_0,M_0)$ with $\Hnorm{C^r,M_0}{G-F} < \eta$. 
	
	Fix $G\in W$ with $G(M) \subseteq M$, and set $M_1 \= M \cup \{a_0, \, b_0\}$. 
   Since $G\in C_0^r(M_0,M_0)$ and $G(M)\subseteq M$, we have $G(M_1)\subseteq M_1$. 
To show that every $G$-invariant probability measure is supported on $M_1$, it suffices to show that for every $x \in M_0 \smallsetminus M_1 = (a_0,a) \cup (b,b_0)$, there exists $n \in \N$ such that $G^n(x) \in M_1$. Indeed, if this holds, the forward invariance $G(M_1) \subseteq M_1$ ensures that the forward orbit of $x$ never returns to the open set $M_0 \smallsetminus M_1$, meaning no point in $M_0 \smallsetminus M_1$ is recurrent. By the Poincar\'{e} recurrence theorem, any $G$-invariant measure must assign zero mass to $M_0 \smallsetminus M_1$, and therefore 
has support contained in $M_1$.

    We now prove this forward orbit property.
    First consider the case where $x\in (a_0,a) $. 
    Suppose for a contradiction that for every $n\in \N$, $G^n(x) \notin M_1$. 
    By Lemma~\ref{extension}~(iii), $F([a_1,a]\cup [b,b_1]) \subseteq [a+\tau,b-\tau]$. 
    Since $\Hnorm{\infty,M_0}{G-F}  <\eta \le \tau$, it follows that $G([a_1,a] \cup [b,b_1]) \subseteq M$. 
    Hence $G^n(x) \notin [a_1, b_1]$ for all $n\in \N$.
	
	Now since $\Hnorm{\infty,M_0}{F'-G'}< \eta \le 2^{-1}(\theta-1)$ and $\abs{F'}>\theta$ on $(a_0,a_1) \cup (b_1,b_0)$, again by Lemma~\ref{extension}~(iii), we obtain that
    \begin{equation*}
        \abs{G'}\geq  \abs{F'}-\Hnorm{\infty,M_0}{F'-G'} >2^{-1}(\theta+1)>1\quad \text{ on }(a_0,a_1) \cup (b_1, b_0).
    \end{equation*} 
    So $G$ is monotone on both $(a_0,a_1)$ and $(b_1,b_0)$. As we have shown that $G(a_1) , \, G(b_1) \in M$, we have
	\begin{equation}\label{inc}
		\begin{aligned}
			\text{ either }\quad G([a_0,a_1]) &\subseteq [a_0,b] \quad\text{ or }\quad G([a_0,a_1]) \subseteq [a,b_0], \\
			\text{ either }\quad G([b_1,b_0]) &\subseteq [a_0,b] \quad\text{ or }\quad G([b_1,b_0]) \subseteq [a,b_0].
		\end{aligned}
	\end{equation}
	Since $\abs{G'} > 2^{-1}(1+\theta)$ on $[a_0,x]$ and $G(x) \notin [a_1,b_1]$, the image $G([a_0,x])$ is a closed interval with 
    one endpoint in $\{a_0, b_0\}$ and length
    $\abs{ G([a_0,x]) } > 2^{-1}(1+\theta) (x-a_0)$.
    Thus, because $G([a_0, a_1])$ cannot cover both $(a_0, a_1)$ and $(b_1, b_0)$ simultaneously without violating \eqref{inc}, the fact that $G(x) \notin [a_1, b_1]$ forces $G([a_0,x])$ to be contained entirely in $[a_0,a_1)$ or entirely in $(b_1,b_0]$.
    In particular, $\abs{G'} > 2^{-1} (1+\theta)$ on all of $G([a_0,x])$. 
    Using (\ref{inc}) and the fact that $G^n(x) \notin [a_1,b_1]$ for all $n\in \N$, we may iterate the arguments above to obtain 
    \begin{equation*}
        \abs{ G^n([a_0,x]) } > 2^{-n} (1+\theta)^{n} (x-a_0) \quad\text{ for all }n\in \N.
    \end{equation*} 
    However, $\lim_{n\to +\infty} 2^{-n} (1+\theta)^{n} (x-a_0) = +\infty$, which is impossible since $G^n([a_0,x])\subseteq [a_0,b_0]$. 
    This contradiction proves that there exists $n\in \N$ such that $G^n(x)\in M_1$.
    
    The proof in the case that $x\in (b,b_0)$ is
    entirely analogous to the one above for the case that $x\in(a_0,a)$, and will be omitted. Therefore, (ii) is proved.
	
	\smallskip
	(iii) Set
    $K_0 \= \{a_0, \, b_0\} \cup  \bigcup_{i=0}^{+\infty} F^{-i}(K(f))$.
    We next show that $F$ is hyperbolic for $K(F) = K_0$.
	
	Since $F(\{a_0, \, b_0\}) \subseteq \{a_0, \, b_0\}$, by the definition of $K_0$ we have $F(K_0) \subseteq K_0$. 
    Fix $x\in [a_0,b_0] \smallsetminus K_0$. 
    
    If $x\in M$, then $x\in M \smallsetminus K(f)$. 
    By the hyperbolicity of $f$, the orbit $\bigl\{F^i(x)\bigr\}_{i\in \N}$ converges to an attracting periodic orbit of $f$, which is also an attracting periodic orbit of $F$. 
    
    If $x \notin M$, then $x\in (a_0,a)\cup (b,b_0)$. 
    By (ii), there exists $n\in \N$ such that $F^n(x) \in M$. 
    Since $x\notin K_0$, we have $x\notin F^{-n}(K(f))$, and hence $F^n(x) \in M \smallsetminus K(f)$. 
    Consequently, the orbit $\bigl\{F^{n+i}(x)\bigr\}_{i\in \N}$ converges to an attracting periodic orbit of $F$. 
    In conclusion, for every point $x\in M_0 \smallsetminus K_0$, the orbit $\bigl\{F^i(x)\bigr\}_{i\in \N}$ converges to an attracting periodic orbit of $F$.

	Next, we prove that $F$ is expanding on $K_0$. 
    By the hyperbolicity of $f$, there exist constants $C>0$ and $\lambda>1$ such that $\abs{ (f^m)'(y) } \ge C \lambda^m$ for all $y \in K(f)$ and $m\in \N$. 
    We claim that $f^{-n}(K(f)) \subseteq K(f)$ for all $n\in \N$. 
    Otherwise, there would exist $n\in \N$ and $x\in f^{-n}(K(f))\smallsetminus K(f)$. 
    Then $\bigl\{f^i(x)\bigr\}_{i\in \N}$ converges to an attracting periodic orbit of $f$, which contradicts the fact that $\abs{ (f^m)' (f^n(x)) } \ge C \lambda^m$ for every $m\in \N$. Using this, we obtain $M \cap K_0 = K(f)$. Indeed, clearly, $K(f) \subseteq M\cap K_0$. 
    Conversely, if $y\in K_0 \cap M$, then there exists $n\in \N_0$ such that $f^n(y)=F^n(y) \in K(f)$. 
    Since $f^{-n}(K(f)) \subseteq K(f)$, we conclude that $y\in K(f)$. 
	
	Fix $x\in K_0$. 
    If $x\in \{a_0, \, b_0\}$, by Lemma~\ref{extension}~(i)(iii), we have
    $F(\{a_0, \, b_0\}) \subseteq \{a_0, \, b_0\}$, $\abs{F'(a_0)} > \theta$, and $\abs{F'(b_0) }> \theta$. 
    Thus, for every $m\in \N$, we have $\abs{ (F^m)' (a_0)} > \theta^m$ and $\abs{ (F^m)' (b_0)} > \theta^m$. 
    If $x\in M$, then $x\in K(f)$, and hence $\abs{ (F^m)'(x) } \ge C \lambda^m$ for every $m \in \N$.
	If $x\in (a_0,a)\cup (b,b_0)$, by (ii) there exists $k\in \N$ such that $F^k(x) \in M$ and $F^{k-1}(x) \notin M$. 
    Then, $F^k(x) \in K(f)$. 
    By Lemma~\ref{extension}~(iii), we have $F([a_1,b_1]) \subseteq M$, and hence $F^i(x) \notin [a_1,b_1]$ for all $i\in \{0,\,\dots,\,k-2\}$. 
    So Lemma~\ref{extension}~(iii) implies that for all $i\in \{0,\,\dots,\,k-2\}$, $\Absbig{F'\bigl(F^i(x)\bigr)} > \theta$. 
    Moreover by Lemma~\ref{extension}~(ii), the constant $s \= \min \{ \abs{F'(y)} : y\in [a_0,a] \cup [b,b_0] \} > 0$ is well defined. 
    Thus, $\abs{(F^m)'(x)} \ge s\theta^{m-1}$ if $1 \le m \le k$ and $\abs{(F^m)'(x)} \ge C s \theta^{k-1} \lambda^{m-k}$ if $m\ge k+1$. 
    Consequently, in all cases, we have $\abs{ (F^m)'(x) } \ge \min  \{ 1, \, C , \, s/\theta, \, Cs/\theta \} \cdot \min\{\theta, \, \lambda\}^m $ for every $m\in \N$. 
    Hence, $F$ is expanding on $K_0$.

	It remains to prove that $K_0$ is compact. 
    Since $F^{-i}(K(f))$ is compact for every $i\in \N_0$, it suffices to show that for any sequence $\{x_i\}_{i\in \N_0}$ satisfying $x_i\in F^{-i}(K(f)) \smallsetminus \bigcup_{j=0}^{i-1} F^{-j}(K(f))$ for each $i\in \N_0$, every limit point of $\{x_i\}_{i\in \N_0}$ belongs to $K_0$. 
    Let $\{x_i\}_{i\in \N_0}$ be such a sequence and let $x_*$ be a limit point. 
    We claim that $x_*\in \{a_0,\, b_0\}$. 
    
    Fix $i\in \N$ with $i\ge 2$. Then $F^i(x_i) \in K(f)$, but $F^j(x_i) \notin K(f) $ if $j\in \{0,\,\dots,\, i-1\}$. 
    Since $f^{-n}(K(f)) \subseteq K(f)$ for all $n\in \N$, it follows that $F^j(x_i) \notin M  $ if $j\in \{0,\,\dots,\, i-1\}$. 
    Moreover since Lemma~\ref{extension}~(iii) implies $F([a_1,b_1]) \subseteq M$, we obtain $F^j(x_i) \notin [a_1,b_1]$ if $0\le j <i-1$. 
    If $x_i \in [a_0,a_1)$, using the same inductive argument in (ii), we have $\abs{b_0-a_0} \ge \Absbig{ F^{i-1}([a_0,x_i]) } > \theta^{i-1} \abs{a_0- x_i}$. 
    Similarly, if $x_i \in (b_1,b_0] $, we also have $\abs{b_0-a_0} \ge \Absbig{ F^{i-1}([x_i,b_0]) } > \theta^{i-1} \abs{b_0- x_i}$. 
    Therefore, $\lim_{i\to +\infty} \min \{ \abs{a_0-x_i} , \, \abs{b_0-x_i} \} =0$, and hence the only possible limit points of $\{x_i\}_{i\in \N_0}$ are $a_0$ and $b_0$. 
    This shows that $K_0$ is compact.
		
	\smallskip
	(iv) Fix $G\in W$. 
    Since $G(\{ a_0, \, b_0 \}) \subseteq \{a_0, \, b_0\}$ and $G(M) \subseteq M$,  it follows from (ii) that every $G$-invariant measure is a convex combination of a $G$-invariant measure supported on $\{a_0, \, b_0\}$ and a $G$-invariant measure supported on $M$. 
    More precisely, for every $\mu \in \cM(M_0,G)$, there exist $t_{\mu,1}, \, t_{\mu,2} \in [0,1]$ with $t_{\mu,1}+t_{\mu,2}=1$ and $\mu_M ,\, \mu_e\in \cM(M_0,G)$ with $\supp\mu_M \subseteq M$ and $\supp\mu_e \subseteq \{a_0, \, b_0\}$ such that $\mu=t_{\mu,1} \mu_M +t_{\mu,2} \mu_e$. 
    If $\phi \in C(M_0)$ satisfies $\phi(a_0)<\mpe(G,\phi)$ and $\phi(b_0)< \mpe(G,\phi)$,  and if $\mu\in \Mmax(G,\phi)$, then $\mu_e$ can not be $\phi$-maximizing. 
    Hence $t_{\mu,2}=0$, and therefore $\mu=\mu_M$ is supported on $M$.
\end{proof}

After the above preparatory results, we are now ready to prove Theorem~\ref{1dim}.

\begin{proof}[\bf Proof of Theorem~\ref{1dim}]
	If $M$ is the circle, then Theorem~\ref{1dim} follows by combining Proposition~\ref{endpointsp}, Lemma~\ref{inTPO1dim}, \cite[Theorem~2]{KSS07}, and the fact that hyperbolicity is an open property in $C^r(M,M)$ (see \cite[Theorem~2.4, Chapter~3]{dMvS93}), by an argument analogous to the proof of Theorem~\ref{realqua}.

    We now assume that $M$ is a compact interval and write $M=[a,b]$.  
	Fix $k\in C^r(M,M)$, $\phi\in \cP$, and $\ve>0$. 
    Note that $\{l \in C^r(M,M) : \{l(a) , \, l(b)\} \subseteq (a,b), \, l'(a)\neq 0\text{ and } l'(b)\neq 0\}$ is open and dense in $C^r(M,M)$.
    Moreover, by \cite[Theorem~2]{KSS07}, hyperbolic maps are dense in $C^r(M,M)$. 
	Therefore, there exists a hyperbolic map $f\in C^r(M,M)$ such that $\{f(a),f(b)\}\subseteq (a,b)$, $f'(a)\neq 0$, $f'(b)\neq 0$, and
	\begin{equation}\label{pertmap}
		\Hnorm{C^r}{f-k} < \ve.
	\end{equation}
    By Lemma~\ref{extension} and Lemma~\ref{extensionprop}~(iii), $f$ admits a hyperbolic extension $F \in C_0^r([a_0,b_0],[a_0,b_0])$ with $a_0<a<b<b_0$. 
    Set $M_0\=[a_0,b_0]$, and define $\cP_0\=\Holder{\alpha}(M_0,\R)$ if $\cP=\Holder{\alpha}(M,\R)$, and $\cP_0\=C^1(M_0,\R)$ if $\cP=C^1(M,\R)$. If $\cP = \Holder{\alpha}(M)$ for some $\alpha\in (0,1]$, then for simplicity we write $\Hnorm{\cP}{\cdot} \= \Hnorm{\alpha,M}{\cdot}$ and $\Hnorm{\cP_0}{\cdot} \= \Hnorm{\alpha,M_0}{\cdot}$. Similarly, if $\cP = C^1(M)$, then for simplicity we write  $\Hnorm{\cP}{\cdot} \= \Hnorm{C^1,M}{\cdot}$ and $\Hnorm{\cP_0}{\cdot} \= \Hnorm{C^1,M_0}{\cdot}$.
	
	Let $\Phi\in \cP_0$ satisfy $\Phi|_M=\phi$, and 
    \begin{equation}\label{Phia0b0}
	        	\Phi(a_0) < \mpe(f,\phi)-1 \quad \text{ and }\quad
	\Phi(b_0) < \mpe(f,\phi)-1.
	\end{equation}
Since $F|_M = f$ and $F(M)=f(M)\subseteq M$, Lemma~\ref{extensionprop}~(ii) guarantees that every $F$-invariant measure is supported on $M\cup\{a_0,b_0\}$. We claim that $\mpe(F,\Phi) = \mpe(f,\phi)$.
Indeed, any $(f,\phi)$-maximizing measure is $F$-invariant with $\Phi$-average equal to $\mpe(f,\phi)$
     (since $F|_M=f$ and $\Phi|_M=\phi$), so $\mpe(F,\Phi)\ge\mpe(f,\phi)$.
Conversely, for each $F$-invariant measure $\mu$, writing
     $\mu = t\mu_M + (1-t)\mu_e$ where $\mu_M$ is supported on $M$,
     $\mu_e$ is supported on $\{a_0,b_0\}$, and $t\in[0,1]$,
    we get from (\ref{Phia0b0}) that
\begin{equation*}
    \int \Phi\,\mathrm{d}\mu 
    = t\int\phi\,\mathrm{d}\mu_M + (1-t)\int\Phi\,\mathrm{d}\mu_e
    \le t \cdot \mpe(f,\phi) + (1-t)\cdot\max\{\Phi(a_0),\Phi(b_0)\}
    \le \mpe(f,\phi) .
\end{equation*}
Consequently $\mpe(F,\Phi) \le \mpe(f,\phi)$, so indeed $\mpe(F,\Phi) = \mpe(f,\phi)$.

	By Lemma~\ref{inTPO1dim}, there exists $\Psi\in \Lock(F,\cP_0)$ satisfying $\Hnorm{\cP_0}{\Phi-\Psi} < \min\{\ve,\,1/4\}.$
	Consequently,
	\begin{equation}\label{mpeless}
		\begin{aligned}
			\Psi(a_0) 
			< \Phi(a_0) + 1/4 
			&< \mpe(F,\Phi) - 3/4 
			< \mpe(F,\Psi) - 1/2,\\
			\Psi(b_0) 
			< \Phi(b_0) + 1/4 
			&< \mpe(F,\Phi) - 3/4 
			< \mpe(F,\Psi) - 1/2.
		\end{aligned}
	\end{equation}

	We may assume without loss of generality that $\Psi$ is not a constant. Indeed, since $\Lock(F,\cP_0)$ is open, we may replace $\Psi$ by a sufficiently small perturbation satisfying the same properties.
	By Proposition~\ref{endpointsp}, there exists a neighbourhood $U_1 \subseteq C^r_0(M_0,M_0) \times \cP_0$ of $(F,\Psi)$ such that $U_1 \subseteq \Lock(C^r_0(M_0,M_0), \cP_0)$.
	
    By Proposition~\ref{endpointsp}, $HC^r_0(M_0,M_0)$ is an open subset of $C^r_0(M_0,M_0)$ containing $F$. Because $F$ satisfies conditions (IS) and (ML) in Definition~\ref{Fstablyhyperbolic}, and $U_1$ guarantees that $\Psi\in\Lock(G,\cP_0)$ for all $G$ in a neighbourhood of $F$, Remark~\ref{nonwander}~(ii) ensures that the map $\mpe(\cdot, \Psi)$ is continuous at $F$ (noting that if $\cP_0=C^1(M_0,\R)$ then we use the continuous embedding $C^1 \subseteq \Holder{1}$). 
    Moreover, for each $G\in C^r_0(M_0,M_0)$, the map $\mpe(G,\cdot) \: \cP_0 \to \R$ is $1$-Lipschitz with respect to the $\Hnorm{\infty}{\cdot}$ norm. 
    More precisely, for all $G\in C^r_0(M_0,M_0)$, $\xi_1, \, \xi_2\in \cP_0$, we have $\abs{\mpe(G,\xi_1) - \mpe(G,\xi_2)} \le \Hnorm{\infty}{\xi_1-\xi_2} \le \Hnorm{\cP_0}{\xi_1-\xi_2}$. 
    Hence there exists a neighbourhood $U_2 \subseteq C^r_0(M_0,M_0) \times \cP_0$ of $(F,\Psi)$ such that 
    $\abs{\mpe(G,\Psi_1)-\mpe(F,\Psi)} < 1/4$ for every pair $(G,\Psi_1) \in U_2$. 
    
    Let $W \subseteq C^r_0(M_0,M_0)$ be the neighbourhood of $F$ given by Lemma~\ref{extensionprop}~(ii), and define $U_3 \= W \times B_{\cP_0}(\Psi,1/4)$. Set 
\begin{equation*}
    U_0 \= U_1 \cap U_2 \cap U_3, 
\end{equation*}
	so that $U_0$ is a neighbourhood of $(F,\Psi)$.
    Since $U_0\subseteq U_1$, for each pair $(G,\Psi_1)\in U_0$, the $(G,\Psi_1)$-maximizing measure is unique and periodic. Moreover, by \eqref{mpeless} we obtain
	\begin{equation}\label{mpeless2}
		\begin{aligned}
			\Psi_1(a_0) 
			&< \Psi(a_0)+1/4 
			< \mpe(F,\Psi)-1/4 
			< \mpe(G,\Psi_1),\\
			\Psi_1(b_0) 
			&< \Psi(b_0)+1/4 
			< \mpe(F,\Psi)-1/4 
			< \mpe(G,\Psi_1).
		\end{aligned}
	\end{equation}

	Let $\kappa>0$ be the constant in Lemma~\ref{extensionprop}~(i) applied to $F$ and $M_0$.
	Choose $\delta\in (0,\kappa)$ such that
	\begin{equation}\label{neighdistance}
		B_{C^r_0(M_0,M_0)}(F,\delta)\times B_{\cP_0}(\Psi,\delta)\subseteq U_0.
	\end{equation}
	Set $\psi\=\Psi|_M$. Then
	\begin{equation}\label{pertfunction}
		\Hnorm{\cP}{\phi-\psi} \le \Hnorm{\cP_0}{\Phi-\Psi} < \ve.
	\end{equation}
	
	We now show that $\psi$ belongs to the interior of $\Lock(C^r(M,M),\cP)$.
	Applying Lemma~\ref{exnormbound} to $M\subseteq M_0$, we obtain constants $D_1\ge 1$ and $D_r\ge 1$. 
	Denote
	\begin{equation*}
	    	V \= B_{C^r(M,M)}(f,\delta/D_r)\times B_{\cP}(\psi,\delta/D_1),
	\end{equation*}
     and choose an arbitrary pair $(g,\xi)\in V$.

	By Lemma~\ref{exnormbound}, the map $g-f$ admits an extension $H\in C^r(M_0,\R)$ satisfying $H(a_0)=H(b_0)=0$ and
	\begin{equation}\label{ext1}
		\Hnorm{C^r,M_0}{H}
		\le D_r\Hnorm{C^r,M}{g-f}
		\le \delta
		<\kappa.
	\end{equation}
	Define $G\=F+H$. Then $G|_M=g$, and by \eqref{ext1} together with Lemma~\ref{extensionprop}~(i), we have $G\in C^r_0(M_0,M_0)$.
	On the other hand, if $\cP=\Holder{\alpha}(M,\R)$, then by \cite[Theorem~1.33]{Wea18} there exists an extension $\zeta_1\in \Holder{\alpha}(M_0,\R)$ of $\xi-\psi$ such that
	$\Hnorm{\alpha, M_0}{\zeta_1} \le \Hnorm{\alpha,M}{\xi-\psi}$.
	If $\cP=C^1(M,\R)$, then by Lemma~\ref{exnormbound} there exists an extension $\zeta_2\in C^1(M_0,\R)$ of $\xi-\psi$ such that
	$\Hnorm{C^1,M_0}{\zeta_2} \le D_1\Hnorm{C^1,M}{\xi-\psi}$.
     
   Since $D_1\ge 1$, in either case there exists an extension $\zeta\in \cP_0$ of $\xi-\psi$ satisfying
	\begin{equation}\label{ext2}
		\Hnorm{\cP_0}{\zeta} \le D_1 \Hnorm{\cP}{\xi-\psi} \le \delta.
	\end{equation}
	Define $\Xi\=\Psi+\zeta$. Then $\Xi|_M=\xi$. Moreover, by \eqref{ext1} and \eqref{ext2},
    \begin{equation*}
        	(G,\Xi)\in B_{C^r_0(M_0,M_0)}(F,\delta)\times B_{\cP_0}(\Psi,\delta).
    \end{equation*}
Hence by \eqref{neighdistance}, $(G,\Xi)\in U_0 \subseteq U_1$. Because $U_1 \subseteq \Lock(C^r_0(M_0,M_0), \cP_0)$, the pair $(G,\Xi)$ has the locking property: its maximizing measure is unique and periodic, and $\Mmax(G,\Xi_1) = \Mmax(G,\Xi)$ for every $\Xi_1 \in \cP_0$ sufficiently close to $\Xi$. 
    
    Now $G|_M = g \in C^r(M,M)$, so $G(M) \subseteq M$.
    Since $(G, \Xi) \in U_0 \subseteq U_3$, we also know that $G\in W$.
    Moreover, the bounds in \eqref{mpeless2} establish that $\Xi(a_0) < \mpe(G,\Xi)$ and $\Xi(b_0) < \mpe(G,\Xi)$.
    Combining these, Lemma~\ref{extensionprop}~(iv) gives that the unique $(G,\Xi)$-maximizing measure, $\mu^*$ say, has support contained in $M$. Since $G|_M=g$ and $\Xi|_M=\xi$, it follows that $\mu^*$ is precisely the unique $(g,\xi)$-maximizing measure.
    
    It remains to verify that $(g,\xi)$ has the locking property (rather than merely the PO property). Given any $\xi_1\in\cP$ sufficiently close to $\xi$, by Lemma~\ref{exnormbound} (if $\cP=C^1(M,\R)$) or by \cite[Theorem~1.33]{Wea18} (if $\cP=\Holder{\alpha}(M,\R)$), there exists an extension $\zeta_1\in\cP_0$ of $\xi_1-\xi$ satisfying $\Hnorm{\cP_0}{\zeta_1}\le D_1\Hnorm{\cP}{\xi_1-\xi}$. 
    Setting $\Xi_1 \= \Xi+\zeta_1$, we have $\Xi_1|_M=\xi_1$ and $\Hnorm{\cP_0}{\Xi_1-\Xi} \le D_1\Hnorm{\cP}{\xi_1-\xi}$, so by choosing $\xi_1$ sufficiently close to $\xi$, we can make the extension $\Xi_1$ arbitrarily close to $\Xi$. The locking property of $(G,\Xi)$ then yields $\Mmax(G,\Xi_1) = \Mmax(G,\Xi) = \{\mu^*\}$. 
Since $\mu^*$ is supported on $M$ and $G|_M = g$, the measure $\mu^*$
is $g$-invariant with $\int \xi_1\,\mathrm{d}\mu^* 
= \int \Xi_1\,\mathrm{d}\mu^* = \mpe(G,\Xi_1)$.
Conversely, any $(g,\xi_1)$-maximizing measure $\nu$ is also
$(G,\Xi_1)$-maximizing (since $\nu$ is $G$-invariant and $\int\Xi_1\,\mathrm{d}\nu
= \int\xi_1\,\mathrm{d}\nu = \mpe(g,\xi_1) = \mpe(G,\Xi_1)$),
so $\nu = \mu^*$ by uniqueness.
Hence $\Mmax(g,\xi_1) = \{\mu^*\} = \Mmax(g,\xi)$, and therefore
$(g,\xi)\in\Lock(C^r(M,M),\cP)$.
    
	Since $(g,\xi)\in V$ was arbitrary, we conclude that $(f,\psi)$ belongs to the interior of $\Lock(C^r(M,M),\cP)$.
	
	Finally, by \eqref{pertmap}, \eqref{pertfunction}, the arbitrariness of $\ve>0$, $k\in C^r(M,M)$, and $\phi\in \cP$, we conclude that the interior of $\Lock(C^r(M,M),\cP)$ is an open dense subset of $C^r(M,M)\times \cP$. This completes the proof of Theorem~\ref{1dim}.
\end{proof}

We conclude this section with the proof of Theorem~\ref{realqua01}.

\begin{proof}[\bf Proof of Theorem~\ref{realqua01}]
Since $g_a(x) = ax(1-x)$ is linear in $a$, 
the topology on $\cF$ 
coincides with the subspace topology inherited from $C^1([0,1],[0,1])$. 
Since $g_a(0) = g_a(1) = 0$ for each $a \in [0,4]$, every map in $\cF$
preserves the boundary of $[0,1]$, so
$\cF \subseteq C^1_0([0,1],[0,1])$. 

Let $\cH\cF \subseteq \cF$ denote the subset of hyperbolic maps.
By Graczyk--\'Swi\c{a}tek~\cite{GS97}
and Lyubich~\cite[p.~4]{Ly97},
$\cH\cF $ is open and dense in $\cF$.
Since $\cH\cF$ is open and dense in $\cF$, the product
$\cH\cF \times \cP$ is open and dense in $\cF \times \cP$.
It therefore suffices to show that the joint locking set
$\Lock(\cH\cF, \cP)$ contains an open dense subset
of $\cH\cF \times \cP$.

To show that maps in $\cH\cF$ are $\cH\cF$-stably hyperbolic,
fix $f \in \cH\cF$.
Since $f \in HC^1_0([0,1],[0,1])$ and $\cH\cF \subseteq C^1_0([0,1],[0,1])$,
Proposition~\ref{endpointsp} guarantees that $f$ is
$C^1_0([0,1],[0,1])$-stably hyperbolic in the sense of
Remark~\ref{nonwander}~(i).
Restricting the neighbourhood of $C^1_0([0,1],[0,1])$ furnished by Definition~\ref{Fstablyhyperbolic}
to the smaller family $\cH\cF \subseteq C^1_0([0,1],[0,1])$,
it follows that $f$ is $\cH\cF$-stably hyperbolic.

Now note that Individual TPO holds for maps in $\cH\cF$:
by Lemma~\ref{inTPO1dim}, for every $f \in \cH\cF$,
the locking set $\Lock(f, \mathcal{P})$ is an open dense subset of $\cP$.

We now claim that the set
\begin{equation*}
	  \mathcal{L}_0
  \=
  \{
    (g, \phi) \in \Lock(\cH\cF,\,\cP)
    :
    \phi \text{ is nonconstant}
  \}
\end{equation*}
is an open dense subset
of $\cH\cF \times \mathcal{P}$.

To check openness, suppose $(g, \phi) \in \cL_0$.
Then $g$ is $\cH\cF$-stably hyperbolic by the above, and $\phi$ is nonconstant,
so Theorem~\ref{openness} implies that $(g, \phi)$ lies in the interior of
$\Lock(\cH\cF,\cP)$. Clearly, being nonconstant is an open property in $\cP$, so $(g,\phi)$ lies in the interior of $\cL_0$.
Since $(g, \phi) \in \cL_0$ was arbitrary, it follows that $\cL_0$ is
open in $\cH\cF \times \cP$.

To check density of $\cL_0$ in $\cH\cF \times \cP$,
let $\cW$ be any nonempty open set in $\cH\cF \times \cP$.
Then $\cW$ contains a basic open set of the form $\cV \times \cB$,
where $\cV$ is nonempty and open in $\cH\cF$, and $\cB$ is
nonempty and open in $\cP$.
Choose any $f \in \cV$.
We know that $\Lock(f,\cP)$ is dense in $\cP$,
by the above,
so there exists a nonconstant potential
$\phi \in \cB \cap \Lock(f,\cP)$.
It follows that $(f, \phi) \in \cL_0 \cap \cW$, so indeed $\cL_0$ is
dense in $\cH\cF \times \cP$.

Since $\cL_0$ is open and dense in $\cH\cF \times \cP$,
and $\cH\cF \times \cP$ is open and dense in $\cF \times \cP$,
the set $\cL_0$ is in particular dense in $\cF \times \cP$.
Moreover, every element of $\cL_0$ lies in the interior of
$\Lock(\cF,\cP)$ (since $\cH\cF$ is open in $\cF$,
any $\cH\cF$-locking neighbourhood is also a $\cF$-locking neighbourhood),
therefore $\Lock(\cF,\cP)$ contains an open dense subset of
$\cF \times \cP$; 
in other words $\cF \times \cP$ has the
Joint TPO property, as required.
\end{proof}

\appendix
\section{Proof of some technical lemmas}\label{A}

First, we give a complete proof of Lemma~\ref{A(b)}.

Let $f\: M \to M$ be an Axiom~A diffeomorphism with the splitting $T_x M = E^s(x) \oplus E^u(x)$. A Riemannian metric on $M$ is said to be \emph{adapted to $f$} if there exists a constant $\tau_f\in (0,1)$ such that for the induced norm $\Hseminorm{}{\,\cdot\,}$, and every $x\in \Omega(f)$,  
\begin{equation*}
    \Hseminorm{}{\mathrm{D}f(x)(v)} \le \tau_f \Hseminorm{}{v}   \text{ for all } v\in E^s(x) \quad \text{ and } \quad
    \Hseminormbig{}{\mathrm{D}f^{-1}(x)(w)} \le \tau_f \Hseminorm{}{w}   \text{ for all }w\in E^u(x).
\end{equation*}
By \cite[Theorem~4.4]{Wen16}, there exists a smooth Riemannian metric adapted to $f$. Furthermore, we denote by $\Hnorm{f}{\cdot}$ the \emph{box-adjusted norm} of $\Hseminorm{}{\,\cdot\,}$, defined as
\begin{equation}\label{box}
	\Hnorm{f}{v} \= \max \{ \Hseminorm{}{v_s}, \, \Hseminorm{}{v_u} \} \quad \text{ for } x\in \Omega(f) \text{ and } v\in T_x M,
\end{equation}
where $v = v_s + v_u$ is the unique splitting with $v_s\in E^s(x)$ and $v_u \in E^u(x)$.

\begin{proof}[\bf Proof of Lemma~\ref{A(b)}]
	Fix $r\in \N$ and $f\in \cA^r(M)$. We consider a Riemannian metric $\Hseminorm{a}{\,\cdot\,}$ that is adapted to $f$, and use the distance $d_a$ induced by $\Hseminorm{a}{\,\cdot\,}$. Since $M$ is compact, both $\Hseminorm{}{\,\cdot\,}$ and $\Hseminorm{a}{\,\cdot\,}$ are equivalent, as are $d$ and $d_a$, so it is sufficient to prove condition (RHE) in Definition~\ref{Fstablyhyperbolic} for $d_a$. Throughout this proof, we always consider exponential maps and Lipschitz constants with respect to $\Hseminorm{a}{\,\cdot\,}$ and $d_a$.

	For every $x\in M$, let $\exp_x$ be the exponential map at $x$, and we write $T_x M (\delta) \= \{v\in T_x M : \Hseminorm{a}{v} \le \delta\}$ for every $\delta>0$.
    
	Recall the following well-known result in Riemannian geometry
    (see e.g.~\cite[Chapter~10]{Le18}): there exists a constant $\rho>0$ 
    (the injectivity radius) such that for every $x\in M$, $\exp_x \: T_x M(\rho) \to \overline{B(x,\rho)}$ is a homeomorphism and $d_a(x,\exp_x(v)) = \Hseminorm{a}{v}$ for all $v\in T_x M(\rho)$.

	For every Axiom~A diffeomorphism $g\: M \to M$, we consider the local map on the tangent spaces 
    \begin{equation*}
        F_g(x,v) = \bigl( \exp_{g(x)}^{-1} \circ g \circ \exp_x \bigr) (v) \quad \text{ if } x\in M \text{ and } v \in T_x M(\rho/\LIP(g)),
    \end{equation*}
    where $\exp_{g(x)}^{-1}$ denotes the inverse of $\exp_{g(x)}|_{T_{g(x)}M(\rho)}$.
Since $g$ is a diffeomorphism, $g(M)=M$, so 
$\diam(M) =\diam(g(M)) \le \LIP(g)\diam(M)$, therefore
    $\LIP(g)\ge 1$.
    
    By the definition of $F_g$, if $x, \, y \in M$ and $n\in \N$ satisfy $d_a\bigl(g^i(x),g^i(y)\bigr)< \rho/\LIP(g)$ for all $0\le i\le n-1$, then letting $v\in T_xM(\rho)$ satisfy $\exp_x(v) = y$ gives $\exp_{g^i(x)} \bigl(F^i_g(x,v) \bigr) = g^i(y)$ and $d_a\bigl(g^i(x),g^i(y)\bigr) < \rho$ for all $0\le i\le n$.
    So if $0\le i\le n$ then $F_g^i(x,v)$ is well defined and
	\begin{equation}\label{equaldistance}
		d_a\bigl(g^i(x),g^i(y)\bigr) = \Hseminormbig{a}{F_g^i(x,v)}. 
	\end{equation}

	In the following, we construct a neighbourhood $U$ of $f$.
    First, define 
	\begin{equation}\label{Ar}
		\eta\= \rho/(2\LIP(f)).
	\end{equation}
	There exists a $C^1$ neighbourhood (hence also a $C^r$ neighbourhood) $U_0$ of $f$ such that if $g\in U_0$ then 
	\begin{equation}\label{Alipbound}
		\LIP(g) \le 2\LIP(f).
	\end{equation}
For each $g\in U_0$, the map $F_g$ is well defined on
$TM(\eta)\=\bigcup_{x\in M}T_xM(\eta)$ by
     (\ref{Ar}) and (\ref{Alipbound}).
     For each $x\in M$ and $v\in T_xM(\eta)$, define the nonlinear remainder
\begin{equation*}
    \phi_g(x,v) \= \exp_{g(x)}^{-1}\bigl(g(\exp_x(v))\bigr)
    - \mathrm{D}g(x)(v) \in T_{g(x)}M,
\end{equation*}
     so that $\phi_g$ is also well defined on $TM(\eta)$ for $g\in U_0$.
     Note that $\phi_g(x,0) = 0\in T_{g(x)}M$ for every $x\in M$.

	Next, applying \cite[Lemma~4.8]{Wen16} to $f$ and $\Hseminorm{a}{\,\cdot\,}$, there exist constants $\ve>0$, $C>1$, and a $C^1$ neighbourhood (hence also a $C^r$ neighbourhood) $U_1$ of $f$ such that if $g\in U_1$ and $\Omega(g) \subseteq B(\Omega(f), \ve)$ then $\Hseminorm{a}{\,\cdot\,}$ is also adapted to $g$, with 
	\begin{equation}\label{U11}
		\tau_g < \tau_f + 3^{-1}(1-\tau_f) <1,
	\end{equation}
	and the norm $\Hnorm{g}{\cdot}$ defined by (\ref{box}) on $\Omega(g)$ satisfies
	\begin{equation}\label{U12}
		C^{-1} \Hnorm{g}{v} \le \Hseminorm{a}{v} \le C \Hnorm{g}{v} \quad \text{ for } x\in \Omega(g) \text{ and } v\in T_xM(\eta),
	\end{equation}
    where the constant $C$ is uniform in $g\in U_1$.
	By \cite[Lemma~4.11]{Wen16} applied to $f$, there exists a constant $\delta_0>0$ and a $C^1$ neighbourhood (hence also a $C^r$ neighbourhood)\footnote{We choose $U_2$ to be contained in $U_0$, in order that $\phi_g$ is well-defined for all $g\in U_2$.} $U_2 \subseteq U_0$ of $f$ such that if $g \in U_2$ then 
	\begin{equation}\label{U2}
		\LIP_{v,\delta_0}^{\Hseminorm{a}{\cdot}} (\phi_g) < 3^{-1}C^{-2} (1-\tau_f),
	\end{equation} 
	where $\LIP_{v,\delta_0}^{\Hseminorm{a}{\cdot}} (\phi_g) \= \sup_{x\in M} \bigl(\LIP_{\Hseminorm{a}{\cdot}} \phi_g(x, \cdot)|_{T_xM(\delta_0)} \bigr) $.
	Finally, since $f$ has $C^r$ $\ve$-$\Omega$-stability (\cite[Theorem~5.8]{Wen16}), there exists a $C^r$ neighbourhood $U_3$ of $f$ such that if $g\in U_3$ then $\Omega(g) \subseteq B(\Omega(f), \ve)$.
	Now define $U \= U_0 \cap U_1 \cap U_2 \cap U_3$,
	so that  (\ref{Alipbound}), (\ref{U11}), (\ref{U12}),  (\ref{U2}) hold for all $g\in U$.
	
	Now define constants
	\begin{align}
		K_0 &\= C^2, \label{AK} \\
		\delta &\= \min\{\eta, \, \delta_0\}, \label{Adelta} \\
		\lambda &\= \min \bigl\{  3 (2+\tau_f)^{-1} , \,  3 (1+2\tau_f)^{-1} - 3^{-1} (1-\tau_f) \bigr\}, \label{Alambda}
	\end{align}
	noting that $\lambda>1$ since $0<\tau_f<1$.
    
	Next we shall verify that the inequality (\ref{assumptionb}) from condition (RHE) in Definition~\ref{Fstablyhyperbolic} holds for all $g\in U$. 
    Fix $g\in U$, $n\in\N$, $x\in\Omega(g)$, and $y\in M$ with
$\max_{0\le i\le n}d_a(g^i(x),g^i(y))\le\delta$.
Now $d_a(x,y)\le\delta\le\eta\le\rho$, so 
$v\=\exp_x^{-1}(y)\in T_xM(\rho)$ is well defined.
We verify by induction that $F_g^i(x,v)$ is well defined for
all $0\le i\le n$, and that \eqref{equaldistance} holds at each step.
At $i=0$ both are immediate
from $\exp_x(v)=y$.
Assuming \eqref{equaldistance} holds at
step $i$, we have
\begin{equation*}
    d_a(g^i(x),g^i(y)) = \Hseminormbig{a}{F_g^i(x,v)}
    \le \delta \le \eta = \rho/(2\LIP(f)) \le \rho/\LIP(g),
\end{equation*}
using \eqref{Adelta}, \eqref{Ar}, and \eqref{Alipbound},
so $F_g^{i+1}(x,v)$ is well defined and \eqref{equaldistance} holds
at step $i+1$.
Since 
$\exp_x\:T_xM(\delta)\to B_{(M,d_a)}(x,\delta)$ is a homeomorphism and (\ref{equaldistance})
identifies
$d_a(g^i(x),g^i(y))$ with $\Hseminormbig{a}{F_g^i(x,v)}$,
it now suffices to prove that
	\begin{equation}\label{Aeq}
		\Hseminormbig{a}{ F_g^i (x,v) } \le K_0 \lambda^{-\min \{ i, \,  n-i\}} \bigl( \Hseminorm{a}{v} + \Hseminormbig{a}{F^n_g(x,v)} \bigr).
	\end{equation}	 
    By (\ref{U12}), (\ref{Adelta}), and (\ref{U2}), we have 
    \begin{equation*}
               \LIP_{v,\delta}^{\Hnorm{g}{\cdot}} (\phi_g )\leq C^2   \LIP_{v,\delta}^{\Hseminorm{a}{\cdot}} (\phi_g) \le C^2\LIP_{v,\delta_0}^{\Hseminorm{a}{\cdot}} (\phi_g) <3^{-1}(1-\tau_f),
    \end{equation*} 
    where $\LIP_{v,\delta}^{\Hnorm{g}{\cdot}} (\phi_g )\=\sup_{z\in M} \bigl(\LIP_{\Hnorm{g}{\cdot}} \phi_g(z, \cdot)|_{T_zM(\delta)} \bigr)$. Combining this with (\ref{U11}) and (\ref{Alambda}) gives
	\begin{equation}\label{Albound}
    \begin{aligned}
        \tau_g + \LIP_{v,\delta}^{\Hnorm{g}{\cdot}} (\phi_g )&<\tau_f+3^{-1}(1-\tau_f)+3^{-1}(1-\tau_f)< \lambda^{-1} \quad \text{ and }\\
        \tau_g^{-1} - \LIP_{v,\delta}^{\Hnorm{g}{\cdot}} (\phi_g) &> 3 (1+2\tau_f)^{-1}-3^{-1}(1-\tau_f)>\lambda.    
        \end{aligned}		
	\end{equation}
	For each $z\in \Omega(g)$, let $T_z M = E_g^s(z) \oplus E_g^u(z)$ be the hyperbolic splitting of $g$, and for every $w\in T_zM$, write $w= w_s+w_u$ for the unique decomposition with $w_s \in E_g^s(z)$ and $w_u \in E_g^u(z)$. 
    Since $\phi_g(z,0)=0$, the definition of
    $\LIP_{v,\delta}^{\Hnorm{g}{\cdot}}(\phi_g)$ and (\ref{box}) give,
    for each $z\in\Omega(g)$ and $w\in T_zM(\delta)$,
    \begin{equation}\label{phig_bound}
        \Hseminorm{a}{(\phi_g(z,w))_\sigma}
        \le \Hnorm{g}{\phi_g(z,w)}
        \le \LIP_{v,\delta}^{\Hnorm{g}{\cdot}}(\phi_g)\,\Hnorm{g}{w},
        \quad\text{for each } \sigma\in\{s,u\}.
    \end{equation}
    Then by (\ref{box}), the fact that $\Hseminorm{a}{\,\cdot\,}$ is adapted to $g$, 
    (\ref{phig_bound}),
    and (\ref{Albound}), for each $z\in \Omega(g)$ and each $w\in T_zM(\delta)$, we obtain
	\begin{equation}\label{es}
		\begin{aligned}
			\Hseminorm{a}{( F_g(z,w) )_s } 
            &\le  \Hseminorm{a}{\mathrm{D}g(z)(w_s)} + \Hseminorm{a}{ ( \phi_g(z,w) )_s } 
			 \le \tau_g \Hseminorm{a}{w_s} + \Hseminorm{a}{ ( \phi_g(z,w) )_s } \\
			& \le \tau_g \Hnorm{g}{w} + \LIP_{v,\delta}^{\Hnorm{g}{\cdot}} (\phi_g) \Hnorm{g}{w} 
			 < \lambda^{-1} \Hnorm{g}{w}.
		\end{aligned}
	\end{equation}
Similarly, 
$\mathrm{D}g(z)(w_u)\in E_g^u(g(z))$, so $\Hseminorm{a}{\mathrm{D}g(z)(w_u)} \ge \tau_g^{-1}
\Hseminorm{a}{w_u}$. 
Since $(F_g(z,w))_u = \mathrm{D}g(z)(w_u) + (\phi_g(z,w))_u$,
it follows from 
(\ref{phig_bound}) 
and (\ref{Albound})
that
\begin{equation}\label{eu}
    \Hseminorm{a}{(F_g(z,w))_u}
    \ge \Hseminorm{a}{\mathrm{D}g(z)(w_u)}
     - \Hseminorm{a}{(\phi_g(z,w))_u}
    \ge \tau_g^{-1}\Hseminorm{a}{w_u}
     - \LIP_{v,\delta}^{\Hnorm{g}{\cdot}}(\phi_g)\,\Hnorm{g}{w}.
\end{equation}

Next, we consider separately the following two cases.

\smallskip
\emph{Case~1.} Assume that $\Hseminormbig{a}{\bigl( F_g^i(x,v) \bigr)_u } \le \Hseminormbig{a}{\bigl( F_g^i(x,v) \bigr)_s }$ for all $0\le i \le n$. 
In this case, we conclude from (\ref{es}) that for each 
$1 \le i \le n$,
\begin{equation}\label{Ac1}
	\begin{aligned}
		\Hnormbig{g}{F_g^i(x,v)} 
        &= \Hseminormbig{a}{\bigl( F_g^i(x,v) \bigr)_s } 
        < \lambda^{-1} \Hseminormbig{a}{\bigl( F_g^{i-1}(x,v) \bigr)_s } 
        < \cdots 
        < \lambda^{-i} \Hnorm{g}{v} \\
		&\le\lambda^{-\min\{ i, \, n-i \}} \bigl( \Hnorm{g}{v} + \Hnormbig{g}{F^n_g(x,v)} \bigr),
	\end{aligned}
\end{equation}
and the case $i=0$ holds trivially.

\smallskip
\emph{Case~2.} Assume that 
$\Hseminormbig{a}{\bigl( F_g^k(x,v) \bigr)_u } > \Hseminormbig{a}{\bigl( F_g^k(x,v) \bigr)_s }$
for some $0 \le k\le n$, where $k$ is chosen to be the smallest number with this property. An argument analogous to Case~1 gives that if $i\notin \{k,\,\dots,\,n\}$, then 
\begin{equation}\label{Ac21}
	\Hnormbig{g}{F_g^i(x,v)} < \lambda^{-i} \Hnorm{g}{v}.
\end{equation}
Since $\Hseminormbig{a}{\bigl( F_g^k(x,v) \bigr)_u} = \Hnormbig{g}{ F_g^k(x,v) }$, by (\ref{eu}), (\ref{Albound}), and (\ref{es}), we obtain that if $k\le n-1$ then
\begin{equation}\label{Ac2}
	\Hseminormbig{a}{\bigl( F_g^{k+1}(x,v) \bigr)_u } >
    \bigl(\tau_g^{-1} - \LIP_{v,\delta}^{\Hnorm{g}{\cdot}} (\phi_g) \bigr) \Hnormbig{g}{ F_g^k(x,v)  } > \lambda \Hnormbig{g}{ F_g^k(x,v)  } >\Hseminormbig{a}{\bigl( F_g^{k+1}(x,v) \bigr)_s }.
\end{equation}
Since $\Hseminormbig{a}{(F_g^i(x,v))_u} > \Hseminormbig{a}{(F_g^i(x,v))_s}$
for all $k \le i \le n$ (which follows by iterating (\ref{Ac2})
forward from $k$), the backward iteration of (\ref{Ac2}) from step
$n$ gives, for all $k \le i\le n-1$,
\begin{equation}\label{Ac22}
	 \Hnormbig{g}{ F_g^i(x,v)  } 
     = \Hseminormbig{a}{(F_g^i(x,v))_u}
     \leq \lambda^{i-n}  \Hnormbig{g}{ F_g^n(x,v)  }.
\end{equation}
Clearly, (\ref{Ac22}) is also true when $k=n$.
In this case, we conclude from (\ref{Ac21}) and (\ref{Ac22}) that if $0\le i \le n$ then
\begin{equation}\label{Ac23}
	\Hnormbig{g}{ F_g^i(x,v)  } 
    \leq \lambda^{-\min\{ i, \, n-i \}} \bigl( \Hnorm{g}{v} + \Hnormbig{g}{F^n_g(x,v)} \bigr).
\end{equation}

Finally, combining (\ref{Ac1}), (\ref{Ac23}), (\ref{AK}), (\ref{U12}) gives (\ref{Aeq}).
Together with (\ref{equaldistance}) and the fact that
$\exp_x\:T_xM(\delta)\to B_{d_a}(x,\delta)$ is a homeomorphism,
this establishes the required (\ref{assumptionb})
for the distance $d_a$ with
constant $K_0$.
Since $M$ is compact, the original metric $d$ and the adapted metric $d_a$ are bi-Lipschitz equivalent. 
Adjusting the constants $\delta$ and $K_0$ by the bi-Lipschitz equivalence constants immediately yields \eqref{assumptionb} for the metric $d$, 
thus completing the proof.
\end{proof}

Next we consider Lemma~\ref{disexpanding}, 
which is essentially a one-sided version of Lemma~\ref{A(b)}: since $f|_{J(f)}$ is expanding (with no stable
subbundle), only the expanding half of the hyperbolic splitting
argument is required, and the proof proceeds by a straightforward
adaptation of the argument above.
We therefore give only a sketch, referring the reader to the proof of
Lemma~\ref{A(b)} for the full details of the analogous steps.

\begin{proof}[\bf Proof of Lemma~\ref{disexpanding}]
We first show that the restriction $f|_{J(f)}\: J(f) \to J(f)$ is an open map. Since $\abs{\mathrm{D}f} > \lambda > 0$ on $J(f)$, there exists an open 
neighbourhood $V \subseteq \widehat{\C}$ of the Julia set on which the derivative is nonvanishing; hence, the restriction $f|_V : V \to \widehat{\C}$ 
is an open map. The complete invariance of the Julia set, specifically the fact that $f^{-1}(J(f)) = J(f)$ (see \cite[Lemma~4.3]{Mi06}), 
ensures that for each open set $U \subseteq V$, we have $f(U \cap J(f)) = f(U) \cap J(f)$. Because $f(U)$ is open in 
$\widehat{\C}$, its intersection with $J(f)$ is open in the subspace topology, thus $f|_{J(f)}$ is an open map.
	
	It remains to prove the second assertion.
	Let $\exp$ be the exponential map induced by $\Hseminorm{}{\,\cdot\,}$. 
    Firstly, there exists $\eta_0>0$ such that $\exp_x \: T_x \widehat{\C}(\eta_0) \to \overline{B(x,\eta_0)}$ is a homeomorphism and $d(x,\exp_x(v)) = \Hseminorm{}{v}$ for every $x\in \widehat{\C}$ and every $v\in T_x \widehat{\C}(\eta_0)$. 
    For $g\in \cH\cR^m$, $x\in \widehat{\C}$, and $v\in T_x \widehat{\C}(\eta_0/\LIP(g))$, define $F_g(x,v) \= \bigl( \exp_{g(x)}^{-1} \circ g \circ \exp_x \bigr) (v)$, where $\exp_{g(x)}^{-1}$ denotes the inverse of $\exp_{g(x)}|_{T_{g(x)} \widehat{\C}(\eta_0)}$. 
    Then define $\phi_g \= F_g - Tg$, where $Tg$ is the tangent map of $g$. 
	
	Since $\mathrm{D} g(x)$ depends continuously on $g \in \cH \cR^m$ and $x\in \widehat{\C}$, and the spaces $\widehat{\C}$ and $J(f)$ are compact, there exist constants $\delta>0$, $\theta_0 >1$, and a neighbourhood $N_0$ of $f$ such that $\LIP(g) \le 2\LIP(f)$ and $\Hseminorm{}{\mathrm{D}g(x)} > \theta_0$ for every $g\in N_0$ and every $x\in B(J(f),\delta)$. 
    
    Then, using Lemma~\ref{A(a)hr}, there exists a neighbourhood $N_1$ of $f$ such that if $g\in N_1$ then $J(g) \in B(J(f), \delta)$. 
    
    Next, given $\ve \in (0,\theta_0-1)$, by \cite[Lemma~4.11]{Wen16}, there exists a neighbourhood $N_2$ of $f$ and $\delta_0>0$ such that if $g\in N_2$ then $\LIP_{v,\delta_0}^{\Hseminorm{}{\cdot}}(\phi_g) < \ve$ on $T \widehat{\C}(\delta_0)$.\footnote{The statement of \cite[Lemma~4.11]{Wen16} is given for diffeomorphisms, but an inspection of its proof shows that the argument carries over unchanged to local diffeomorphisms, and in particular to hyperbolic rational maps.} 
    
    Finally, defining $\eta \= \min\{\eta_0/(2\LIP(f)), \, \delta_0\}$, $\theta \= \theta_0 - \ve>1$, and $N \= N_0 \cap N_1 \cap N_2$, if  $g\in N$ and distinct points $x, \,  y\in J(g)$ with $d(x,y)< \eta$, then we have
	\begin{equation*}
		\begin{aligned}
			d(g(x) , g(y)) 
            &= \Absbig{\exp_{g(x)}^{-1}(g(y))} 
            = \Absbig{F_g\bigl(\exp_{x}^{-1}(y)\bigr)} 
            \ge \Absbig{Tg\bigl(\exp_{x}^{-1}(y)\bigr)}- \Absbig{\phi_g\bigl(\exp_{x}^{-1}(y)\bigr)} \\
			&\ge \theta \Absbig{\exp_{x}^{-1}(y)} 
            = \theta d(x,y),
		\end{aligned}
	\end{equation*}  
	and the result follows.
\end{proof}

Next, we consider Lemmas~\ref{extension} and~\ref{exnormbound}, concerning extensions of $C^r$ interval maps. 
Before proving these lemmas, we first describe the construction of such extensions.

Consider $r \in \N$, $a < b$, and $f \in C^r([a,b],\R)$. 
To construct an extension $F \in C^r(\R,\R)$, it suffices to find two functions $g \in C^r((-\infty,a],\R)$ and $h \in C^r([b,+\infty),\R)$ such that for each integer $0 \le i \le r$,
\begin{equation*}
	g^{(i)}(a) = f^{(i)}(a) \text{ and } h^{(i)}(b) = f^{(i)}(b) ,
\end{equation*}
where $f^{(i)}$ denotes the $i$-th derivative of $f$;
then defining $F$ by $F = f$ on $[a,b]$, $F = g$ on $(-\infty,a)$, and $F = h$ on $(b,+\infty)$ yields the desired extension.

In what follows, we restrict our attention to functions $g \in C^r((-\infty,a],\R)$ and $h \in C^r([b,+\infty),\R)$ of the following form: for each $k \in \R$, define
\begin{align}
g_k(x) &= \sum_{i=0}^{r} \frac{f^{(i)}(a)}{i!}(x-a)^i + \frac{(-1)^{r+1}k}{(r+1)!}(x-a)^{r+1}, \label{f-k}\\
h_k(x) &= \sum_{i=0}^{r} \frac{f^{(i)}(b)}{i!}(x-b)^i + \frac{k}{(r+1)!}(x-b)^{r+1}. \label{f+k}
\end{align}
Clearly, for each $k \in \R$ and each $i \in \{0,\,1,\,\dots,\,r\}$, we have $g_k^{(i)}(a) = f^{(i)}(a)$ and $h_k^{(i)}(b) = f^{(i)}(b)$, so $g_k$ and $h_k$ satisfy the required matching conditions. 
Moreover, for every $x \in (-\infty,a]$, the function $g_k(x)$ is linear and monotone increasing with respect to $k$, and for every $x \in [b,+\infty)$, the function $h_k(x)$ is also linear and monotone increasing with respect to $k$.

\begin{proof}[\bf Proof of Lemma~\ref{extension}]
Clearly $\tau \in (0, \min \{ f(a)-a,\, b-f(b)\})$. 

\smallskip
\emph{Claim A.} If $f'(b) > 0$, then there exist $l_+ \in \mathbb{R}$ and $b_1 > b$ such that $h_{l_+}' > 0$ on $[b,+\infty)$, $h_{l_+}' > 2$ on $(b_1,+\infty)$, and $h_{l_+}(b_1) = b - \tau$.
\smallskip

\emph{Proof of Claim A.} Since $h_0'(b) = f'(b) > 0$ and $h_0'$ is continuous, there exists $\theta > 0$ such that $h_0' > 0$ on $[b, b+\theta)$. Hence, for all $k \in [0,+\infty)$ and $x \in [b, b+\theta)$, it follows from (\ref{f+k}) that
\begin{equation}\label{eq:hprime-positive}
h_k'(x) = h_0'(x) + \frac{k}{r!}(x-b)^r > 0.
\end{equation}
For each $w \in (0,\theta)$, define
\begin{equation}\label{eq:kw}
k_w \= \frac{r!}{ w^r }\biggl( \Hnorm{C^r }{f} \sum\limits_{i=0}^{r-1} \frac{w^i}{i!} + 2 \biggr) > 0.
\end{equation}
If $x > b + w$, a direct estimation using (\ref{f+k}) and (\ref{eq:kw}) gives
\begin{align}
h_{k_w}'(x)
&\ge \frac{k_w}{r!}(x-b)^r - \Hnorm{C^r }{f} \sum\limits_{i=0}^{r-1} \frac{(x-b)^i}{i!} 
\ge \frac{ 2 (x-b)^r}{w^r}+
\Hnorm{C^r}{f} \sum\limits_{i=0}^{r-1} \biggl(  \frac{(x-b)^r w^i}{i! w^r} - \frac{(x-b)^i}{i!} \biggr) \notag \\
&> 2 + \Hnorm{C^r}{f} \sum\limits_{i=0}^{r-1} \frac{(x-b)^i}{i!} \biggl(\frac{(x-b)^{r-i}}{ w^{r-i}} - 1 \biggr) 
 \ge 2.\label{eq:derivative-large}
\end{align}
Next, since $\lim_{w \to 0} h_0(b+w) = f(b)$ and, by (\ref{eq:kw}), $\lim_{w \to 0} k_w w^{r+1} = 0$, it follows from (\ref{f+k}) that
\begin{equation}\label{eq:limit}
\lim\limits_{w \to 0} h_{k_w}(b+w) 
= \lim\limits_{w \to 0} \Bigl( h_0(b+w) + \frac{k_w}{(r+1)!} w^{r+1} \Bigr) = f(b)<b-\tau.
\end{equation}
Hence, there exists $w_0 \in (0,\theta)$ such that $h_{k_{w_0}}(b+w_0) < b - \tau$. 
By (\ref{eq:derivative-large}), there exists $b_1 > b + w_0$ such that $h_{k_{w_0}}(b_1) = b - \tau$.
Define $l_+ \= k_{w_0}$. 
Combining (\ref{eq:hprime-positive}), (\ref{eq:derivative-large}), and the fact that $w_0<\theta$, we have $h_{l_+}' > 0$ on $[b,+\infty)$. 
Combining (\ref{eq:derivative-large}) and the fact that $b_1>b+w_0$ gives $h_{l_+}' > 2$ on $(b_1,+\infty)$, 
thus completing the proof of Claim~A.

Arguing similarly, the following three claims can also be established.
	
	\smallskip
	\emph{Claim B.} If $f'(b) <0$, then there exist $l_+ \in \R$ and $b_1>b$ such that $h_{l_+}' < 0$ on $[b,+\infty)$, $h_{l_+}' < -2$ on $(b_1,+\infty)$, and $h_{l_+}(b_1) = a+\tau$.
	
	\smallskip
	\emph{Claim C.} If $f'(a) >0$, then there exist $l_- \in \R$ and $a_1<a$ such that $g_{l_-}' > 0$ on $(-\infty, a]$, $g_{l_-}' > 2$ on $(-\infty, a_1)$, and $g_{l_-}(a_1) = a +\tau$.
	
	\smallskip
	\emph{Claim D.} If $f'(a) <0$, then there exist $l_- \in \R$ and $a_1<a$ such that $g_{l_-}' < 0$ on $(-\infty, a]$, $g_{l_-}' < -2$ on $(-\infty, a_1)$, and $g_{l_-}(a_1) = b-\tau$.
	
\smallskip

Since $f'(a) \neq 0 \neq f'(b)$, combining Claims~A, B, C, and~D gives the following claim:

\smallskip
\emph{Claim E.} There exist $l_+ \in \R$, $l_- \in \R$, $a_1<a$, and $b_1>b$ such that $\Absbig{h_{l_+}'} \neq 0$ on $[b,+\infty)$, $\Absbig{g_{l_-}'} \neq 0$ on $(-\infty,a]$, $\Absbig{h_{l_+}'} >2$ on $(b_1,+\infty)$, $\Absbig{g_{l_-}'} > 2$ on $(-\infty,a_1)$, and $h_{l_+} (b_1), \, g_{l_-} (a_1) \in \{a+\tau,\, b-\tau\}$.

\medskip

Let $l_+ \in \R$, $l_- \in \R$, $b_1>b$, and $a_1 < a$ be as in Claim~E.
Suppose $a_0$ and $b_0$ have been chosen such that 
$a_0 < a_1$ and $b_0 > b_1$. Define a function $F \: [a_0,b_0] \to \R$ by
\begin{equation}\label{F}
F(x) \=
\begin{cases}
g_{l_-}(x) & \text{ if } x \in [a_0,a),\\
f(x) & \text{ if }  x \in M,\\
h_{l_+}(x) &  \text{ if } x \in (b,b_0].
\end{cases}
\end{equation}
By construction, $F \in C^r([a_0,b_0], \R)$ and $F|_M = f$. 
Moreover, since $g_{l_-}' \neq 0$ on $(-\infty,a]$, and $h_{l_+}' \neq 0$ on $[b,+\infty)$, the function $F$ 
is such that $F'(x) \neq 0$
for all $x\in [a_0,a] \cup [b,b_0]$, 
i.e.,~the required property~(ii) of the lemma is satisfied. 

We now show that property~(iii) is also satisfied.
Since $\tau  \in (0, \min \{ f(a)-a, \, b-f(b)\} )$, then $f(b) \in (a+\tau, b-\tau)$, and by 
property~(ii), the function $h_{l_+}$ is monotone. 
Since $h_{l_+}(b_1)$ is equal to either $a+\tau$ or $b-\tau$, by Claim~E, it follows that
\begin{equation*}
	h_{l_+}((b_1,b_0])\cap[a+\tau,b-\tau]=\emptyset \quad\text{ and }\quad h_{l_+}([b,b_1])\subseteq[a+\tau,b-\tau].
\end{equation*}
Moreover, Claim~E also gives that $\Absbig{h_{l_+}'}> 2$ on $(b_1,b_0]$.
Arguing similarly, we also see that
\begin{equation*}
	g_{l_-}([a_0,a_1))\cap[a+\tau,b-\tau]=\emptyset, \quad 
    g_{l_-}([a_1,a])\subseteq[a+\tau,b-\tau],
\end{equation*}
and $\abs{g_{l_-}'} > 2$ on $[a_0,a_1)$, so $F$ does indeed satisfy property~(iii).	

\smallskip

Note that property~(i) 
(i.e.,~ $\{F(a_0),\, F(b_0) \} \subseteq \{a_0, \, b_0\}$ and $F|_{M} = f$) is satisfied
if additionally
\begin{equation}\label{a0b0_additional}
	\{g_{l_-} (a_0),\, h_{l_+} (b_0)\}\subseteq \{a_0, \, b_0 \}.
\end{equation}
Moreover, property~(iii) then implies that $F([a_1,b_1]) \subseteq M$, and $F$ is monotone on $[a_0,a_1]$ and $[b_1,b_0]$.
If $\max\{ F(x) : x\in[a_0,b_0] \}> b_0$, then the maximum is attained at $a_0$ or $b_0$, contradicting the above assumption. Similarly, $\min\{ F(x) : x\in[a_0,b_0]\} \ge a_0$. 
Hence,
$F([a_0,b_0])\subseteq[a_0,b_0]$,
and thus $F \in C^r_0([a_0,b_0],[a_0,b_0])$.
So it remains to show that $a_0 < a_1$ and $b_0 > b_1$ can be chosen such that (\ref{a0b0_additional}) holds.

For this we distinguish four cases according to the
possible signs of $f'(a)$ and $f'(b)$.
	
	\smallskip
	\emph{Case~1.} Assume that $f'(a) >0$ and $f'(b) >0$. 	
Using Claim~E and the assumption $f'(b) >0$, we have $h_{l_+}(b_1) < b_1$ and $h_{l_+}' > 2$ on $(b_1,+\infty)$. 
Hence $x \mapsto h_{l_+}(x) - x$ is strictly increasing with derivative greater than $1$, and is negative at $x=b_1$. 
Hence, there exists $b_0 > b_1$ such that $h_{l_+}(b_0) = b_0$.
Similarly, there exists $a_0 < a_1$ such that $g_{l_-}(a_0) = a_0$.

	\smallskip

	\emph{Case~2.} Assume that $f'(a) <0$ and $f'(b) >0$. 	
	Using Claim~E and the assumption that $f'(b)>0$, the same argument as in Case 1 yields $b_0 > b_1$ such that $h_{l_+}(b_0)=b_0$.
    Moreover, using Claim~E and the assumption that $f'(a)<0$, the function $g_{l_-}$ is strictly decreasing on $(-\infty,a]$ with $g_{l_-}'<-2$ on $(-\infty,a_1)$; 
    in particular, $g_{l_-}(x)\to+\infty$ as $x\to-\infty$. 
    Since $g_{l_-}(a_1)  < b < b_0$, the intermediate value theorem yields 
$a_0 < a_1$ such that $g_{l_-}(a_0) = b_0$.

	\smallskip

	\emph{Case~3.} Assume that $f'(a) >0$ and $f'(b) <0$. 	
	This case is analogous to Case 2 and is omitted.
	
	\smallskip
	\emph{Case~4.} Assume that $f'(a) <0$ and $f'(b) <0$. Using Claim~E and our assumptions $f'(a)<0$ and $f'(b)<0$, we have $g_{l_-}' <0$ on $(-\infty,a]$ and $h_{l_+}'<0$ on $[b,+\infty)$. Then we seek $a_0 < a_1$ and $b_0 > b_1$ such that
\begin{equation*}
	g_{l_-}(a_0) = b_0 \text{ and } h_{l_+}(b_0) = a_0.
\end{equation*}
Set $p \= b_1 - a_1$, $a_2 \= a_1 - p$, and $b_2 \= b_1 + p$. By the derivative bounds, one checks that
\begin{equation*}
	g_{l_-}(a_2) >b_2 \text{ and } h_{l_+}(b_2) < a_2.
\end{equation*}
More precisely, since $g_{l_-}(a_1) > a$ and $g_{l_-}' < -2$ on $[a_2,a_1]$, we have $g_{l_-} (a_2) > 2p + a > p+b_1 = b_2$. 
Similarly, $h_{l_+}(b_2) < a_2$.

Define
$H(x,y)\=\abs{g_{l_-}(x)-y}+\abs{h_{l_+}(y)-x}$
on $[a_2,a_1]\times[b_1,b_2]$. 

Let $\mathrm{D}_x^- H$ and $\mathrm{D}_x^+ H$ respectively denote the left and right partial derivatives of $H$ with respect to the first variable, 
and let $\mathrm{D}_y^- H$ and $\mathrm{D}_y^+ H$ 
respectively denote the left and right partial derivatives of $H$ with respect to the second variable. 
Let $a_* \in [a_2,a_1]$ and $b_* \in [b_1,b_2]$ be such that $H(a_*,b_*) = \min\{ H(x,y) : x\in [a_2,a_1], \, y\in [b_1,b_2] \}$. 
If $a_* = a_1$, the minimality of $H(a_*,b_*)$ implies $\mathrm{D}_x^- H(a_*,b_*) \le 0$. 
However, since $g_{l_-}(a_1) <b < b_*$, we have 
$\mathrm{D}_x^- H(a_*,b_*) \ge -g_{l_-}'(a_1) -1 >1> 0$,
which leads to a contradiction. 
If $a_*=a_2$, the minimality of $H(a_*,b_*)$ implies $\mathrm{D}_x^+ H(a_*,b_*) \ge 0$. 
 However, since $g_{l_-}(a_2) > b_2\ge b_*$, we have 
 $\mathrm{D}_x^+ H(a_*,b_*) \le 
    g_{l_-}'(a_2) +1 
    < -1 < 0$, 
which leads to a contradiction. 
Thus, $a_* \in (a_2,a_1)$. 
Similarly, we have $b_* \in (b_1,b_2)$. 
Now assume that $g_{l_-}(a_*) \neq b_*$, by the minimality of $H(a_*,b_*)$, 
we have $\mathrm{D}^+_x H(a_*,b_*) \ge 0$ and $\mathrm{D}_x^- H (a_*,b_*) \le 0$. 
However, if $g_{l_-}(a_*) < b_*$, 
then $\mathrm{D}_x^+ H(a_*,b_*) \ge - g_{l_-}'(a_*) -1>1>0$ and $\mathrm{D}_x^- H(a_*,b_*) \ge - g_{l_-}'(a_*) -1>1>0$, which leads to a contradiction. 
If $g_{l_-}(a_*) > b_*$, 
then $\mathrm{D}_x^+ H(a_*,b_*) \le g_{l_-}'(a_*) +1 < -1 <0$ and $\mathrm{D}_x^- H(a_*,b_*) \le g_{l_-}'(a_*) +1 < -1 <0$, which leads to a contradiction. 
Thus $g_{l_-}(a_*) =b_*$. Similarly, $h_{l_+}(b_*) =a_*$. This completes the proof.
\end{proof}

\begin{proof}[\bf Proof of Lemma~\ref{exnormbound}]
We first control the left extension. 
For $k \in \R$, recall the candidate extension $g_{k}$ on $[a_0,a]$ defined (cf.~(\ref{f-k})) by
$$g_k(x) = \sum_{i=0}^{r} \frac{f^{(i)}(a)}{i!}(x-a)^i + \frac{(-1)^{r+1}k}{(r+1)!}(x-a)^{r+1}.$$ 
Define the parameter boundary
\begin{equation}\label{eq:kminus}
k_- \= \Hnorm{C^r,M}{f} (r+1)!(r+1)\bigl(1 + (a-a_0)^{-r-1}\bigr).
\end{equation}
Note that if $y>0$ then $\sum_{i=0}^{r} y^i \le (r+1)\max\{1,\, y^{r+1}\} \le (r+1)\bigl(1+y^{r+1}\bigr)$. Setting $y \= a-a_0$, we get from a direct estimate that
\begin{equation*} 
	\begin{aligned} 
		g_{k_-}(a_0) &\ge \frac{k_-}{(r+1)!} (a-a_0)^{r+1} - \sum_{i=0}^{r} \Absbig{f^{(i)} (a)} (a-a_0)^i \\ 
		&\ge \Hnorm{C^r,M}{f}(r+1) \bigl(1+ (a-a_0)^{r+1} \bigr) - \Hnorm{C^r,M}{f}\sum_{i=0}^{r}(a-a_0)^i  
		\ge 0. 
	\end{aligned} 
\end{equation*}
By a symmetric lower bound, $g_{-k_-}(a_0) \le 0$.
Since $g_k(a_0)$ is continuous (indeed affine) with respect to $k$, the intermediate value theorem guarantees the existence of some $k_1 \in [-k_-,k_-]$ such that $g_{k_1}(a_0)=0$.
Next we estimate the $C^r$-norm of $g_{k_1}$ on $[a_0,a]$. A straightforward computation yields, for each $0\le i\le r$,
\begin{equation*}
	g_{k_1}^{(i)}(x) = \sum_{j=0}^{r-i} \frac{f^{(i+j)}(a)}{j!} (x-a)^j + \frac{(-1)^{r+1}k_1}{(r+1-i)!} (x-a)^{r+1-i}.
\end{equation*}
Since $\abs{k_1}\le k_-$ and
$k_-$ is proportional to $\Hnorm{C^r,M}{f}$, with a proportionality factor depending only on $r$ and $a-a_0$ 
(cf.~(\ref{eq:kminus})), 
each derivative $\Absbig{g_{k_1}^{(i)}(x)}$ is bounded by a constant multiple of $\Hnorm{C^r,M}{f}$. Maximizing these over $0 \le i \le r$
gives a uniform bound
	\begin{equation*}
		\Hnorm{C^r,[a_0,a]}{g_{k_1}} \le C_- \Hnorm{C^r,M}{f},
	\end{equation*}
for some constant $C_- \ge 1$ depending only on $r$ and $a-a_0$.

An entirely analogous argument on the right side yields an extension
$h_{k_2}$ on $[b,b_0]$ such that $h_{k_2}(b_0) = 0$ and $\Hnorm{C^r,[b,b_0]}{h_{k_2}} \le C_+ \Hnorm{C^r,M}{f}$ for some constant $C_+ \ge 1$ depending only on $r$ and $b_0-b$.
Finally, define the glued function $F$ 
on $M_0 = [a_0, b_0]$ by 
\begin{equation*}
	F(x) = \begin{cases}
		g_{k_1}(x) &\text{ if } x \in [a_0,a), \\
		f(x) &\text{ if } x \in M, \\
		h_{k_2}(x) &\text{ if } x \in (b,b_0],
	\end{cases}
\end{equation*}
where the matching of the first $r$ derivatives, at both points $a$ and $b$, ensures that $F$ is $C^r$ on $M_0$.
Note that $F(a_0)=g_{k_1}(a_0)=0$
and $F(b_0)=h_{k_2}(b_0)=0$,
and moreover
\begin{equation*}
	\Hnorm{C^r,M_0}{F}\le \Hnorm{C^r,[a_0,a]}{g_{k_1}} + \Hnorm{C^r,M}{f} + \Hnorm{C^r,[b,b_0]}{h_{k_2}} \le D_r \Hnorm{C^r,M}{f},
\end{equation*}
where $D_r \= 1 + C_- + C_+$ depends only on $r$ and the lengths of the intervals.
This completes the proof.
\end{proof}

\end{document}